\documentclass{isypaper}     

\begin{document}

 \title{J\o rgensen's Inequality and Purely Loxodromic 2--Generator Free Kleinian Groups}

\author{\.Ilker S. Y\"{u}ce}

\newcommand{\noi}{\noindent}
\newcommand{\tb}{\textbf}
\newcommand{\tr}{\textrm}
\newcommand{\pr}{\partial}
\newcommand{\fr}{\frac}
\newcommand{\ra}{\rightarrow}
\newcommand{\Ra}{\Rightarrow}
\newcommand{\lag}{\langle}
\newcommand{\rag}{\rangle}
\newcommand{\vs}{\vspace{.05in}}
\newcommand{\vvs}{\vspace{.05in}}
\newcommand{\ep}{\epsilon}
\newcommand{\de}{\delta}
\newcommand{\dis}{\displaystyle}
\newcommand{\co}{:}
\newcommand{\tnr}{\textnormal}
\newcommand{\C}{\mathbb{C}}
\newcommand{\xs}{\textbf{x}^*}
\newcommand{\R}{\mathbb{R}}
\newcommand{\ys}{\textbf{y}^*}

\numberwithin{equation}{section}
\newtheorem*{Definition}{Definition}
\newtheorem*{Log 3 Theorem}{Log 3 Theorem}
\newtheorem*{thmyuce}{Theorem 2.1}
\newtheorem*{thmyuce2}{Theorem 2}
\newtheorem*{thmyuce3}{Theorem 3}
\newtheorem*{thm}{Theorem}
\newtheorem{Conjecture}{Conjecture}
\newcommand{\hyp}{\mathbb{H}^3}
\newcommand{\alphas}{\alpha_*}
\newcommand{\dseven}{\Delta^7}
\newcommand{\dgammaiprm}{\textrm{dist}(z_0,\ \gamma_i'\cdot z_0)}
\newcommand{\dxiprm}{\textrm{dist}(z_0,\ \xi'\cdot z_0)}
\newcommand{\detaprm}{\textrm{dist}(z_0,\ \eta'\cdot z_0)}
\newcommand{\dxietaprm}{\textrm{dist}(z_0,\ \xi'\eta'\cdot z_0)}
\newcommand{\dxi}{\textrm{dist}(z_0,\ \xi\cdot z_0)}
\newcommand{\deta}{\textrm{dist}(z_0,\ \eta\cdot z_0)}
\newcommand{\dxieta}{\textrm{dist}(z_0,\ \xi\eta\cdot z_0)}
\newcommand{\dxiinveta}{\textrm{dist}(z_0,\ \xi^{-1}\eta\cdot z_0)}
\newcommand{\detaxiinv}{\textrm{dist}(z_0,\ \eta\xi^{-1}\cdot z_0)}
\newcommand{\dxiinvetainv}{\textrm{dist}(z_0,\ \xi^{-1}\eta^{-1}\cdot z_0)}
\newcommand{\dgamma}{\textrm{dist}(z_0,\ \gamma\cdot z_0)}
\newcommand{\dgammaprm}{\textrm{dist}(z_0,\ \gamma'\cdot z_0)}
\newcommand{\dgammainv}{\textrm{dist}(z_0,\ \gamma^{-1}\cdot z_0)}
\newcommand{\dxietainv}{\textrm{dist}(\xi\cdot z_0,\ \eta^{-1}\cdot z_0)}
\newcommand{\detainvxiinv}{\textrm{dist}(\eta^{-1}\cdot z_0,\ \xi^{-1}\cdot z_0)}
\newcommand{\dxiinvxi}{\textrm{dist}(\xi^{-1}z_0,\ \xi\cdot z_0)}
\newcommand{\detainveta}{\textrm{dist}(\eta^{-1}\cdot z_0,\ \eta\cdot z_0)}
\newcommand{\dxisqr}{\textrm{dist}(z_0,\ \xi^{2}\cdot z_0)}
\newcommand{\detasqr}{\textrm{dist}(z_0,\ \eta^{2}\cdot z_0)}
\newcommand{\ddxiinveta}{\textrm{dist}(\xi^{-1}\cdot z_0,\ \eta\cdot z_0)}
\newcommand{\ddxieta}{\textrm{dist}(\xi\cdot z_0,\ \eta\cdot z_0)}
\newcommand{\ddetainvxiinv}{\textrm{dist}(\eta^{-1}\cdot z_0,\ \xi^{-1}\cdot z_0)}
\newcommand{\jrg}{\textnormal{\o}}
\newcommand{\dxietaxiinv}{\textrm{dist}(z_0,\ \xi\eta\xi^{-1}\cdot z_0)}
\newcommand{\detaxietainv}{\textrm{dist}(z_0,\ \eta\xi\eta^{-1}\cdot z_0)}
\newcommand{\dgammaa}{\textrm{dist}(z_1,\ \gamma\cdot z_1)}                
\newcommand{\dgammaz}{\textrm{dist}(z,\ \gamma\cdot z)} 
\newcommand{\dgammatz}{\textrm{dist}(z,\ \widehat{\gamma}\cdot z)}
\newcommand{\dgammatzZ}{\textrm{dist}(z_0,\ \widehat{\gamma}\cdot z_0)}
\newcommand{\dgammaj}{\textrm{dist}(z_0,\ \gamma_j\cdot z_0)}              
\newcommand{\dzetak}{\textrm{dist}(z_0,\ \zeta_k\cdot z_0)}                
\newcommand{\ddeltar}{\textrm{dist}(z_0,\ \delta_r\cdot z_0)}              
\newcommand{\dxizone}{\textrm{dist}(z_1,\ \xi\cdot z_1)}                   %
\newcommand{\dpsi}{\textrm{dist}(z,\ \psi\cdot z)}                         
\newcommand{\dpsiz}{\textrm{dist}(z_0,\ \psi\cdot z_0)}                      
\newcommand{\dpsizinv}{\textrm{dist}(z_0,\ \psi_0^{-1}\cdot z_0)}              
\newcommand{\dphi}{\textrm{dist}(z,\ \phi\cdot z)}                         
\newcommand{\dxiinv}{\textrm{dist}(z_0,\ \xi^{-1}\cdot z_0)}               
\newcommand{\detazone}{\textrm{dist}(z_1,\ \eta\cdot z_1)}                 %
\newcommand{\detainv}{\textrm{dist}(z_0,\ \eta^{-1}\cdot z_0)}             
\newcommand{\dpsiphi}{\textrm{dist}(z_0,\ \psi\phi\cdot z)}                
\newcommand{\detaxi}{\textrm{dist}(z_0,\ \eta^{-1}\xi^{-1}\cdot z_0)}      
\newcommand{\detantwo}{\textrm{dist}(z_0,\ \eta^{-2}\cdot z_0)}            
\newcommand{\detanonexi}{\textrm{dist}(z_0,\ \eta^{-1}\xi\cdot z)}         
\newcommand{\dxitwo}{\textrm{dist}(z_0,\ \xi^{2}\cdot z_0)}                
\newcommand{\dxintwo}{\textrm{dist}(z_0,\ \xi^{-2}\cdot z_0)}              
\newcommand{\dxietanone}{\textrm{dist}(z_0,\ \xi\eta^{-1}\cdot z_0)}       
\newcommand{\detaxinone}{\textrm{dist}(z_0,\ \eta\xi^{-1}\cdot z_0)}       
\newcommand{\disetaxi}{\textrm{dist}(z_0,\ \eta\xi\cdot z_0)}              
\newcommand{\disetatwo}{\textrm{dist}(z_0,\ \eta^2\cdot z_0)}              
\newcommand{\disxiinvetainv}{\textrm{dist}(z_0,\ \xi^{-1}\eta^{-1}\cdot z_0)}  
\newcommand{\disxiinveta}{\textrm{dist}(z_0,\ \xi^{-1}\eta\cdot z_0)}          
\newcommand{\disbetaij}{\textrm{dist}(z_0,\ \beta_{i,j}\cdot z_0)}             
\newcommand{\disgammaij}{\textrm{dist}(z_0,\ \gamma_{i,j}\cdot z_0)}           
\newcommand{\diszetaij}{\textrm{dist}(z_0,\ \zeta_{i,k}\cdot z_0)}             
\newcommand{\disupsilonij}{\textrm{dist}(z_0,\ \upsilon_{i,j}\cdot z_0)}       
\newcommand{\hype}{\overline{\mathbb{H}}^3}

\maketitle

\begin{abstract}    
Let $\xi$ and $\eta$ be two non--commuting isometries of the hyperbolic $3$--space $\mathbb{H}^3$ so that  $\Gamma=\langle\xi,\eta\rangle$ is a purely loxodromic free Kleinian group. For $\gamma\in\Gamma$ and $z\in\hyp$, let $d_{\gamma}z$ denote the distance between $z$ and $\gamma\cdot z$. Let $z_1$ and $z_2$ be the mid-points of the shortest geodesic segments connecting the axes of $\xi$, $\eta\xi\eta^{-1}$ and $\eta^{-1}\xi\eta$, respectively. In this manuscript it is proved that if $d_{\gamma}z_2<1.6068...$ for every $\gamma\in\{\eta, \xi^{-1}\eta\xi, \xi\eta\xi^{-1}\}$ and $d_{\eta\xi\eta^{-1}}z_2\leq d_{\eta\xi\eta^{-1}}z_1$, then
$$
|\textnormal{trace}^2(\xi)-4|+|\textnormal{trace}(\xi\eta\xi^{-1}\eta^{-1})-2|\geq 2\sinh^2\left(\tfrac{1}{4}\log\alpha\right) = 1.5937....
$$ 
Above $\alpha=24.8692...$ is the unique real root of the polynomial 
$21 x^4 - 496 x^3 - 654 x^2 + 24 x + 81$ that is greater than $9$. Also generalisations of this inequality for finitely generated purely loxodromic free Kleinian groups are conjectured.
\end{abstract}

\maketitle


\section{Introduction}

Let $\xi$ and $\eta$ be two non-commuting isometries of $\hyp$ represented by $A$ and $B$ in PSL($2,\mathbb{C}$), respectively. Since $A$ and $B$  are determined up to a factor of $-1$, the product $ABA^{-1}B^{-1}$ is uniquely determined by these two isometries. Therefore, in the rest of this text we will write $\tnr{trace}^2(\xi)$ and $\tnr{trace}(\xi\eta\xi^{-1}\eta^{-1})$ in the places of  $\tnr{trace}^2(A)$ and $\tnr{trace}(ABA^{-1}B^{-1})$, respectively, without any confusion.  

In his well-known result, called the J{\o}rgensen's inequality, J{\o}rgensen \cite{JT} proved the statement below:

\begin{theorem*} 
If $\langle\xi,\eta\rangle$ is a Kleinian group then, the lower bound being the best possible, 
\begin{equation}\label{eqn:1:1}
|\textnormal{trace}^2(\xi)-4|+|\textnormal{trace}(\xi\eta\xi^{-1}\eta^{-1})-2|\geq 1.
\end{equation}
\end{theorem*}

\noindent An immediate application of this result on hyperbolic displacements was given by Beardon in \cite[Theorem 5.4.5]{B}. The work in this paper is mainly motivated by this theorem:

\begin{theorem*}
Suppose that $\langle \xi,\eta\rangle$ is a Kleinian group. If $\xi$ is elliptic or strictly loxodromic so that  $|\tnr{trace}^2(\xi)-4|<\tfrac{1}{4}$, then for any $z$ in $\hyp$ we have $$\max\{\sinh(\tfrac{1}{2}d_{\xi}z),\ \sinh(\tfrac{1}{2}d_{\eta\xi\eta^{-1}}z)\}\geq\tfrac{1}{4}.$$
\end{theorem*}
 
Due to an extension introduced in  \cite{Y1} and \cite{Y2} by the author,  the machinery developed by Culler and Shalen in \cite{CSParadox} allows one to compute a lower bound for the maximum of hyperbolic displacements under any finite set of isometries in  a purely loxodromic finitely generated free Kleinian group $\Gamma$. In particular in the case of 2-generator, eg if $\Gamma=\langle\xi,\eta\rangle$, it is possible to compute a lower bound for the maximum of the hyperbolic displacements given by the set $\Gamma_{\jrg}$ of isometries
\begin{equation}\label{conj}
\Phi_1\cup\{\xi\eta\xi^{-1},\xi^{-1}\eta\xi,\eta\xi\eta^{-1},\eta^{-1}\xi\eta,\xi\eta^{-1}\xi^{-1},\xi^{-1}\eta^{-1}\xi,\eta\xi^{-1}\eta^{-1},\eta^{-1}\xi^{-1}\eta\},
\end{equation}
where $\Phi_1=  \{\xi,\eta,\eta^{-1},\xi^{-1}\}$. 
Explicitly we shall first establish the following statement: 

\medskip
\noindent {\bf \fullref{thm:4:1}}\quad
{\sl 
Suppose that $\Gamma=\langle\xi,\eta\rangle$ is a purely loxodromic free Kleinian group. 
Then, for $\Gamma_{\jrg}$ in (\ref{conj}), we have $\max\nolimits_{\gamma\in\Gamma_{\jrg}}\left\{d_{\gamma}z\right\}\geq
1.6068...$ for any $z\in\hyp$.
}
\medskip

\noindent Let $z_1$ and $z_2$ be the mid-points of the shortest geodesic segments connecting the axes of $\xi$, $\eta\xi\eta^{-1}$ and $\eta^{-1}\xi\eta$, respectively.
Then we will show that the theorem above implies that

\medskip
\noindent {\bf \fullref{thm:4:2}}\quad
{\sl If $d_{\gamma}z_2<1.6068...$ for $\gamma\in\Phi_2=\{\eta,  \xi^{-1}\eta\xi, \xi\eta\xi^{-1}\}$ and $d_{\eta\xi\eta^{-1}}z_2\leq d_{\eta\xi\eta^{-1}}z_1$, then we have $|\textnormal{trace}^2(\xi)-4|+|\textnormal{trace}(\xi\eta\xi^{-1}\eta^{-1})-2| \geq 1.5937....$}
\medskip

The proof of \fullref{thm:4:2} will involve the computations given in the proof of Theorem 5.4.5 in \cite{B} which uses  the geometry of the action of loxodromic isometries together with some elementary inequalities involving hyperbolic trigonometric functions. But most of the work in this paper will be required to prove \fullref{thm:4:1}. We start by reviewing briefly the necessary ingredients used in the proof of \fullref{thm:4:1} including a summary of the Culler--Shalen machinery introduced in \cite{CSParadox}. 

Let us define $\Psi$ as the set of isometries in $\Gamma=\langle\xi,\eta\rangle$ whose elements are listed and enumerated below:
\begin{equation}\label{list:1:4}
\scalebox{1}{$
\begin{array}{llllllll}
\xi\eta^{-1}\xi^{-1} & \mapsto  1,  & \eta^{-1}\xi^{-1}\eta^{-1} & \mapsto  8,   & \eta\xi^{-1}\eta^{-1}                         & \mapsto  15,     & \xi^{-1}\eta^{-1}\xi^{-1}  &\mapsto 22,\\
\xi\eta^{-1}\xi         & \mapsto 2,  & \eta^{-1}\xi^{-1}\eta         & \mapsto  9,   & \eta\xi^{-1}\eta & \mapsto  16,     & \xi^{-1}\eta^{-1}\xi &\mapsto  23,\\
\xi\eta^{-2}             & \mapsto  3,  & \eta^{-1}\xi^{-2}               & \mapsto 10,  & \eta\xi^{-2}           & \mapsto  17,     & \xi^{-1}\eta^{-2}  &\mapsto  24,\\
\xi\eta^{2}              & \mapsto  4,   & \eta^{-1}\xi^{2}                & \mapsto 11,   & \eta\xi^2                   & \mapsto  18,     & \xi^{-1}\eta^2  &\mapsto  25,\\
\xi\eta\xi^{-1}        &\mapsto  5,     & \eta^{-1}\xi\eta^{-1}        & \mapsto 12,   & \eta\xi\eta^{-1}                  & \mapsto  19,     & \xi^{-1}\eta\xi^{-1}  &\mapsto  26,\\
\xi\eta\xi                &\mapsto  6,     & \eta^{-1}\xi\eta                & \mapsto 13,   & \eta\xi\eta         & \mapsto  20,     & \xi^{-1}\eta\xi &\mapsto 27,\\
\xi^2                      & \mapsto  7,    & \eta^{-2}                         & \mapsto  14,   & \eta^2                  & \mapsto  21,     & \xi^{-2}           &  \mapsto  28.
\end{array}$}
\end{equation}
We shall denote this enumeration by $p\co\Psi\to\{1,\dots,28\}$.
Let $\Psi_r=\Phi_1=\{\xi,\eta^{-1},\eta,\xi^{-1}\}$. Since it is assumed that $\Gamma=\langle\xi,\eta\rangle$ is free, it can be decomposed as follows:
\begin{equation}\label{dJ}
\Gamma=\{1\}\cup\Psi_r\cup\bigcup_{\psi\in\Psi}J_{\psi},
\end{equation}
where $J_{\psi}$ denotes the set of all words starting with the word $\psi\in\Psi$. We will name this decomposition $\Gamma_{\mathcal{D}}$. Let us define $J_{\Phi}=\cup_{\psi\in\Phi}J_{\psi}$ for $\Phi\subseteq\Psi$. A group--theoretical relation for a given decomposition of $\Gamma=\langle\xi,\eta\rangle$ is a relation among the sets $J_{\psi}$. As an example,  
\begin{equation}\label{example}
\xi\eta\xi^{-1}J_{\xi\eta^{-1}\xi^{-1}}=\Gamma-\left(\{\xi\}
\cup J_{\{\xi^2,\xi\eta^{-1}\xi^{-1},\xi\eta^{-1}\xi,\xi\eta^{-2},\xi\eta^2,\xi\eta\xi^{-1},\xi\eta\xi\}}\right)
\end{equation}
is a group--theoretical relation of the decomposition in (\ref{dJ}) which indicates that when multiplied on the left by $\xi\eta\xi^{-1}$ the set of words in $\Gamma=\langle\xi,\eta\rangle$ starting with $\xi\eta^{-1}\xi^{-1}$ translates into the set of words starting with the words whose initial letters are different than $\xi$. Isometries in $\Psi_r$ which appear in the relations have no effect in the upcoming computations. Therefore, we shall denote a generic group--theoretical relation of $\Gamma_{\mathcal{D}}$ by $(\gamma, s(\gamma), S(\gamma))$, where $\gamma\in\Gamma_{\o}$, $s(\gamma)\in\Psi$ and $S(\gamma)\subset\Psi$. In (\ref{example}) we have $$\gamma=\xi\eta\xi^{-1},\ s(\gamma)=\xi\eta^{-1}\xi^{-1},\  S(\gamma)=\{\xi^{2},\xi\eta^{-1}\xi^{-1},\xi\eta^{-1}\xi,\xi\eta^{-2},\xi\eta^2,\xi\eta\xi^{-1},\xi\eta\xi\}.$$ 
There are $128$ group--theoretical relations for $\Gamma_{\mathcal{D}}$ in total. But we will be interested in $60$ of them listed in \fullref{lem:2:1} (see \fullref{Table1}, \fullref{Table2}, \fullref{Table3} and \fullref{Table4}) for which $\gamma\in\Gamma_{\jrg}\subset\Psi_r\cup\Psi$ defined in (\ref{conj}). Then we consider the cases:
\begin{enumerate}[label=\roman*]
    \renewcommand{\labelenumi}{\roman{enumi}}
       \item\hspace{-.1cm}.\label{I}  when $\Gamma=\langle\xi,\eta\rangle$ is geometrically infinite; that is, $\Lambda_{\Gamma\cdot z}=S_{\infty}$ for every $z\in\hyp$,
       \item\hspace{-.1cm}.\label{II} when $\Gamma=\langle\xi,\eta\rangle$ is geometrically finite.
\end{enumerate}
Above the expression $S_{\infty}$ denotes the boundary of the canonical compactification $\overline{\hyp}$ of $\hyp$. Note that $S_{\infty}\cong S^2$. The notation $\Lambda_{\Gamma\cdot z}$ means the limit set of $\Gamma$--orbit of $z\in\hyp$ on $S_{\infty}$. In the case (\ref{I}) we first prove the statement below:

\medskip
\noindent {\bf \fullref{thm:2:1}}\quad
{\sl Let $\Gamma=\langle\xi,\eta\rangle$ be a purely loxodromic, free, geometrically infinite Kleinian group and $\Gamma_{\mathcal{D}}$ be the decomposition of $\Gamma$ in (\ref{dJ}). If $z$ denotes a point in $\hyp$, then there is a family of Borel measures $\{\nu_{\psi}\}_{\psi\in\Psi}$ defined on $S_{\infty}$ such that we have $(i)\ \ A_{z}=\sum_{\psi\in\Psi}\nu_{\psi}$;  $(ii)\ \ A_{z}(S_{\infty})=1$; and for $\gamma\in\Gamma_{\jrg}$
$$(iii)\quad\dis{\int_{S_{\infty}}\left(\lambda_{\gamma,z}\right)^2d\nu_{s(\gamma)}=1-\sum_{\psi\in S(\gamma)}\int_{S_{\infty}} d\nu_{\psi}}$$ for all group--theoretical relations $(\gamma, s(\gamma), S(\gamma))$ of $\Gamma_{\mathcal{D}}$, where $A_{z}$ is the area measure on $S_{\infty}$ based at $z$. }
\medskip


This theorem basically states that the normalised area measure $A_{z}$ on the sphere at infinity can be decomposed as a sum of Borel measures $\nu_{\psi}$ indexed by $\psi\in\Psi$ so that each group--theoretical relation of $\Gamma_{\mathcal{D}}$ translates into a measure--theoretical relation among the Borel measures $\{\nu_{\psi}\}_{\psi\in\Psi}$ as described in part ($iii$) of the theorem. In particular, each measure $\nu_{\psi}$ is transformed to the complement of certain measures in the set $\{\nu_{\gamma}\co\gamma\in\Psi-\{\psi\}\}$. For example, \fullref{thm:2:1} ($iii$) and the group--theoretical relation given in (\ref{example}) imply that 
\begin{equation}\label{ex:1:2}
\int_{S_{\infty}}\lambda_{\xi\eta\xi^{-1},z}^2\ d\nu_{\xi\eta^{-1}\xi^{-1}}=1-\sum\nolimits_{\psi\in\{\xi^2,\xi\eta^{-1}\xi^{-1},\xi\eta^{-1}\xi,\xi\eta^{-2},\xi\eta^2,\xi\eta\xi^{-1},\xi\eta\xi\}}\nu_{\psi}(S_{\infty}).
\end{equation}
By a formula proved in \cite{CSParadox} and improved in  \cite{CSMargulis} by Culler and Shalen, each hyperbolic displacement $d_{\gamma}z$ for $\gamma\in\Gamma_{\jrg}$ has a lower bound involving the Borel measures in $\{\nu_{\psi}\}_{\psi\in\Psi}$. This formula is given as follows:
\begin{lemma}\label{lem1.2}(\cite[Lemma 5.5]{CSParadox}; \cite[Lemma 2.1]{CSMargulis}) 
Let $a$ and $b$ be numbers in $[0,1]$ which are not both equal to $0$ and are not both equal to $1$. Let $\gamma$ be a loxodromic isometry of $\hyp$ and let $z$
be a point in $\hyp$. Suppose that $\nu$ is a measure on $S_{\infty}$ such that
(i) $\nu\leq A_{z}$,  (ii) $\nu\left(S_{\infty}\right)\leq a$, (iii) $\int_{S_{\infty}}(\lambda_{\gamma,z})^2d\nu\geq b$.
Then $a>0$, $b<1$, and $$d_{\gamma}z\geq \tfrac{1}{2}\log\frac{\sigma(a)}{\sigma(b)},$$
where $\sigma(x)=1/x-1$ for $x\in(0,1)$. 
\end{lemma}
Provided that $0<\nu_{s(\gamma)}(S_{\infty})<1$ for every group--theoretical relation $(\gamma,s(\gamma),S(\gamma))$ of $\Gamma_{\mathcal{D}}$, when we let $\nu=\nu_{s(\gamma)}$, $a=\nu_{s(\gamma)}(S_{\infty})$ and $b=\int_{S_{\infty}}(\lambda_{\gamma,z_0})^2d\nu_{s(\gamma)}$, \fullref{thm:2:1} and \fullref{lem1.2}  produce a set $\mathcal{G}=\{f_l\}_{l=1}^{60}$ of real--valued functions on $\Delta^{27}$ such that 
\begin{equation}\label{eqn:1:3}
e^{2d_{\gamma}z}\geq f_l(\tb{m})=\sigma\left(\dis{\sum_{\psi\in S(\gamma)}\int_{S_{\infty}} d\nu_{\psi}}\right)\sigma\left(\dis{\int_{S_{\infty}}
d\nu_{s(\gamma)}}\right)
\end{equation}
 for every $\gamma\in\Gamma_{\jrg}$ for some $l=1,\dots,60$. This is established in \fullref{dispfunc} in which formulas of the functions in $\mathcal{G}$ are explicitly stated. In the equation in (\ref{eqn:1:3}) above  $\tb{m}=(\nu_{\xi\eta^{-1}\xi^{-1}}(S_{\infty}),\dots,\nu_{\xi^{-2}}(S_{\infty}))$ is a point of the set
 \[
 \Delta^{27}=\left\{\tb{x}=(x_1,x_2,\dots,x_{28})\in\mathbb{R}^{28}_+\co\sum_{l=1}^{28}x_i=1\right\},
 \]
whose entries ordered by $p$ in (\ref{list:1:4}).  As a particular example, by the group--theoretical relation in (\ref{example}), the equality in (\ref{ex:1:2}), \fullref{lem1.2} and \fullref{dispfunc}, for $z\in\mathbb{H}^3$ we have $d_{\xi\eta\xi^{-1}}z\geq\tfrac{1}{2}\log f_1(\tb{m})$, where
\[
f_1(\tb{x})=\frac{1-x_1-x_2-x_3-x_4-x_5-x_6-x_7}{x_1+x_2+x_3+x_4+x_5+x_6+x_7}\cdot\frac{1-x_1}{x_1}.
\]
As a consequence of \fullref{thm:2:1}, \fullref{lem1.2} and \fullref{dispfunc}, in the case (\ref{I})  \fullref{thm:4:1} follows from the statement below and the inequality following;

\medskip
\noindent {\bf \fullref{thm:3:2}}\quad
{\sl If $G\co\Delta^{27}\to\mathbb{R}$ is the function defined by $\tb{x}\mapsto\max\{f(\tb{x}): f\in\mathcal{G}\}$, then we have $\inf_{\tb{x}\in\Delta^{27}}G(\tb{x})=24.8692...,$}
\medskip
\begin{equation}\label{inequality}
\max_{\gamma\in\Gamma_{\jrg}}\left\{d_{\gamma}z\right\} \geq  \tfrac{1}{2}\log G(\tb{m})
                         \geq  \tfrac{1}{2}\log\left(\inf_{\tb{x}\in\Delta^{27}} G(\tb{x})\right).
\end{equation}
\medskip

Let $\mathfrak{X}$ denote the character variety $PSL(2,\C)\times PSL(2,\C)$ and $\mathfrak{GF}$ be the set of pairs of isometries $(\xi,\eta)\in\mathfrak{X}$ such that $\langle\xi,\eta\rangle$ is free, geometrically finite and without any parabolic. In the case (\ref{II}), when $\Gamma=\langle\xi,\eta\rangle$ is geometrically finite, for a fixed $z\in\hyp$ we define the function $f_{z}\co\mathfrak{X}\to\R$ for $\Gamma_{\jrg}$, described in (\ref{conj}), with the formula
\begin{displaymath}
f_{z}(\xi,\eta)=\max_{\psi\in\Gamma_{\jrg}}\{\dpsi\}.
\end{displaymath}
This function is continuous and proper. Moreover by similar arguments given in \cite[Theorem 9.1]{CSParadox}, \cite[Theorem 5.1]{Y1} and  \cite[Theorem 4.1]{Y2}  it can be shown that it takes its minimum value in $\overline{\mathfrak{GF}}-\mathfrak{GF}$ on the open set $\mathfrak{GF}$. It is known by \cite[Propositions 9.3 and 8.2]{CSParadox}, \cite[Main Theorem]{CSH} and \cite{CCHS} that the set of $(\xi,\eta)$ such that $\langle\xi,\eta\rangle$ is free, geometrically infinite and without any parabolic is dense in $\overline{\mathfrak{GF}}-\mathfrak{GF}$ and, every $(\xi,\eta)\in\mathfrak{X}$ with  $\langle\xi,\eta\rangle$ is free and without any parabolic is in $\overline{\mathfrak{GF}}$. This reduces geometrically finite case to geometrically infinite case completing the proof of \fullref{thm:4:1}.

We shall use the geometry of the action of the loxodromic elements of $\tnr{Isom}^+(\hyp)$ to prove \fullref{thm:4:2}. Let $\xi$ and $\eta$ be two non-commuting loxodromic isometries of $\hyp$ and $z\in\hyp$. Then the displacement $d_{\xi}z$ given by $\xi$ can be expressed as
\[
\sinh^2\tfrac{1}{2}d_{\xi}z =  \sinh^2(\tfrac{1}{2}T_{\xi})\cosh^2d_z\mathcal{A}+\sin^2\theta\sinh^2d_z\mathcal{A},
\]
where $T_{\xi}$, $\theta$ and $\mathcal{A}$ are the translation length, rotational angle and axis of $\xi$, respectively. Above $d_z{\mathcal{A}}$ denotes the distance between $z$ and $\mathcal{A}$. Let $\mathcal{B}$ be the axis of $\eta\xi\eta^{-1}$. Similarly $d_{\eta\xi\eta^{-1}}z$ can be expressed as 
\[
\sinh^2\tfrac{1}{2}d_{\eta\xi\eta^{-1}}z  =  \sinh^2(\tfrac{1}{2}T_{\xi})\cosh^2d_z\mathcal{B}+\sin^2\theta\sinh^2d_z\mathcal{B}.
\]
Because $d_{\xi}z_1=d_{\eta\xi\eta^{-1}}z_1$, by reversing the inequalities used to prove \cite[Theorem 5.4.5]{B} it is possible to show that 
$$|\textnormal{trace}^2(\xi)-4|+|\textnormal{trace}(\xi\eta\xi^{-1}\eta^{-1})-2|\geq 2\sinh^2\tfrac{1}{2}d_{\xi}z_1$$
for the mid--point $z_1$ of the shortest geodesic segment joining $\mathcal{A}$ and $\mathcal{B}$. Then the main result of this paper \fullref{thm:4:2} follows from the inequality above and \fullref{thm:4:1}.  

To prove \fullref{thm:3:2}, we shall show that there exists a subset $\mathcal{F}=\{f_1,\dots,f_{28}\}$ of $\mathcal{G}$ such that $\inf_{\tb{x}\in\Delta^{27}} G(\tb{x})=\inf_{\tb{x}\in\Delta^{27}} F(\tb{x})$, where $F(\tb{x})=\max\{f(\tb{x}): f\in\mathcal{F}\}$. We will compute $\inf_{\tb{x}\in\Delta^{27}} F(\tb{x})$ by using the following properties of $F$:
\begin{enumerate}[label=\arabic*]
\renewcommand{\labelenumi}{\alph{enumi}}
\item\hspace{-.1cm}.\label{a} $\inf_{\tb{x}\in\Delta^{27}} F(\tb{x})=\min_{\tb{x}\in\Delta^{27}} F(\tb{x})=\alpha_*$ at some $\xs\in\Delta^{27}$,
\item\hspace{-.1cm}.\label{b} $\xs$ is unique and $\xs\in\Delta_{27}=\{\tb{x}\in\Delta^{27}\co f_i(\tb{x})=f_j(\tb{x})\tnr{ for every }f_i,f_j\in\mathcal{F}\}$.
\end{enumerate}
The property in (\ref{a}) is proved in \fullref{lemtwo} which exploits the fact that on any sequence $\{\tb{x}_n\}\subset\Delta^{27}$ that limits on the boundary of the simplex $\Delta^{27}$ some of the displacement functions $f_i\in\mathcal{F}$ approach to infinity. 


Each statement in the property in (\ref{b}) is proved in \fullref{prop:3:1} and \fullref{prop:3:2}, respectively. 
We shall first prove \fullref{prop:3:1}. We will see that  the functions in $\mathcal{F}'=\{f_1,f_5,f_9,f_{13},f_{15},f_{19},f_{23},f_{27}\}$ in $\mathcal{F}$ play a more important role in computing $\alpha_*$. 
At least one of the functions in $\mathcal{F}'$ takes the value $\alpha_*$. This is showed in \fullref{firstfive}. Each function $f_l$ in $\mathcal{F}'$ is a strictly convex function on an open convex subset $C_{f_l}$, defined in (\ref{C1}), of $\Delta^{27}$ for $l\in J=\{1,5,9,13,15,19,23,27\}$.  Moreover by \fullref{unique5} and \fullref{unique6} we shall show that $\xs\in C=\bigcap_{l\in J}C_{f_l}$ which is itself convex. The minimum of the maximum of the functions in $\mathcal{F}'$ on $C$ is calculated as $\alpha_*$ in \fullref{unique7}. Then by standard facts from convex analysis, \fullref{prop:3:1} will follow.

\fullref{prop:3:1} reduces the computation of $\alpha_*$ to the comparison of only four values $f_1(\xs)=\alpha_*$, $f_2(\xs)\leq\alpha_*$, $f_3(\xs)\leq\alpha_*$ and $f_7(\xs)\leq\alpha_*$, which is proved in \fullref{lem:3:3}. Considering $\Delta^{27}$ as a submanifold of $\R^{28}$, if $f_l(\xs)<\alpha_*$ for some $l\in\{2,3,7\}$, the fact that there are directions in the tangent space $T_{\xs}\Delta^{27}$ of $\Delta^{27}$ at $\xs$ so that all of the displacement functions in $\mathcal{F}$ take values strictly less than $\alpha_*$ on the line segments extending in these directions will prove \fullref{prop:3:2}.  Existence of these directions will be showed either by a direct calculation or by \fullref{flipside}.

Since the coordinate sum of $\xs$ is $1$, \fullref{prop:3:1} and \fullref{prop:3:2} together give a method to calculate the coordinates of $\xs$ explicitly. By evaluating any of the displacement functions in $\mathcal{F}$ at $\xs$ we find the value of $\alpha_*$. Details of this method will be given in \fullref{thm:3:1}. Finally we will show that $f(\xs)<\alpha_*$  for every $f\in\mathcal{G}-\mathcal{F}$ which implies that  $\alpha_*=\inf_{\tb{x}\in\Delta^{27}}G(\tb{x})$ completing the proof of \fullref{thm:3:2}. 

All of the computations summarised above to prove \fullref{thm:4:1} and \fullref{thm:4:2}  for purely loxodromic $2$-generator free Kleinian groups can be generalised to prove analogous results for purely loxodromic finitely generated free Kleinian groups. We will finish this paper by phrasing these generalisations in \fullref{conj:4:1} and \fullref{conj:4:2} and by presenting their proof sketches.



\section{Displacement functions for the isometries in $\Gamma_{\jrg}$}\label{S2}

In this section we shall determine the displacement functions for the hyperbolic displacements given by the isometries in $\Gamma_{\jrg}$. We introduce the following subsets of $\Psi$ defined in (\ref{list:1:4}): Let $\Phi_1  =  \{\xi,\eta^{-1},\eta,\xi^{-1}\}$ and $\Psi=\{\xi^2,\eta^{-2},\eta^2,\xi^{-2}\}\cup\bigcup_{l=1}^8\Psi_l$, where
\[
\begin{array}{ll}
\Psi_1 =   \{\xi\eta^{-1}\xi^{-1},\xi\eta^{-1}\xi,\xi\eta^{-2}\}, & \Psi_2 =\{\xi\eta^2,\xi\eta\xi^{-1},\xi\eta\xi\},\\
\Psi_3  =  \{\eta^{-1}\xi^{-1}\eta^{-1},\eta^{-1}\xi^{-1}\eta,\eta^{-1}\xi^{-2}\}, & \Psi_4 = \{\eta^{-1}\xi^2,\eta^{-1}\xi\eta^{-1},\eta^{-1}\xi\eta\} , \\
\Psi_5  =  \{\eta\xi^{-1}\eta^{-1},\eta\xi^{-1}\eta,\eta\xi^{-2}\}, & \Psi_6 = \{\eta\xi^2,\eta\xi\eta^{-1},\eta\xi\eta\},  \\
\Psi_7  = \{\xi^{-1}\eta^{-1}\xi^{-1},\xi^{-1}\eta^{-1}\xi,\xi^{-1}\eta^{-2}\}, & \Psi_8  = \{\xi^{-1}\eta^2,\xi^{-1}\eta\xi^{-1},\xi^{-1}\eta\xi\}.
\end{array}
\]
First we prove the statement below which gives the relevant group-theoretical relations of the decomposition $\Gamma_{\mathcal{D}}$ for the isometries in $\Gamma_{\jrg}$: 


\begin{lemma}\label{lem:2:1}
Let $\Gamma=\langle\xi,\eta\rangle$ be a $2$--generator free group and $\Gamma_{\mathcal{D}}$ be the decomposition of $\Gamma$ in (\ref{dJ}). Then there are $60$ group--theoretical relations ($\gamma,s(\gamma),S(\gamma)$) for $\gamma\in\Gamma_{\jrg}$. 
\end{lemma}
\begin{proof}
We list all of the group--theoretical relations of $\Gamma_{\mathcal{D}}$ for $\gamma\in\Gamma_{\jrg}$ defined in (\ref{conj}):
\begin{table}[H]
\begin{center}
\scalebox{1}{
\begin{tabular}{|c|c|c|c|c|c|c|c|}
  \hline
          & $\gamma$ & $s(\gamma)$ & $S(\gamma)$ &  & $\gamma$ & $s(\gamma)$ & $S(\gamma)$\\
  \hline
  $1$ & $\xi\eta\xi^{-1}$  &$\xi\eta^{-1}\xi^{-1}$ & $\{\xi^2\}\cup\Psi_{1}\cup\Psi_{2}$ &  $5$  & $\eta\xi\eta^{-1}$  & $\eta\xi^{-1}\eta^{-1}$ & $\{\eta^2\}\cup\Psi_{5}\cup\Psi_{6}$\\ 
  \hline
  $2$ &  $\xi\eta^{-1}\xi^{-1}$  & $\xi\eta\xi^{-1}$ & $\{\xi^{2}\}\cup\Psi_{1}\cup\Psi_{2}$ &  $6$  & $\eta\xi^{-1}\eta^{-1}$  &$\eta\xi\eta^{-1}$ & $\{\eta^{2}\}\cup\Psi_{5}\cup\Psi_{6}$\\
  \hline
 $3$ &  $\eta^{-1}\xi\eta$  & $\eta^{-1}\xi^{-1}\eta$ &  $\{\eta^{-2}\}\cup\Psi_{3}\cup\Psi_{4}$ &  $7$ & $\xi^{-1}\eta\xi$  &$\xi^{-1}\eta^{-1}\xi$ & $\{\xi^{-2}\}\cup\Psi_{7}\cup\Psi_{8}$  \\ 
  \hline
 $4$ &  $\eta^{-1}\xi^{-1}\eta$  & $\eta^{-1}\xi\eta$  &  $\{\eta^{-2}\}\cup\Psi_{3}\cup\Psi_{4}$ &   $8$ &  $\xi^{-1}\eta^{-1}\xi$  &$\xi^{-1}\eta\xi$ & $\{\xi^{-2}\}\cup\Psi_{7}\cup\Psi_{8}$\\ 
  \hline
\end{tabular}}
\caption{Group--theoretical relations of  $ \Gamma_{\mathcal{D}}$ with $3$--cancellation.}\label{Table1}
\end{center}
\end{table}
\begin{table}[H]
\begin{center}
\scalebox{1}{
\begin{tabular}{|c|c|c|c|c|c|c|c|}
  \hline
  & $\gamma$ & $s(\gamma)$ & $S(\gamma)$ & & $\gamma$ & $s(\gamma)$ & $S(\gamma)$\\
  \hline
 $1$ &  $\xi\eta^{-1}\xi^{-1}$  &$\xi\eta^{2}$ & $\Psi-\Psi_2$ &   $5$ & $\eta\xi^{-1}\eta^{-1}$  & $\eta\xi^{2}$ &  $\Psi-\Psi_6$\\ 
  \hline
 $2$ &   $\xi\eta\xi^{-1}$  & $\xi\eta^{-2}$ &  $\Psi-\Psi_1$ &  $6$ & $\eta\xi\eta^{-1}$  & $\eta\xi^{-2}$ &  $\Psi-\Psi_5$\\
  \hline
 $3$ & $\eta^{-1}\xi^{-1}\eta$ &  $\eta^{-1}\xi^2$  &   $\Psi-\Psi_4$ &  $7$ & $\xi^{-1}\eta^{-1}\xi$  & $\xi^{-1}\eta^2$ &  $\Psi-\Psi_8$  \\ 
  \hline
 $4$ &  $\eta^{-1}\xi\eta$  & $\eta^{-1}\xi^{-2}$ &   $\Psi-\Psi_3$ &  $8$ &  $\xi^{-1}\eta\xi$  & $\xi^{-1}\eta^{-2}$ &  $\Psi-\Psi_7$\\ 
  \hline
\end{tabular}}
\caption{Group--theoretical relations of  $\Gamma_{\mathcal{D}}$ with $2$--cancellation.}\label{Table2}
\end{center}
\end{table}
\begin{table}[H]
\begin{center}
\scalebox{1}{
\begin{tabular}{|c|c|c|c|c|c|c|c|}
  \hline
 & $\gamma$ & $s(\gamma)$ & $S(\gamma)$ &  & $\gamma$ & $s(\gamma)$ & $S(\gamma)$\\
  \hline
   $1$ & $\xi^{-1}$  & $\xi\eta^{-1}\xi^{-1}$ & $ \Psi-\Psi_3$ &  $15$ & $\eta^{-1}$  &$\eta\xi^{-1}\eta^{-1}$ &  $ \Psi-\Psi_7$\\ 
  \hline
  $2$ &  $\xi^{-1}$  & $\xi\eta^{-1}\xi$ & $ \Psi-\Psi_4$  &  $16$ &$\eta^{-1}$  &$\eta\xi^{-1}\eta$ &  $ \Psi-\Psi_8$\\
  \hline
 $3$ &  $\xi^{-1}$  & $\xi\eta^{-2}$ &  $ \Psi-\{\eta^{-2}\}$  &  $17$ &$\eta^{-1}$  & $\eta\xi^{-2}$ &  $ \Psi-\{\xi^{-2}\}$  \\ 
  \hline
 $4$ &  $\xi^{-1}$  & $\xi\eta^2$ &   $ \Psi-\{\eta^{2}\}$ &   $18$ &$\eta^{-1}$  & $\eta\xi^{2}$ &  $ \Psi-\{\xi^{2}\}$\\ 
 \hline
  $5$ &$\xi^{-1}$  & $\xi\eta\xi^{-1}$ &  $ \Psi-\Psi_5$   &   $19$ &$\eta^{-1}$  & $\eta\xi\eta^{-1}$ &  $ \Psi-\Psi_1$\\ 
 \hline
 $6$ & $\xi^{-1}$  & $\xi\eta\xi$ & $ \Psi-\Psi_6$   &  $20$ & $\eta^{-1}$  & $\eta\xi\eta$ &  $ \Psi-\Psi_2$\\ 
 \hline
 $7$ & $\xi^{-1}$  & $\xi^{2}$ &   $ \Psi-\{\xi^{2}\}\cup\Psi_1\cup\Psi_2$  &    $21$ &$\eta^{-1}$  &$\eta^{2}$ &  $ \Psi-\{\eta^2\}\cup\Psi_5\cup\Psi_6$\\ 
 \hline
  $8$ & $\eta$  & $\eta^{-1}\xi^{-1}\eta^{-1}$ &  $ \Psi-\Psi_7$ &  $22$ & $\xi$  & $\xi^{-1}\eta^{-1}\xi^{-1}$ & $ \Psi-\Psi_3$ \\ 
  \hline
  $9$ & $\eta$  & $\eta^{-1}\xi^{-1}\eta$ &  $ \Psi-\Psi_8$ &   $23$ & $\xi$  & $\xi^{-1}\eta^{-1}\xi$ & $ \Psi-\Psi_4$ \\
  \hline
 $10$ &   $\eta$  & $\eta^{-1}\xi^{-2}$ & $ \Psi-\{\xi^{-2}\}$  & $24$ & $\xi$  & $\xi^{-1}\eta^{-2}$ &   $ \Psi-\{\eta^{2}\}$ \\ 
  \hline
   $11$ & $\eta$  &$\eta^{-1}\xi^{2}$ &  $ \Psi-\{\xi^{2}\}$ &   $25$ & $\xi$  & $\xi^{-1}\eta^2$ &  $ \Psi-\{\eta^{-2}\}$ \\ 
 \hline
  $12$ &  $\eta$  &$\eta^{-1}\xi\eta^{-1}$ &  $ \Psi-\Psi_1$ &  $26$ & $\xi$  & $\xi^{-1}\eta\xi^{-1}$ &  $ \Psi-\Psi_5$ \\ 
 \hline
 $13$ &  $\eta$  &$\eta^{-1}\xi\eta$ &  $ \Psi-\Psi_2$ &  $27$ & $\xi$  & $\xi^{-1}\eta\xi$ &   $ \Psi-\Psi_6$ \\ 
 \hline
 $14$ &  $\eta$  &$\eta^{-2}$ & $ \Psi-\{\eta^{-2}\}\cup\Psi_3\cup\Psi_4$ &  $28$ & $\xi$  & $\xi^{-2}$ &  $ \Psi-\{\eta^{-2}\}\cup\Psi_7\cup\Psi_8$    \\ 
 \hline
\end{tabular}}
\caption{Group--theoretical relations of  $ \Gamma_{\mathcal{D}}$ with $1$--cancellation.}\label{Table3}
\end{center}
\end{table}
\begin{table}[H]
\begin{center}
\scalebox{1}{
\begin{tabular}{|c|c|c|c|c|c|c|c|}
  \hline
  & $\gamma$ & $s(\gamma)$ & $S(\gamma)$ & & $\gamma$ & $s(\gamma)$ & $S(\gamma)$\\
  \hline
 $1$ &  $\xi\eta^{-1}\xi^{-1}$  & $\xi\eta\xi$ & $ \Psi-\{\xi^{2}\}$ &   $5$ & $\eta\xi^{-1}\eta^{-1}$ &  $\eta\xi\eta$ &  $ \Psi-\{\eta^{2}\}$\\ 
  \hline
 $2$ &  $\xi\eta\xi^{-1}$   & $\xi\eta^{-1}\xi$ &  $ \Psi-\{\xi^{2}\}$ &  $6$ & $\eta\xi\eta^{-1}$  & $\eta\xi^{-1}\eta$ &  $ \Psi-\{\eta^{2}\}$\\
  \hline
 $3$ &  $\eta^{-1}\xi^{-1}\eta$  & $\eta^{-1}\xi\eta^{-1}$ &   $ \Psi-\{\eta^{-2}\}$ &  $7$ & $\xi^{-1}\eta^{-1}\xi$  &$\xi^{-1}\eta\xi^{-1}$ &  $ \Psi-\{\xi^{-2}\}$  \\ 
  \hline
 $4$ &  $\eta^{-1}\xi\eta$  & $\eta^{-1}\xi^{-1}\eta^{-1}$ &   $ \Psi-\{\eta^{-2}\}$ &  $8$ & $\xi^{-1}\eta\xi$   &  $\xi^{-1}\eta^{-1}\xi^{-1}$ &  $ \Psi-\{\xi^{-2}\}$\\ 
  \hline
\end{tabular}}
\caption{Group--theoretical relations of  $ \Gamma_{\mathcal{D}}$ with $2$--cancellation.}\label{Table4}
\end{center}
\end{table}
\begin{table}[H]
\begin{center}
\scalebox{1}{
\begin{tabular}{|c|c|c|c|c|c|c|c|}
  \hline
 & $\gamma$ & $s(\gamma)$ & $S(\gamma)$ & &  $\gamma$ & $s(\gamma)$ & $S(\gamma)$\\
  \hline
  $1$ &  $\xi\eta^{-1}\xi^{-1}$  & $\xi^{2}$ & $ \Psi-\{\xi\eta^{-1}\xi\}$ &  $5$ & $\xi\eta\xi^{-1}$  &$\xi^{2}$ &  $ \Psi-\{\xi\eta\xi\}$\\ 
  \hline
  $2$ &  $\eta^{-1}\xi^{-1}\eta$  & $\eta^{-2}$ &  $ \Psi-\{\eta^{-1}\xi^{-1}\eta^{-1}\}$ &  $6$ & $\eta^{-1}\xi\eta$  &$\eta^{-2}$ &  $ \Psi-\{\eta^{-1}\xi\eta^{-1}\}$\\
  \hline
  $3$ & $\eta\xi^{-1}\eta^{-1}$  & $\eta^2$ &   $ \Psi-\{\eta\xi^{-1}\eta\}$ &  $7$ & $\eta\xi\eta^{-1}$  &$\eta^2$ &  $ \Psi-\{\eta\xi\eta\}$  \\ 
  \hline
  $4$ & $\xi^{-1}\eta^{-1}\xi$  & $\xi^{-2}$ &   $ \Psi-\{\xi^{-1}\eta^{-1}\xi^{-1}\}$ &  $8$ & $\xi^{-1}\eta\xi$  &$\xi^{-2}$ &  $ \Psi-\{\xi^{-1}\eta\xi^{-1}\}$\\ 
 \hline
\end{tabular}}
\caption{Group--theoretical relations of  $ \Gamma_{\mathcal{D}}$ with $1$--cancellation.}\label{Table5}
\end{center}
\end{table}
\noindent In \fullref{Table1}-\fullref{Table5} all of the group--theoretical relations $(\gamma, s(\gamma), S(\gamma))$ of $\Gamma_{\mathcal{D}}$ for $\gamma\in\Gamma_{\jrg}$ are counted. 
\end{proof}


Given the group--theoretical relations in \fullref{lem:2:1}, we decompose the area measure on $S_{\infty}$ accordingly. This is stated in the following theorem. To save space we will not give a proof of this theorem which uses analogous  arguments presented in the proofs of  \cite[Lemma 5.3]{CSParadox}, \cite[Lemma 3.3, Theorem 3.4]{Y1} and \cite[Theorem 2.1]{Y2}.
\begin{theorem}\label{thm:2:1}
Let $\Gamma=\langle\xi,\eta\rangle$ be a purely loxodromic, free, geometrically infinite Kleinian group and $\Gamma_{\mathcal{D}}$ be the decomposition of $\Gamma$ given in (\ref{dJ}). If $z$ denotes a point in $\hyp$, then there is a family of Borel measures $\{\nu_{\psi}\}_{\psi\in\Psi}$ defined on $S_{\infty}$ such that $(i)\ \ A_{z}=\sum_{\psi\in\Psi}\nu_{\psi}$;  $(ii)\ \ A_{z}(S_{\infty})=1$; and
$$(iii)\quad {\int_{S_{\infty}}\left(\lambda_{\gamma,z}\right)^2d\nu_{s(\gamma)}=1-\sum_{\psi\in S(\gamma)}\int_{S_{\infty}} d\nu_{\psi}}$$ for each group--theoretical relation $(\gamma, s(\gamma), S(\gamma))$ of $\Gamma_{\mathcal{D}}$, where $A_{z}$ is the area measure on $S_{\infty}$ based at $z$.
\end{theorem}
Let $I=J_1\cup J_2\cup J_3\cup J_4=\{1,2,\dots,28\}$ and  $I_l$ for $l\in\{1,\dots,8\}$ be the following index sets:
\begin{equation}\label{Ind:2:1}
\scalebox{.92}{$
\begin{array}{llll}
I_1=\{1,2,3\}, & I_2=\{4,5,6\}, & I_3=\{8,9,10\}, &  I_4=\{11,12,13\},\\
I_5=\{15,16,17\}, & I_6=\{18,19,20\}, & I_7=\{22,23,24\}, & I_8=\{25,26,27\}, \\
J_1=\{1,\dots,7\}, & J_2=\{8,\dots,14\}, & J_3=\{15,\dots,21\}, & J_4=\{22,\dots,28\}.
\end{array}$}
\end{equation}
We shall use the functions $\sigma\co (0,1)\to (0,\infty)$, $\Sigma_J^i\co\Delta^{27}\to (0,1)$,  $\Sigma_i^J\co\Delta^{27}\to (0,1)$, $\Sigma_I^j\co\Delta^{27}\to (0,1)$ and $\Sigma^n\co\Delta^{27}\to (0,1)$ with formulas $\sigma(x)=1/x-1$, 
\begin{equation}\label{sigma}
\Sigma_i^J(\tb{x})=\sum_{l\in I-J_i}x_l,\ \ \Sigma_J^i(\tb{x})=\sum_{l\in J_i}x_l,\ \ \Sigma_I^j(\tb{x})=\sum_{l\in I-I_j}x_l,\ \ \Sigma^n(\tb{x})=\sum_{l\in I-\{n\}}x_l
\end{equation}
for $i\in\{1,2,3,4\}$, $j\in\{1,2,3,4,5,6,7,8\}$ and $n\in\{1,2,\dots,28\}$, respectively, to express the displacement functions compactly. In particular we prove the following:
\begin{proposition}\label{dispfunc}
Let $\Gamma=\langle\xi,\eta\rangle$ be a purely loxodromic, free, geometrically infinite Kleinian group and $ \Gamma_{\mathcal{D}}$ be the decomposition of $\Gamma$ defined in (\ref{dJ}). For any $z\in\hyp$ and for each $\gamma\in\Gamma_{\jrg}$, the value $e^{2d_{\gamma}z}$ is bounded below by $f_l(\tb{x})$, $g_i(\tb{x})$, $h_j(\tb{x})$ or $u_n(\tb{x})$ for $\tb{x}\in\Delta^{27}$ for at least one of the displacement functions $f_l$, $g_i$, $h_j$ or $u_n$ whose formulas are listed in the tables below
\begin{table}[H]
\begin{center}
\scalebox{.95}{
\begin{tabular}{|c|c|c|}
\hline
 $f_{1}(\tb{x})=\sigma\left(\Sigma_{J}^1(\tb{x})\right)\sigma(x_1)$    & $f_{15}(\tb{x})=\sigma\left(\Sigma_{J}^3(\tb{x})\right)\sigma(x_{15})$ &
 $f_{5}(\tb{x})=\sigma\left(\Sigma_{J}^1(\tb{x})\right)\sigma(x_5) $\\  
 \hline
$ f_{19}(\tb{x})=\sigma\left(\Sigma_{J}^3(\tb{x})\right)\sigma(x_{19})$ &
 $f_{9}(\tb{x})=\sigma\left(\Sigma_{J}^2(\tb{x})\right)\sigma(x_9)$  &  $f_{23}(\tb{x})=\sigma\left(\Sigma_{J}^4(\tb{x})\right)\sigma(x_{23})$\\
 \hline
$ f_{13}(\tb{x})=\sigma\left(\Sigma_{J}^2(\tb{x})\right)\sigma(x_{13})$  & & $f_{27}(\tb{x})=\sigma\left(\Sigma_{J}^4(\tb{x})\right)\sigma(x_{27})$  \\
 \hline
\end{tabular}}
\caption{Displacement functions obtained from the group--theoretical relations in \fullref{Table1}.}\label{table:2:6}
\end{center}
\end{table}
\begin{table}[H]
\begin{center}
\scalebox{.97}{
\begin{tabular}{|c|c|c|}
\hline
 $f_{4}(\tb{x})=\sigma\left(\Sigma_{I}^2(\tb{x})\right)\sigma(x_4)$    & $f_{18}(\tb{x})=\sigma\left(\Sigma_{I}^6(\tb{x})\right)\sigma(x_{18})$ &
 $f_{3}(\tb{x})=\sigma\left(\Sigma_{I}^1(\tb{x})\right)\sigma(x_3) $\\  
 \hline
$ f_{17}(\tb{x})=\sigma\left(\Sigma_{I}^5(\tb{x})\right)\sigma(x_{17})$ &
 $f_{11}(\tb{x})=\sigma\left(\Sigma_{I}^4(\tb{x})\right)\sigma(x_{11})$  &  $f_{25}(\tb{x})=\sigma\left(\Sigma_{I}^8(\tb{x})\right)\sigma(x_{25})$\\
 \hline
$ f_{10}(\tb{x})=\sigma\left(\Sigma_{I}^3(\tb{x})\right)\sigma(x_{10})$  & & $f_{24}(\tb{x})=\sigma\left(\Sigma_{I}^7(\tb{x})\right)\sigma(x_{24})$ \\
 \hline
\end{tabular}}
\caption{Displacement functions obtained from the group--theoretical relations in \fullref{Table2}.}\label{table:2:7}
\end{center}
\end{table}
\begin{table}[H]
\begin{center}
\scalebox{.95}{
\begin{tabular}{|c|c|c|}
\hline
$\dis g_{1}(\tb{x})=\sigma\left(\Sigma_{I}^3(\tb{x})\right)\sigma(x_1)$  &  $f_{2}(\tb{x})=\sigma\left(\Sigma_{I}^4(\tb{x})\right)\sigma(x_2)$ &
$\dis g_{3}(\tb{x})=\sigma\left(\Sigma^{14}(\tb{x})\right)\sigma(x_3)$        \\  \hline $g_{4}(\tb{x})=\sigma\left(\Sigma^{21}(\tb{x})\right)\sigma(x_4)$ &
$g_{5}(\tb{x})=\sigma\left(\Sigma_{I}^5(\tb{x})\right)\sigma(x_5)$  &  $f_{6}(\tb{x})=\sigma\left(\Sigma_{I}^6(\tb{x})\right)\sigma(x_6)$ \\  \hline
$f_{7}(\tb{x})=\sigma\left(\Sigma_{1}^J(\tb{x})\right)\sigma(x_7)$  &  $ f_{8}(\tb{x})=\sigma\left(\Sigma_{I}^7(\tb{x})\right)\sigma(x_8)$ &
$g_{9}(\tb{x})=\sigma\left(\Sigma_{I}^8(\tb{x})\right)\sigma(x_9)$ \\   \hline $g_{10}(\tb{x})=\sigma\left(\Sigma^{28}(\tb{x})\right)\sigma(x_{10})$  &
$g_{11}(\tb{x})=\sigma\left(\Sigma^7(\tb{x})\right)\sigma(x_{11})$  &  $f_{12}(\tb{x})=\sigma\left(\Sigma_{I}^1(\tb{x})\right)\sigma(x_{12})$ \\ \hline
$g_{13}(\tb{x})=\sigma\left(\Sigma_{I}^2(\tb{x})\right)\sigma(x_{13})$  &  $f_{14}(\tb{x})=\sigma\left(\Sigma_{2}^J(\tb{x})\right)\sigma(x_{14})$ &
$g_{15}(\tb{x})=\sigma\left(\Sigma_{I}^7(\tb{x})\right)\sigma(x_{15})$ \\   \hline $f_{16}(\tb{x})=\sigma\left(\Sigma_{I}^8(\tb{x})\right)\sigma(x_{16})$ &
$g_{17}(\tb{x})=\sigma\left(\Sigma^{28}(\tb{x})\right)\sigma(x_{17})$  &  $g_{18}(\tb{x})=\sigma\left(\Sigma^7(\tb{x})\right)\sigma(x_{18})$\\  \hline
$g_{19}(\tb{x})=\sigma\left(\Sigma_{I}^1(\tb{x})\right)\sigma(x_{19})$ &  $f_{20}(\tb{x})=\sigma\left(\Sigma_{I}^2(\tb{x})\right)\sigma(x_{20})$ &
$f_{21}(\tb{x})=\sigma\left(\Sigma_{3}^J(\tb{x})\right)\sigma(x_{21})$ \\  \hline  $f_{22}(\tb{x})=\sigma\left(\Sigma_{I}^3(\tb{x})\right)\sigma(x_{22})$ &
$g_{23}(\tb{x})=\sigma\left(\Sigma_{I}^4(\tb{x})\right)\sigma(x_{23})$ &  $g_{24}(\tb{x})=\sigma\left(\Sigma^{14}(\tb{x})\right)\sigma(x_{24})$ \\ \hline
$g_{25}(\tb{x})=\sigma\left(\Sigma^{21}(\tb{x})\right)\sigma(x_{25})$  &  $f_{26}(\tb{x})=\sigma\left(\Sigma_{I}^5(\tb{x})\right)\sigma(x_{26})$ &
$g_{27}(\tb{x})=\sigma\left(\Sigma_{I}^6(\tb{x})\right)\sigma(x_{27})$  \\  \hline & $f_{28}(\tb{x})=\sigma\left(\Sigma_{4}^J(\tb{x})\right)\sigma(x_{28})$  & \\
 \hline
\end{tabular}}
\caption{Displacement functions obtained from the group--theoretical relations in  \fullref{Table3}.}\label{table:2:8}
\end{center}
\end{table}
\begin{table}[H]
\begin{center}
\scalebox{.95}{
\begin{tabular}{|c|c|c|}
\hline
$h_{1}(\tb{x})=\sigma\left(\Sigma^{28}(\tb{x})\right)\sigma(x_{1})$ & $ h_{15}(\tb{x})=\sigma\left(\Sigma^{14}(\tb{x})\right)\sigma(x_{15})$ &
$h_{5}(\tb{x})=\sigma\left(\Sigma^{28}(\tb{x})\right)\sigma(x_{5})$  \\ \hline  $h_{19}(\tb{x})=\sigma\left(\Sigma^{14}(\tb{x})\right)\sigma(x_{19})$ &
$h_{9}(\tb{x})=\sigma\left(\Sigma^{21}(\tb{x})\right)\sigma(x_{9})$ &  $h_{23}(\tb{x})=\sigma\left(\Sigma^{7}(\tb{x})\right)\sigma(x_{23})$ \\ \hline
$h_{13}(\tb{x})=\sigma\left(\Sigma^{21}(\tb{x})\right)\sigma(x_{13})$  &  & $h_{27}(\tb{x})=\sigma\left(\Sigma^{7}(\tb{x})\right)\sigma(x_{27})$\\
 \hline
\end{tabular}}
\caption{Displacement functions obtained from the group--theoretical relations in  \fullref{Table4}.}\label{table:2:9}
\end{center}
\end{table}
\begin{table}[H]
\begin{center}
\scalebox{.94}{
\begin{tabular}{|c|c|c|}
\hline
$h_{7}(\tb{x})=\sigma\left(\Sigma^{2}(\tb{x})\right)\sigma(x_{7})$ & $ u_{7}(\tb{x})=\sigma\left(\Sigma^{6}(\tb{x})\right)\sigma(x_{7})$ &
$h_{14}(\tb{x})=\sigma\left(\Sigma^{8}(\tb{x})\right)\sigma(x_{14})$ \\ \hline       $u_{14}(\tb{x})=\sigma\left(\Sigma^{12}(\tb{x})\right)\sigma(x_{14})$ &
$h_{21}(\tb{x})=\sigma\left(\Sigma^{16}(\tb{x})\right)\sigma(x_{21})$  & $u_{21}(\tb{x})=\sigma\left(\Sigma^{20}(\tb{x})\right)\sigma(x_{21})$ \\ \hline
$h_{28}(\tb{x})=\sigma\left(\Sigma^{22}(\tb{x})\right)\sigma(x_{28})$  &  & $u_{28}(\tb{x})=\sigma\left(\Sigma^{26}(\tb{x})\right)\sigma(x_{28})$  \\ 
 \hline
\end{tabular}}
\caption{Displacement functions obtained from the group--theoretical relations in  \fullref{Table5}.}\label{table:2:10}
\end{center}
\end{table}
\end{proposition}
\begin{proof}
Let $\{\nu_{\psi}\}_{\psi\in\Psi}$ be the family of Borel measures on $S_{\infty}$ given by  \fullref{thm:2:1}. 
Since every isometry $\psi\in\Psi$ other than $\xi\eta^{-2}$, $\xi\eta^2$, $\eta^{-1}\xi^{-2}$, $\eta^{-1}\xi^2$, $\eta\xi^{-2}$, $\eta\xi^2$, $\xi^{-1}\eta^{-2}$ and $\xi^{-1}\eta^2$ has an inverse in $\Psi$, an analogous argument used in \cite[Proposition 2.1]{Y2} shows that $0<\nu_{\psi}(S_{\infty})<1$ for these isometries.

It is clear that $\nu_{\xi\eta^{-2}}(S_{\infty})\neq 1$. Because otherwise we get $\nu_{\psi}(S_{\infty})=0$ for every $\psi\in\Psi-\{\xi\eta^{-2}\}$ by \fullref{thm:2:1} (i), a contradiction. Assume that $\nu_{\xi\eta^{-2}}(S_{\infty})=0$. By the group--theoretical relation in \fullref{Table2} (2) and \fullref{thm:2:1} (iii), we derive that $\nu_{\psi}(S_{\infty})=0$ for every $\psi\in\Psi_1=\{\xi\eta^{-1}\xi^{-1},\xi\eta^{-1}\xi,\xi\eta^{-2}\}$. This is a contradiction. By using the group--theoretical relations in \fullref{Table2} together with similar arguments given above for $\xi\eta^{-2}$, we conclude that  $0<\nu_{\psi}(S_{\infty})<1$ for every $\psi\in\Psi$.

Let $m_{p(\psi)}=\int_{S_{\infty}} d\nu_{\psi}$ for the bijection $p$ in (\ref{list:1:4}). Also let $\tb{m}=(m_1,m_2,\dots,m_{28})\in\Delta^{27}$. Since $0<\nu_{\psi}(S_{\infty})<1$ for every $\psi\in\Psi$, we see by \fullref{thm:2:1} ($iii$) and ($ii$) that  $\nu_{s(\gamma)}(S_{\infty})$ and $\int_{S_{\infty}}\lambda^{2}_{\gamma,z_0}d\mu_{V_{s(\gamma)}}$ satisfy the hypothesis of \fullref{lem1.2} for each group-theoretical relation $(\gamma,s(\gamma),S(\gamma))$ of $\Gamma_{\mathcal{D}}$ for $\gamma\in\Gamma_{\jrg}$. By setting $\nu=\nu_{s(\gamma)}$, $a=\nu_{s(\gamma)}(S_{\infty})$ and $b=\int_{S_{\infty}}\lambda^{2}_{\gamma,z_0}d\mu_{V_{s(\gamma)}}$ in \fullref{lem1.2} we obtain the lower bound
\begin{equation}\label{eqn4}
\scalebox{.98}{$
\begin{array}{c}
e^{2d_{\gamma}z} \geq 
\sigma\left(\sum_{\psi\in S(\gamma)}m_{p(\psi)}\right)\sigma\left(m_{p(s(\gamma))}\right)
\end{array}
$}
\end{equation}
for each group--theoretical relation $(\gamma,s(\gamma), S(\gamma))$ of $\Gamma_{\mathcal{D}}$ so that $\gamma\in\Gamma_{\jrg}$. We replace each constant $m_{p(\psi)}$ appearing in (\ref{eqn4}) with the variable $x_{p(\psi)}$ which gives the functions listed in \fullref{table:2:6}, \fullref{table:2:7}, \fullref{table:2:8}, \fullref{table:2:9} and \fullref{table:2:10} proving the proposition.
\end{proof}

Let $\mathcal{G}=\{f_1,\dots,f_{28},g_1,g_3,\dots,g_{27},h_1,h_5,\dots,h_{27},u_7,u_{14},\dots,u_{28}\}$ be the set of all displacement functions given in the tables in the proposition above. Let $\mathcal{F}=\{f_1,\dots,f_{28}\}$. Let  $G$ be the continuous function defined as
\begin{equation}\label{G}
\begin{array}{lllll}
G&: & \Delta^{27} &\ra&\mathbb{R}\\
  &  & \tb{x} &\mapsto&\max\{f(\tb{x}):f\in\mathcal{G}\}.
\end{array}
\end{equation}
In the next section we calculate $\inf_{\tb{x}\in\Delta^{27}}G(\tb{x})$ by  using the subset $\mathcal{F}$ of functions in $\mathcal{G}$. 

We finish \fullref{S2} by listing explicit formulas of some of the displacement functions from each group $\{f_l\}$, $\{g_i\}$, $\{h_j\}$ and $\{u_k\}$ in $\mathcal{G}$ as examples to clarify the use of compact forms in these functions. For the index sets $J_1=\{1,2,3,4,5,6,7\}$, $J_2=\{8,9,10,11,12,13,14\}$ and $I_3=\{8,9,10\}$ we have
\[
f_9(\tb{x})=\sigma(\Sigma_J^2(\tb{x}))\sigma(x_9)=\frac{1-x_8-x_9-x_{10}-x_{11}-x_{12}-x_{13}-x_{14}}{x_8+x_9+x_{10}+x_{11}+x_{12}+x_{13}+x_{14}}\cdot\frac{1-x_9}{x_9},
\]
\[
f_7(\tb{x})=\sigma(\Sigma^J_1(\tb{x}))\sigma(x_7)=\frac{1-x_8-x_9-\dots-x_{27}-x_{28}}{x_8+x_9+\dots+x_{27}+x_{28}}\cdot\frac{1-x_7}{x_7},
\]
\[
g_1(\tb{x})=\sigma(\Sigma_I^3(\tb{x}))\sigma(x_1)=\frac{1-x_1-x_2-\dots-x_7-x_{11}-\dots-x_{28}}{x_1+x_2+\dots+x_7+x_{11}+\dots+x_{28}}\cdot\frac{1-x_1}{x_1},
\] 
\[
g_{18}(\tb{x})=\sigma(\Sigma^{7}(\tb{x}))\sigma(x_{18})=\frac{1-x_1-x_2-\dots-x_6-x_{8}-\dots-x_{28}}{x_1+x_2+\dots+x_6+x_{8}+\dots+x_{28}}\cdot\frac{1-x_{18}}{x_{18}},
\]
\[
h_{1}(\tb{x})=\sigma(\Sigma^{28}(\tb{x}))\sigma(x_{1})=\frac{1-x_1-x_2-x_3-\dots-x_{27}}{x_1+x_2+x_{3}+\dots+x_{27}}\cdot\frac{1-x_{1}}{x_{1}},
\]  
\[
u_{7}(\tb{x})=\sigma(\Sigma^{6}(\tb{x}))\sigma(x_{7})=\frac{1-x_1-\dots-x_5-x_7-\dots-x_{28}}{x_1+\dots+x_5+x_7+\dots+x_{28}}\cdot\frac{1-x_{7}}{x_{7}}.
\] 
Note that in the formula of $f_9$ only variables enumerated by the elements of $J_2$ appear in the first multiple.  In the formula of $f_7$, variables enumerated by the elements of $J_1$ are missing in the first factor. Similarly in the formula of $g_1$ variables enumerated by the elements of $I_3$ are missing. In the formulas of $g_{18}$, $h_{1}$ and $u_7$, variables $x_{7}$, $x_{28}$ and $x_6$ are missing, respectively, in the first quotients.



\section{Infima of the Maximum of the Functions in $\mathcal{G}$ on $\Delta^{27}$}\label{S3}

In this section we will mostly be dealing  with the functions in $\mathcal{F}=\{f_l\}_{l\in I}$, where $I=\{1,2,\dots, 28\}$. We will show that $\inf_{\tb{x}\in\Delta^{27}}G(\tb{x})=\inf_{\tb{x}\in\Delta^{27}}F(\tb{x})$ (see \fullref{thm:3:1} and \fullref{thm:3:2}), such that  $F$ is the continuous function which has the formula
\begin{equation}\label{F}
\begin{array}{lllll}
F &\co & \Delta^{27} &\to &\mathbb{R}\\
  &  & \tb{x} &\mapsto &\max\left(f_1(\tb{x}),f_2(\tb{x}),\dots,f_{28}(\tb{x})\right).
\end{array}
\end{equation}
Therefore, it is enough to calculate $\inf_{\tb{x}\in\Delta^{27}}F(\tb{x})$. We start with the following lemma:

\begin{lemma}\label{lemtwo}
If $F$ is the function defined in (\ref{F}), then $\inf_{\tb{x}\in\Delta^{27}}F(\tb{x})$ is attained in $\Delta^{27}$ and contained in the interval $[1,\alpha]$, where $\alpha=24.8692...$ the only real root of  the polynomial $21 x^4 - 496 x^3 - 654 x^2 + 24 x + 81$ that is greater than $9$.
\end{lemma}
\begin{proof}
To save space we refer the readers to  \cite[Lemma 4.2]{Y1} and \cite[Lemma 3.1]{Y2} for the details of the proof of the statement $\inf_{\tb{x}\in\Delta^{27}}F(\tb{x})=\min_{\tb{x}\in\Delta^{27}}F(\tb{x})$. Briefly, the equality follows from the observation that on any sequence in $\Delta^{27}$ that limits on the boundary of  $\Delta^{27}$ some of the functions in $\mathcal{F}$ approach to infinity. 

For some $l\in I=\{1,2,\dots,28\}$ we have $f_l(\tb{x})>1$ for every $\tb{x}\in\Delta^{27}$  which shows $\min_{\tb{x}\in\Delta^{27}}F(\tb{x})\geq 1$. Consider the point $\tb{y}^*=(y_1,y_2,\dots,y_{28})$ in $\Delta^{27}$ such that $y_l= 1/(1+3\alpha)= 0.0132...$ for $l\in\{7,14,21,28\}$, $y_l=3/(3+\alpha)= 0.1076...$ for $l\in\{1,5,9,13,15,19,23,27\}$ and $y_l=3(\alpha-1)/(21\alpha^2+14\alpha-3)= 0.0053...$ for indices $l\in\{2,6,8,12,16,20,22,26\}$ and $l\in\{3,4,10,11,17,18,24,25\}$. Then we see that $f_l(\ys)=\alpha$ for every $l\in I$. This completes the proof. 
\end{proof}

In the rest of this text we will consider $\Delta^{27}$ as a sub--manifold of $\mathbb{R}^{28}$.
The tangent space $T_{\tb{x}}\Delta^{27}$ at any $\tb{x}\in\Delta^{27}$ consists of vectors whose coordinates sum to $0$. Note that each displacement function $f_i$ for $i\in I$ is smooth in an open neighbourhood of $\Delta^{27}$. Therefore, the directional derivative of $f_i$ in the direction of any $\vec v\in T_{\tb{x}}\Delta^{27}$ is given by $\nabla f_i(\tb{x})\cdot\vec v$ for any $i\in I=\{1,2,\dots,28\}$. The notation $\xs=(x_1^*,x_2^*,\dots,x_{28}^*)$ will be used to denote a point at which the infimum of $F$ is attained on $\Delta^{27}$. We shall use $\alpha_*$ to denote the infimum of the maximum of the functions in $\mathcal{F}$ on $\Delta^{27}$, ie 
\[
\alpha_*=\min_{\tb{x}\in\Delta^{27}}F(\tb{x}).
\]

The displacement functions $\{f_l\}_{l\in J}$ for $J=\{1,5,9,13,15,19,23,27\}$ in $\mathcal{F}$ play a special role in computing $\alpha_*$. In particular we have the following statement:
\begin{lemma}\label{firstfive}
Let $\xs\in\Delta^{27}$ so that $F(\xs)=\alpha_*$. We have $f_l(\xs)=\alpha_*$ for some $l\in J$.
\end{lemma}
\begin{proof}
Assume on the contrary that $f_l(\xs)<\alpha_*$ for every $l\in J$. Let $C_i^j$ denote the partial derivative of $f_i$ with respect to $x_j$ at $\xs=(x_1^*,x_2^*,\dots,x_{28}^*)$. We form the $20\times 28$ matrix below whose rows are $\nabla f_l(\xs)$ for $l\in I-J$:  
\begin{equation*}
\scalebox{0.65}{$
\left[
\begin{array}{cccccccccccccccccccccccccccc}
 C_2^1 & C_2^2 & C_2^1 & C_2^1 &C_2^1& C_2^1 & C_2^1 & C_2^1 & C_2^1 & C_2^1 & 0 & 0 & 0 & C_2^1 & C_2^1 & C_2^1 & C_2^1 & C_2^1 & C_2^1 & C_2^1 & C_2^1 & C_2^1 & C_2^1 & C_2^1 & C_2^1 & C_2^1 & C_2^1 & C_2^1 \\
 0 & 0 & C_3^3 & C_3^4 &C_3^4 & C_3^4 & C_3^4 & C_3^4 & C_3^4 & C_3^4 & C_3^4 & C_3^4 & C_3^4 & C_3^4 & C_3^4 & C_3^4 & C_3^4 & C_3^4 & C_3^4 & C_3^4 & C_3^4 & C_3^4 & C_3^4 & C_3^4 & C_3^4 & C_3^4 & C_3^4 & C_3^4 \\
C_4^1 &  C_4^1 &  C_4^1 & C_4^4 & 0 & 0 &  C_4^1 &  C_4^1 &  C_4^1 &  C_4^1 &  C_4^1 &  C_4^1 &  C_4^1 &  C_4^1 &  C_4^1 &  C_4^1 &  C_4^1 &  C_4^1 &  C_4^1 &  C_4^1 &  C_4^1 &  C_4^1 &  C_4^1 &  C_4^1 &  C_4^1 &  C_4^1 &  C_4^1 & C_4^1 \\
C_6^1 &  C_6^1 &  C_6^1 &  C_6^1& C_6^1 &  C_6^6 & C_6^1 &  C_6^1 &  C_6^1 &  C_6^1 &  C_6^1 &  C_6^1 &  C_6^1 &  C_6^1 &  C_6^1 &  C_6^1 &  C_6^1& 0 & 0 & 0 & C_6^1 &  C_6^1 &  C_6^1 &  C_6^1 &  C_6^1 &  C_6^1 &  C_6^1 &  C_6^1 \\
 0 & 0 & 0 & 0 & 0 & 0 & C_7^7  & C_7^8 & C_7^8 & C_7^8 &C_7^8 &C_7^8 &C_7^8 &C_7^8 &C_7^8 &C_7^8 &C_7^8 &C_7^8 &C_7^8 &C_7^8 &C_7^8 &C_7^8 &C_7^8 &C_7^8 &C_7^8 &C_7^8 &C_7^8 &C_7^8  \\
 C_8^1 & C_8^1 & C_8^1 & C_8^1 & C_8^1 & C_8^1 & C_8^1 & C_8^8 & C_8^1 & C_8^1 & C_8^1 & C_8^1 & C_8^1 & C_8^1 & C_8^1 & C_8^1 & C_8^1 & C_8^1 & C_8^1 & C_8^1 & C_8^1 & 0 & 0 & 0 & C_8^1 & C_8^1 & C_8^1 & C_8^1 \\
C_{10}^1 & C_{10}^1 & C_{10}^1 & C_{10}^1 & C_{10}^1 & C_{10}^1 & C_{10}^1 & 0 & 0 & C_{10}^{10} & C_{10}^1 & C_{10}^1 & C_{10}^1 & C_{10}^1 & C_{10}^1 & C_{10}^1 & C_{10}^1 & C_{10}^1 & C_{10}^1 & C_{10}^1 & C_{10}^1 & C_{10}^1 & C_{10}^1 & C_{10}^1 & C_{10}^1 & C_{10}^1 & C_{10}^1 & C_{10}^1 \\
 C_{11}^1 & C_{11}^1 & C_{11}^1 & C_{11}^1 & C_{11}^1 & C_{11}^1 & C_{11}^1& C_{11}^1 & C_{11}^1 & C_{11}^1  & C_{11}^{11} & 0 & 0 & C_{11}^1 & C_{11}^1 & C_{11}^1 & C_{11}^1 & C_{11}^1 & C_{11}^1 & C_{11}^1 & C_{11}^1 & C_{11}^1 & C_{11}^1 & C_{11}^1 & C_{11}^1 & C_{11}^1 & C_{11}^1 & C_{11}^1 \\
 0 & 0 & 0 & C_{12}^4 & C_{12}^4 & C_{12}^4 & C_{12}^4 & C_{12}^4 & C_{12}^4 & C_{12}^4 & C_{12}^4 & C_{12}^{12 }& C_{12}^4 & C_{12}^4 & C_{12}^4 & C_{12}^4 & C_{12}^4 & C_{12}^4 & C_{12}^4 & C_{12}^4 & C_{12}^4 & C_{12}^4 & C_{12}^4 & C_{12}^4 & C_{12}^4 & C_{12}^4 & C_{12}^4 & C_{12}^4 \\
 C_{14}^1 & C_{14}^1 & C_{14}^1 & C_{14}^1 & C_{14}^1 & C_{14}^1 & C_{14}^1 & 0 & 0 & 0 & 0 & 0 & 0 & C_{14}^{14}  & C_{14}^1 & C_{14}^1 & C_{14}^1 & C_{14}^1 & C_{14}^1 & C_{14}^1 & C_{14}^1 & C_{14}^1 & C_{14}^1 & C_{14}^1 & C_{14}^1 & C_{14}^1 & C_{14}^1 & C_{14}^1 \\
 C_{16}^1 &  C_{16}^1 &  C_{16}^1 &  C_{16}^1 &  C_{16}^1 &  C_{16}^1 &  C_{16}^1 &  C_{16}^1 &  C_{16}^1 &  C_{16}^1 &  C_{16}^1 &  C_{16}^1 &  C_{16}^1 &  C_{16}^1 &  C_{16}^1 &  C_{16}^{16} &  C_{16}^1 &  C_{16}^1 &  C_{16}^1 &  C_{16}^1 &  C_{16}^1 &  C_{16}^1 &  C_{16}^1 &  C_{16}^1 & 0 & 0 & 0 &  C_{16}^1 \\
C_{17}^1& C_{17}^1 & C_{17}^1 & C_{17}^1 & C_{17}^1 & C_{17}^1 & C_{17}^1 & C_{17}^1 & C_{17}^1 & C_{17}^1 & C_{17}^1 & C_{17}^1 & C_{17}^1 & C_{17}^1 & 0 & 0 & C_{17}^{17} & C_{17}^1 & C_{17}^1 & C_{17}^1 & C_{17}^1 & C_{17}^1 & C_{17}^1 & C_{17}^1 & C_{17}^1 & C_{17}^1 & C_{17}^1 & C_{17}^1 \\
 C_{18}^1 &  C_{18}^1 &  C_{18}^1 &  C_{18}^1 &  C_{18}^1 &  C_{18}^1 &  C_{18}^1 &  C_{18}^1 &  C_{18}^1 &  C_{18}^1 &  C_{18}^1 &  C_{18}^1 &  C_{18}^1 &  C_{18}^1 &  C_{18}^1 &  C_{18}^1 &  C_{18}^1 &  C_{18}^{18} & 0 & 0 &  C_{18}^1 &  C_{18}^1 &  C_{18}^1 &  C_{18}^1 &  C_{18}^1 &  C_{18}^1 &  C_{18}^1 &  C_{18}^1 \\
 C_{20}^{1}  & C_{20}^{1} & C_{20}^{1} & 0 & 0 & 0 & C_{20}^{1} & C_{20}^{1} & C_{20}^{1} & C_{20}^{1} & C_{20}^{1} & C_{20}^{1} & C_{20}^{1} & C_{20}^{1} & C_{20}^{1} & C_{20}^{1} & C_{20}^{1} & C_{20}^{1} & C_{20}^{1} & C_{20}^{20} & C_{20}^{1} & C_{20}^{1} & C_{20}^{1} & C_{20}^{1} & C_{20}^{1} & C_{20}^{1} & C_{20}^{1} & C_{20}^{1} \\
 C_{21}^{1} & C_{21}^{1} & C_{21}^{1} & C_{21}^{1} & C_{21}^{1} & C_{21}^{1} & C_{21}^{1} & C_{21}^{1} & C_{21}^{1} & C_{21}^{1} & C_{21}^{1} & C_{21}^{1} & C_{21}^{1} & C_{21}^{1} & 0 & 0 & 0 & 0 & 0 & 0 & C_{21}^{21}  & C_{21}^{1} & C_{21}^{1} & C_{21}^{1} & C_{21}^{1} & C_{21}^{1} & C_{21}^{1} & C_{21}^{1} \\
 C_{22}^{1} & C_{22}^{1} & C_{22}^{1} & C_{22}^{1} & C_{22}^{1} & C_{22}^{1} & C_{22}^{1} & 0 & 0 & 0 & C_{22}^{1} & C_{22}^{1} & C_{22}^{1} & C_{22}^{1} & C_{22}^{1} & C_{22}^{1} & C_{22}^{1} & C_{22}^{1} & C_{22}^{1} & C_{22}^{1} & C_{22}^{1} & C_{22}^{22} & C_{22}^{1} & C_{22}^{1} & C_{22}^{1} & C_{22}^{1} & C_{22}^{1} & C_{22}^{1} \\
 C_{24}^{1} & C_{24}^{1} & C_{24}^{1} & C_{24}^{1} & C_{24}^{1} & C_{24}^{1} & C_{24}^{1} & C_{24}^{1} & C_{24}^{1} & C_{24}^{1} & C_{24}^{1} & C_{24}^{1} & C_{24}^{1} & C_{24}^{1} & C_{24}^{1} & C_{24}^{1} & C_{24}^{1} & C_{24}^{1} & C_{24}^{1} & C_{24}^{1} & C_{24}^{1} & 0 & 0 & C_{24}^{24} & C_{24}^{1} & C_{24}^{1} & C_{24}^{1} & C_{24}^{1} \\
C_{25}^{1} & C_{25}^{1} & C_{25}^{1} & C_{25}^{1} & C_{25}^{1} & C_{25}^{1} & C_{25}^{1} & C_{25}^{1} & C_{25}^{1} & C_{25}^{1} & C_{25}^{1} & C_{25}^{1} & C_{25}^{1} & C_{25}^{1} & C_{25}^{1} & C_{25}^{1} & C_{25}^{1} & C_{25}^{1} & C_{25}^{1} & C_{25}^{1} & C_{25}^{1} & C_{25}^{1} & C_{25}^{1} & C_{25}^{1} & C_{25}^{25} & 0 & 0 & C_{25}^{1} \\
C_{26}^{1} & C_{26}^{1} & C_{26}^{1} & C_{26}^{1} & C_{26}^{1} & C_{26}^{1} & C_{26}^{1} & C_{26}^{1} & C_{26}^{1} & C_{26}^{1} & C_{26}^{1} & C_{26}^{1} & C_{26}^{1} & C_{26}^{1}  & 0 & 0 & 0 & C_{26}^{1} & C_{26}^{1} & C_{26}^{1} & C_{26}^{1} & C_{26}^{1} & C_{26}^{1} & C_{26}^{1} & C_{26}^{1} & C_{26}^{26} & C_{26}^{1} & C_{26}^{1} \\
 C_{28}^1 & C_{28}^1 & C_{28}^1 & C_{28}^1 & C_{28}^1 & C_{28}^1 & C_{28}^1 & C_{28}^1 & C_{28}^1 & C_{28}^1 & C_{28}^1 & C_{28}^1 & C_{28}^1 & C_{28}^1 & C_{28}^1 & C_{28}^1 & C_{28}^1 & C_{28}^1 & C_{28}^1 & C_{28}^1 & C_{28}^1 & 0 & 0 & 0 & 0 & 0 & 0 & C_{28}^{28}  
\end{array}\right],$}
\end{equation*}
where the entries are given as follows:
\begin{equation*}
\begin{array}{ccc}
 \dis C_2^1=-\frac{\sigma(x_2^*)}{ \left(\Sigma_I^4(\xs)\right)^2}, & C_2^2=-\dis\frac{\sigma(x_2^*)}{\left(\Sigma_{I}^4(\xs)\right)^2}-\frac{\sigma\left(\Sigma_{I}^4(\xs)\right)}{(x_2^*)^2},  & C_3^3=-\dis \frac{\sigma\left(\Sigma_I^1(\xs)\right)}
{ \left(x_3^*\right)^2}\\
 C_3^4=-\dis\frac{\sigma(x_3^*)}
{ \left(\Sigma_I^1(\xs)\right)^2}, & C_4^1=-\dis\frac{\sigma(x_4^*)}
{ \left(\Sigma_I^2(\xs)\right)^2}, & C_4^4=-\dis \frac{\sigma\left(\Sigma_I^2(\xs)\right)}
{ \left(x_4^*\right)^2},
\end{array}
\end{equation*}
\begin{equation*}
\begin{array}{ccc}
C_6^1=-\dis\frac{\sigma(x_6^*)}
{\left(\Sigma_I^6(\xs)\right)^2}, & C_6^6=-\dis\frac{\sigma(x_6^*)}{\left(\Sigma_{I}^6(\xs)\right)^2}-\frac{\sigma\left(\Sigma_{I}^6(\xs)\right)}{(x_6^*)^2}, & C_7^7=-\dis\frac{\sigma\left(\Sigma^J_1(\xs)\right)}
{\left(x_7^*\right)^2}, \\
 C_7^8=-\dis\frac{\sigma(x_7^*)}
{\left(\Sigma_1^J(\xs)\right)^2}, & C_8^8=-\dis\frac{\sigma(x_8^*)}{\left(\Sigma_{I}^7(\xs)\right)^2}-\frac{\sigma\left(\Sigma_{I}^7(\xs)\right)}{(x_8^*)^2}, &
C_8^1=-\dis\frac{\sigma(x_8^*)}
{ \left(\Sigma_I^7(\xs)\right)^2},
\end{array}
\end{equation*}
\begin{equation*}
\begin{array}{ccc}
C_{10}^1=-\dis\frac{\sigma(x_{10}^*)}
{ \left(\Sigma_I^3(\xs)\right)^2}, & C_{10}^{10}=-\dis\frac{\sigma\left(\Sigma^3_I(\xs)\right)}
{\left(x_{10}^*\right)^2}, & C_{11}^1=-\dis\frac{\sigma(x_{11}^*)}
{\left(\Sigma_I^4(\xs)\right)^2},\\
 C_{11}^{11}=-\dis\frac{\sigma\left(\Sigma^4_I(\xs)\right)}
{\left(x_{11}^*\right)^2},  & C_{12}^{12}=-\dis\frac{\sigma(x_{12}^*)}{\left(\Sigma_{I}^1(\xs)\right)^2}-\frac{\sigma\left(\Sigma_{I}^1(\xs)\right)}{(x_{12}^*)^2}, & C_{12}^{4}=-\dis\frac{\sigma(x_{12}^*)}
{ \left(\Sigma_I^1(\xs)\right)^2},\\
 C_{14}^1= -\dis\frac{\sigma(x_{14}^*)}
{\left(\Sigma_2^J(\xs)\right)^2}, & C_{14}^{14}=-\dis\frac{\sigma\left(\Sigma^J_2(\xs)\right)}
{\left(x_{14}^*\right)^2},  & C_{16}^1=-\dis\frac{\sigma(x_{16}^*)}
{\left(\Sigma_I^8(\xs)\right)^2}, 
\end{array}
\end{equation*}
\begin{equation*}
\begin{array}{ccc}
C_{17}^1=-\dis\frac{\sigma(x_{17}^*)}
{ \left(\Sigma_I^5(\xs)\right)^2},  & C_{16}^{16}=-\dis\frac{\sigma(x_{16}^*)}{\left(\Sigma_{I}^8(\xs)\right)^2}-\frac{\sigma\left(\Sigma_{I}^8(\xs)\right)}{(x_{16}^*)^2},  &
C_{17}^{17}=-\dis\frac{\sigma\left(\Sigma^5_I(\xs)\right)}
{\left(x_{17}^*\right)^2},\\  C_{18}^1=-\dis\frac{\sigma(x_{18}^*)}
{\left(\Sigma_I^6(\xs)\right)^2},  & C_{18}^{18}=-\dis\frac{\sigma\left(\Sigma^6_I(\xs)\right)}
{ \left(x_{18}^*\right)^2}, & C_{20}^1=-\dis\frac{\sigma(x_{20}^*)}
{\left(\Sigma_I^2(\xs)\right)^2},
\end{array}
\end{equation*}
\begin{equation*}
\begin{array}{ccc}
C_{21}^{1}=-\dis\frac{\sigma(x_{21}^*)}
{\left(\Sigma_3^J(\xs)\right)^2}, & C_{20}^{20}=-\dis\frac{\sigma(x_{20}^*)}{\left(\Sigma_{I}^2(\xs)\right)^2}-\frac{\sigma\left(\Sigma_{I}^2(\xs)\right)}{(x_{20}^*)^2}, &  C_{21}^{21}=-\dis\frac{\sigma\left(\Sigma^J_3(\xs)\right)}
{\left(x_{21}^*\right)^2}, \\
C_{22}^{1}=-\dis\frac{\sigma(x_{22}^*)}
{\left(\Sigma_I^3(\xs)\right)^2}, & C_{22}^{22}=-\dis\frac{\sigma(x_{22}^*)}{\left(\Sigma_{I}^3(\xs)\right)^2}-\frac{\sigma\left(\Sigma_{I}^3(\xs)\right)}{(x_{22}^*)^2}, & C_{24}^{1}=-\dis\frac{\sigma(x_{24}^*)}
{ \left(\Sigma_I^7(\xs)\right)^2},
\end{array}
\end{equation*}
\begin{equation*}
\begin{array}{ccc}
  C_{24}^{24}=-\dis\frac{\sigma\left(\Sigma_I^7(\xs)\right)}
{ \left(x_{24}^*\right)^2}, &
C_{25}^{1}=-\dis\frac{\sigma(x_{25}^*)}
{\left(\Sigma_I^8(\xs)\right)^2}, & C_{25}^{25}=-\dis\frac{\sigma\left(\Sigma^8_I(\xs)\right)}
{\left(x_{25}^*\right)^2}, \\ C_{26}^{1}=-\dis\frac{\sigma(x_{26}^*)}
{\left(\Sigma_I^5(\xs)\right)^2}, & C_{26}^{26}=-\dis\frac{\sigma(x_{26}^*)}{\left(\Sigma_{I}^5(\xs)\right)^2}-\frac{\sigma\left(\Sigma_{I}^5(\xs)\right)}{(x_{26}^*)^2}  & C_{28}^{1}=-\dis\frac{\sigma(x_{28}^*)}
{ \left(\Sigma_4^J(\xs)\right)^2}, \\
  C_{28}^{28}=-\dis\frac{\sigma\left(\Sigma_4^J(\xs)\right)}
{\left(x_{28}^*\right)^2}.  & &
\end{array}
\end{equation*}
Consider the vector $\vec u\in T_{\xs}\Delta^{27}$ with the coordinates:
\begin{equation*}
 (\vec u)_i=\left\{\begin{array}{rl}
 1   & \tnr{if $i=2,3,4,6,7,8,10,11,12,16,17,18,20,21,22,24,25,26$,}\\
-3  & \tnr{if $i=5,9,13,19,23,27$,}\\
2  & \tnr{if $i=14,28$},\\
-2  & \tnr{if $i=1,15$}.
\end{array}\right.
\end{equation*}
For $l\in\{2,3,6,7,8,16,17,20,21,22\}$, $i\in\{14,28\}$, $j\in\{4,10,11,18,24,25\}$ and $k\in\{12,26\}$ we compute that 
\begin{equation*}
\nabla f_l(\xs)\cdot \vec u=C_l^l<0,\quad \nabla f_i(\xs)\cdot \vec u=2C_i^i<0,\quad\nabla f_j(\xs)\cdot \vec u=C_j^1+C_j^j<0,
\end{equation*}
\begin{equation*}
\nabla f_k(\xs)\cdot \vec u=C_k^4+C_k^k<0.
\end{equation*}
This implies that the values of $f_l$ for $l\in I-J$ decrease along a line segment in the direction of $\vec u$. For a sufficiently short distance along $\vec u$ the values of $f_l$ for $l\in J$ are smaller than $\alpha_*$. So there exists a point $\tb{z}\in\Delta^{27}$ such that  $f_l(\tb{z})<\alpha_*$ for every $l\in I=\{1,2,\dots,28\}$. This is a contradiction.  Hence, $f_l(\xs)=\alpha_*$ for some $l\in J=\{1,5,9,13,15,19,23,27\}$. 
\end{proof}

Let $\Delta=\{(x,y)\in\mathbb{R}^2\co x+y<1,\ 0<x, 0<y\}$. Introduce the function $g\co\Delta\to (0,1)$ defined by 
\begin{equation}\label{g}
\dis{g(x,y)=\frac{1-x-y}{x+y}\cdot\frac{1-y}{y}}.
\end{equation}
Given a displacement function $f_l$ in $\mathcal{F}$ for $l\in J=\{1,5,9,13,15,19,23,27\}$, it can be expressed as 
$$
f_l(\tb{x})=g\left(\Sigma_J^i(\tb{x})-x_l,x_l\right)$$ for some $i\in\{1,2,3,4\}$. 
The function $g$ was also used in \cite{Y2}. In fact  the following statement  \cite[Lemma 3.2]{Y2} was proved for $g$:
\begin{lemma}\label{convex2}
Let $C_g=\{(x,y)\in\Delta\co x+2y-xy-y^2<\tfrac{3}{4}\}$. Then $C_g$ is an open convex set and $g(x,y)$ is a strictly convex function on $C_g$. 
\end{lemma}
Therefore, by this lemma, each displacement function $f_l$ for $l\in J$ is a strictly convex function over the open convex subset 
\begin{equation}\label{C1}
C_{f_l}=\{\tb{x}=(x_1,\dots,x_{28})\in\Delta^{27}\co \Sigma(\tb{x})+2x_l-\Sigma(\tb{x})x_l-(x_l)^2<\tfrac{3}{4}\}
\end{equation}
of $\Delta^{27}$, where we set $\Sigma(\tb{x})=\Sigma_{J}^i(\tb{x})-x_l$ for a chosen $i\in\{1,2,3,4\}$ depending on $l$. 

If $C_{f_l}$ for $l\in J$ are as described above, then the subset $C=\dis{\cap_{l\in J}C_{f_l}}$ of $\Delta^{27}$ is nonempty. This is because, if we consider the point $\ys$ given in the proof of \fullref{lemtwo}, then $$\Sigma_J^i(\tb{y}^*)-y_l= 0.1423...$$ for every $i\in\{1,2,3,4\}$. We find that $\Sigma(\tb{y}^*)+2y_l-\Sigma(\tb{y}^*)y_l-(y_l)^2=0.3307...<\tfrac{3}{4}$ for every $l\in J$.  
Thus $\tb{y}^*$ is in $C$. 
Additionally we have $\xs=(x_1^*,x_2^*,\dots,x_{28}^*)\in C$ implied by the following two lemmas:
\begin{lemma}\label{unique5}
Let $\xs\in\Delta^{27}$ so that $\alpha_*=F(\xs)$. Then  $\xs\in C_{f_1}$, defined in (\ref{C1}), where \[
f_1(\tb{x})=\sigma(\Sigma_J^1)\sigma(x_1)=\frac{1-x_1-x_2-x_3-x_4-x_5-x_6-x_7}{x_1+x_2+x_3+x_4+x_5+x_6+x_7}\cdot\frac{1-x_1}{x_1}.
\]
\end{lemma}
\begin{proof}
Assume on the contrary that $\xs\not\in C_{f_1}$. Then by the definition of $C_{f_1}$ we have 
\begin{equation}\label{ineq:3:1}
\sum_{l=2}^7x_l^*+\left(2-\sum_{l=2}^7x_l^*\right)x_1^*-(x_1^*)^2\geq\frac{3}{4}.
\end{equation}
Let us say $N=\tfrac{1}{4}(3-\sqrt{3})\approx 0.3170.$ Also let $\Sigma_1^*=\sum_{l=1}^{7}x_l^*=\Sigma_J^1(\tb{x}^*)$, $\Sigma_2^*=\sum_{l=8}^{14}x_l^*=\Sigma_J^2(\tb{x}^*)$, $\Sigma_3^*=\sum_{l=15}^{21}x_l^*=\Sigma_J^3(\tb{x}^*)$ and $\Sigma_4^*=\sum_{l=22}^{28}x_l^*=\Sigma_J^4(\tb{x}^*)$. Consider the cases:
\begin{equation}\label{ABC2}
\begin{array}{c}
(A)\ \ \Sigma(\xs)\geq N,\ x_1^*\geq N,\ (B)\ \ \Sigma(\xs)\geq N > x_1^*,\ (C)\ \ x_1^*\geq N > \Sigma(\xs),
\end{array}
\end{equation}
where $\Sigma(\xs)=\Sigma_J^1(\xs)-x_1^*=\sum_{l=2}^7x_l^*$.  Assume that (A) is the case. Note that $\Sigma_1^*\geq 2N$. Then we have
\begin{equation}\label{eqn:3:4}
\Sigma_2^*+\Sigma_3^*+\Sigma_4^*\leq M=1-2N\approx 0.3660.
\end{equation}
If $\Sigma_2^*\leq M/3\approx 0.1220$, using \fullref{lemtwo} and $\sigma(M/3)\sigma(x_l^*)\leq\sigma(\Sigma_2^*)\sigma(x_l^*)\leq\alpha$ we find for every $l\in\{9,13\}$  that 
\begin{equation*}
x_l^*\geq\frac{\sigma(M/3)}{(\alpha-1)+\sigma(M/3)}=\frac{3-M}{(\alpha -2) M+3}\approx 0.2317.
\end{equation*}
 Then we see that $x_9^*>\Sigma_2^*$, a contradiction. This implies that $\Sigma_2^*>M/3$. We can repeat this argument with $\Sigma_3^*$ and $\Sigma_4^*$ to show that $\Sigma_3^*>M/3$ and $\Sigma_4^*>M/3$. This is a contradiction. So (A) is not the case. 


Assume that (B) holds. Since we have $\Sigma(\xs)\geq N$, we obtain the following inequality
\begin{equation}\label{ineq:3:2}
x_1^*+\Sigma_2^*+\Sigma_3^*+\Sigma_4^*\leq M=1-N\approx 0.6830.
\end{equation}
If $\Sigma_2^*\leq M/4\approx 0.1707$, then by the inequality $\sigma(M/4)\sigma(x_l^*)\leq\sigma(\Sigma_2^*)\sigma(x_l^*)\leq\alpha$ we find for every $l\in\{9,13\}$  that 
\begin{equation}\label{eqnL}
x_l^*\geq\frac{\sigma(M/4)}{\alpha+\sigma(M/4)}=\frac{4-M}{(\alpha-2)M+4}\approx 0.1691.
\end{equation}
Note that $x_9^*+x_{13}^*>\Sigma_2^*$, a contradiction. So we get $\Sigma_2^*>M/4$. Similar arguments for $\Sigma_3^*$ and $\Sigma_4^*$ show that  $\Sigma_3^*>M/4$ and $\Sigma_4^*>M/4$. Then we compute from (\ref{ineq:3:2}) that $x_1^*\leq M/4$. By (\ref{ineq:3:1}) we calculate that 
 \begin{equation}\label{Sig}
\Sigma(\xs)\geq L=\frac{3-2M}{4-M}\approx 0.4926.
\end{equation}
This implies $\Sigma(\xs)+\Sigma_2^*+\Sigma_3^*+\Sigma_4^*>L+3M/4\approx 1.0049>1$, a contradiction. Hence (B) is also not the case.

Assume that (C) in (\ref{ABC2}) holds. Since $x_1^*\geq N$, we have
\begin{equation}\label{ineq:3:3}
\Sigma(\xs)+\Sigma_2^*+\Sigma_3^*+\Sigma_4^*\leq M=1-N\approx 0.6830.
\end{equation}
If $\Sigma_2^*\leq M/4$, then by (\ref{eqnL}) we derive that $x_9^*+x_{13}^*>\Sigma_2^*$ as in case (B), a contradiction. So we must have $\Sigma_2^*>M/4$. Similar computations for $\Sigma_3^*$ and $\Sigma_4^*$ imply as in case (B) that 
$\Sigma_3^*>M/4$ and $\Sigma_4^*>M/4$. Then we find that $\Sigma(\xs)\leq M/4$. Since  $(2-\Sigma(\xs))x_1^*<2x_1^*$, using the inequality in (\ref{ineq:3:1}) we calculate that 
\begin{equation}\label{x1}
x_1^*\geq L=\frac{1}{4} \left(4-\sqrt{5+\sqrt{3}}\right)\approx 0.3513.
\end{equation}
Since $\Sigma_2^*+\Sigma_3^*+\Sigma_4^*>3M/4$, we find that $\Sigma_1^*<1-3M/4$. By \fullref{lemtwo}, using the inequality
$
\sigma(1-3M/4)\sigma(x_5^*)<\sigma(\Sigma_1^*)\sigma(x_5^*)=f_5(\xs)\leq\alpha
$
we compute that 
$$x_5^*>\frac{\sigma(1-3M/4)}{\alpha+\sigma(1-3M/4)}=\frac{3M}{(4-3M)\alpha+3M}\approx 0.0405.$$
We have $x_1^*+\Sigma_2^*+\Sigma_3^*+\Sigma_4^*<1$. By (\ref{x1}), we get $\Sigma_2^*+\Sigma_3^*+\Sigma_4^*<1-L\approx 0.6487$. By the inequality 
$
\sigma(1-L)\sigma(x_7^*)<\sigma(\Sigma_2^*+\Sigma_3^*+\Sigma_4^*)\sigma(x_7^*)=f_7(\xs)\leq\alpha
$
we derive that 
$$x_7^*>\frac{\sigma(1-L)}{\alpha+\sigma(1-L)}=\frac{L}{(1-L)\alpha+L}\approx 0.0213.$$
We claim that $\Sigma_2^*<\tfrac{1}{4}$. Because, otherwise, we calculate that 
\begin{equation}\label{eqn:3:5}
\begin{multlined}
x_1^*+x_5^*+x_7^*+\Sigma_2^*+\Sigma_3^*+\Sigma_4^*>\\
\shoveleft[3cm]{L+\frac{3M}{(4-3M)\alpha+3M}+\frac{L}{(1-L)\alpha+L}+\frac{M}{2}+\frac{1}{4}\approx 1.0047>1,}
\end{multlined}
\end{equation}
a contradiction. Similarly we find a contradiction in each case if we assume $\Sigma_3^*\geq\tfrac{1}{4}$ or $\Sigma_4^*\geq\tfrac{1}{4}$. Therefore we have $\Sigma_r^*<\tfrac{1}{4}$ for every $r\in\{2,3,4\}$. Then for every $l\in\{9,13,15,19,23,27\}$ we obtain 
$$
x_l^*>\frac{\sigma(1/4)}{\alpha+\sigma(1/4)}\approx 0.1076
$$
by the inequalities $\sigma(M/4)\sigma(x_l^*)\leq\sigma(\Sigma_r^*)\sigma(x_l^*)\leq\alpha$. Finally we get the contradiction
$$
x_1^*+x_5^*+x_7^*+x_9^*+x_{13}^*+x_{15}^*+x_{19}^*+x_{23}^*+x_{27}^*\approx 1.0591>1.
$$
This shows that (C) is not the case either, which completes the proof.
\end{proof}

\begin{lemma}\label{unique6}
Let $\xs\in\Delta^{27}$ so that $\alpha_{*}=F(\xs)$. Then  $\xs\in C_{f_l}$, defined in (\ref{C1}), for every $l\in\{5,9,13,15,19,23,27\}$.
\end{lemma}
\begin{proof}
The proof of \fullref{unique5} is symmetric in the sense that it can be repeated for every  index $l\in\{5,9,13,15,19,23,27\}$. In particular if $l=5$,  we interchange $x_1^*$ with $x_5^*$ and let $\Sigma(\tb{x})=\Sigma_{J}^1(\tb{x})-x_5$. Then we reiterate the computations carried out in the proof above by keeping the same organisations in (\ref{eqn:3:4}), (\ref{ineq:3:2}), (\ref{ineq:3:3}) and (\ref{eqn:3:5}). 

For some $l\in\{9,13,15,19,23,27\}$, we replace $x_1^*$ with $x_l^*$, let $\Sigma(\tb{x})=\Sigma_{J}^1(\tb{x})-x_l$ for some $i\in\{1,2,3,4\}$ and reorganise the inequalities in (\ref{eqn:3:4}), (\ref{ineq:3:2}), (\ref{ineq:3:3}) and (\ref{eqn:3:5}) by choosing relevant sums from $\Sigma_1^*$, $\Sigma_2^*$, $\Sigma_3^*$ and $\Sigma_4^*$. Then we carry out analogous calculations given in the proof of  \fullref{unique5} for the chosen index $l$.
\end{proof}

We shall also need the observation below about $g$, defined in (\ref{g}), in the computation of $\alpha_*$. Its proof is elementary. Therefore we shall omit it. We have 
\begin{lemma}\label{convex3}
For the points $(x,y)\in C_g$ the inequality $g(x,y)<\alpha=24.8692...$ holds  if and only if $0.1670...<y<\tfrac{1}{2}$ and $0<x<(-3+8y-4y^2)/(-4+4y)$ or 
\begin{equation*}
\begin{array}{c}
\dis 0.0134...=\frac{1+3\alpha-\sqrt{1-10\alpha+9\alpha^2}}{8\alpha} <y<\frac{1}{1-\alpha}+\sqrt{\frac{\alpha}{(\alpha-1)^2}}=0.1670...\\
\dis\tnr{and}\quad \frac{1-2y+(1-\alpha)y^2}{1+(\alpha-1)y}<x<\frac{-3+8y-4y^2}{-4+4y}.
\end{array}
\end{equation*}
\end{lemma}

As mentioned earlier, the displacement functions $\{f_l\}$ for $l\in J=\{1,5,9,13,15,19,23,27\}$ play a more important role in the computation of $\alpha_*$. These functions take larger values on $C=\bigcap_{l\in J}C_{f_l}$ than the values of the rest of the displacement functions in $\mathcal{F}$ at the points which are significant to calculate the infimum of the maximum of $F$. In other words  we have the following:

\begin{lemma}\label{unique7}
Let $\widetilde{F}(\tb{x})=\max_{\tb{x}\in C}\{f_l(\tb{x})\co l\in J\}$ for $C=\bigcap_{l\in J}C_{f_l}$.  Then $\widetilde{F}(\tb{x})\geq\alpha_*$.
\end{lemma}
\begin{proof}
Assume on the contrary that $\widetilde{F}(\tb{z})<\alpha_*$ for some $\tb{z}\in C$. Then by \fullref{lemtwo} for every $l\in J$  we have $f_l(\tb{z})<\alpha_*\leq\alpha=24.8692...$. Let $\tb{z}=(z_1,z_2,\dots,z_{28})$. 

Assume that $z_l>3/(3+\alpha)$ for every $l\in\{1,5\}$. Also assume that $z_l\leq 3/(3+\alpha)$ for every $l\in\{9,15,23\}$. By the inequalities $f_l(\tb{z})=\sigma(\Sigma_{J}^i(\tb{z}))\sigma(z_l)<\alpha$  for every $l\in\{9,15,23\}$, for every $i\in\{2,3,4\}$ we get 
\begin{equation}\label{lem:3:8}
\Sigma_J^i(\tb{z})> \frac{\sigma\left(\frac{3}{3+\alpha}\right)}{\alpha+\sigma\left(\frac{3}{3+\alpha}\right)}=\frac{1}{4}.
\end{equation} 
Since $\Sigma_J^1(\tb{z})+\Sigma_J^2(\tb{z})+\Sigma_J^3(\tb{z})+\Sigma_J^4(\tb{z})=1$, we have $\Sigma_J^1(\tb{z})<\tfrac{1}{4}$. This implies that
\begin{equation}\label{3:eqn}
\Sigma_J^1(\tb{z})-z_1<\frac{1}{4}-\frac{3}{3+\alpha}=0.1423....
\end{equation}
Because $\tb{z}\in C\subset C_{f_1}$, by \fullref{convex3} for $g=f_1$, $x=\Sigma_J^1-z_1$ and $y=z_1$, we find $z_1>0.4237...>\Sigma_J^1(\tb{z})$, a contradiction. 
So $z_l>3/(3+\alpha)$ for some $l\in\{9,15,23\}$.

Assume without loss of generality that $z_9>3/(3+\alpha)$ and $z_l\leq3/(3+\alpha)$ for every $l\in\{15,23\}$. Then we have $\Sigma_J^i(\tb{z})>\tfrac{1}{4}$ for every $i\in\{3,4\}$ by the inequalities $f_l(\tb{z})=\sigma(\Sigma_{J}^i(\tb{z}))\sigma(z_l)<\alpha$  for $l\in\{15,23\}$. This implies that $\Sigma_J^1(\tb{z})+\Sigma_J^2(\tb{z})<1/2$. If $\Sigma_J^1(\tb{z})<\tfrac{1}{4}$, then by the argument in the previous paragraph we obtain a contradiction. If $\Sigma_J^2(\tb{z})<\tfrac{1}{4}$, we have $\Sigma_J^2(\tb{z})-z_9<0.1423.... $ Using \fullref{convex3} for $g=f_9$, $x=\Sigma_J^2(\tb{z})-z_9$ and $y=z_9$, we find the contradiction $z_9>\Sigma_J^2(\tb{z})$. This implies that $z_l>3/(3+\alpha)$ for at least two distinct $l\in\{9,15,23\}$. 

Assume again without loss of generality that $z_l>3/(3+\alpha)$ for every $l\in\{9,15\}$ and $z_{23}\leq 3/(3+\alpha)$. Then $\Sigma_J^4(\tb{z})>\tfrac{1}{4}$ by the inequality $f_{23}(\tb{z})=\sigma(\Sigma_{J}^4(\tb{z}))\sigma(z_{23})<\alpha$. This implies that $\Sigma_J^1(\tb{z})+\Sigma_J^2(\tb{z})+\Sigma_J^3(\tb{z})<\tfrac{3}{4}$ which in turn gives that $\Sigma_J^i(\tb{z})<\tfrac{1}{4}$ for some $i\in\{1,2,3\}$. Since $z_l>3/(3+\alpha)$ for every $l\in\{1,5,9,15\}$, depending on $i$, using $z_1$ and $g=f_1$ or, $z_9$ and $g=f_9$ or,  $z_{15}$ and $g=f_{15}$ in (\ref{3:eqn}) and \fullref{convex3}, we obtain a contradiction in each case by repeating the arguments given above. So we must have $z_l>3/(3+\alpha)$ for every $l\in\{9,15,23\}$.

We already know that $\Sigma_J^1(\tb{z})+\Sigma_J^2(\tb{z})+\Sigma_J^3(\tb{z})+\Sigma_J^4(\tb{z})=1$ as $\tb{z}\in C\subset\Delta^{27}$. Then we get $\Sigma_J^i(\tb{z})\leq\tfrac{1}{4}$ for some $i\in\{1,2,3,4\}$. Given $i$, by choosing appropriate $z_l$ from the list $\{z_1,z_9,z_{15},z_{23}\}$ we repeat the relevant argument carried out above and derive a contradiction using \fullref{convex3}. As a result we conclude that $z_l\leq 3/(3+\alpha)$ for some $l\in\{1,5\}$. 

Notice that the computations used to show that  $z_l\leq 3/(3+\alpha)$ for some $l\in\{1,5\}$ are symmetric in the sense that they can be deployed to prove $z_l\leq 3/(3+\alpha)$ for some $l$ in any given pair $\{9,13\}$, $\{15,19\}$ and $\{23,27\}$. This implies that there exist entries $z_m$, $z_n$, $z_r$ and $z_s$ for $m\in\{1,5\}$
$n\in\{9,13\}$, $r\in\{15,19\}$ and $s\in\{23,27\}$ such that $z_l\leq 3/(3+\alpha)$ for every $l\in\{m,n,r,s\}$. By the inequalities $f_{l}(\tb{z})=\sigma(\Sigma_{J}^i(\tb{z}))\sigma(z_{l})<\alpha$ for $l\in\{m,n,r,s\}$, we find that $\Sigma_{J}^i(\tb{z})>\tfrac{1}{4}$ for every $i\in\{1,2,3,4\}$, a contradiction. Hence, the conclusion of the lemma follows. 
\end{proof}

Before we proceed to prove \fullref{prop:3:1} we review three facts from convex analysis. These facts were also used in \cite[Theorem 3.2, Theorem 3.3 and Proposition 3.3]{Y2}. For their proofs interested readers may refer to this source and the references therein:

\begin{theorem}\label{thm:3:8}
If $\{C_i\}$ for $i\in I$ is a collection of finitely many nonempty convex sets in $\mathbb{R}^d$ with $C=\displaystyle\cap_{i\in I} C_i\neq\emptyset$, then $C$ is also convex.
\end{theorem}
\begin{theorem}\label{thm:3:9}
If $\{f_i\}$ for $i\in I$ is a finite set of strictly convex functions defined on a convex set $C\subset\mathbb{R}^d$, then $\max_{\tb{x}\in C}\{f_i(\tb{x})\co i\in I\}$ is also a strictly convex function on $C$. 
\end{theorem}
\begin{proposition}\label{prop:3:0}
Let $F$ be a convex function on an open convex set $C\subset\mathbb{R}^d$. If $\xs$ is a local minimum of $F$, then it is a global minimum of $F$, and the set $\{\ys\in C\co F(\ys)=F(\xs)\}$ is a convex set. Furthermore, if $F$ is strictly convex and $\xs$ is a global minimum then the set $\{\ys\in C\co F(\ys)=F(\xs)\}$ consists of $\xs$ alone. 
\end{proposition}
With these facts we can prove the following statement which gives the first part of (\ref{b}):
\begin{proposition}\label{prop:3:1}
Let $\mathcal{F}=\{f_i\}$ for ${i\in I=\{1,2,\dots,28\}}$ be the set of displacement functions listed in  \fullref{dispfunc} and $F$ be as in (\ref{F}).  If $\xs$ and $\ys$ are two points in $\Delta^{27}$ so that $  \alpha_*=F(\xs)=F(\ys)$, then $\xs=\ys$. 
\end{proposition}
\begin{proof}
We know by \fullref{convex2} that each $f_l$ for $l\in J$ is a strictly convex function over the open convex set $C_{f_l}$. Therefore $\widetilde{F}(\tb{x})$ defined in \fullref{convex3} is also strictly convex on $C=\cap_{l\in J}C_{f_l}$ which is itself an open convex set by \fullref{thm:3:8} and \fullref{thm:3:9}. By \fullref{unique5} and \fullref{unique6} we have $\xs,\ \ys\in C$. Since $\widetilde{F}(\tb{x})\geq\alpha_*$
for every $\tb{x}\in C$ and $\widetilde{F}(\xs)=\alpha_*$ by \fullref{firstfive} and \fullref{unique7}, the value $\alpha_*$ is the global minimum of $\widetilde{F}$. As a result we find that $\xs=\ys$ by \fullref{prop:3:0}.
\end{proof}

The uniqueness of $\xs$ established by \fullref{prop:3:1} simplifies the task of determining the relations among the coordinates of $\xs$ considerably. In fact we have the following statement:

\begin{lemma}\label{lem:3:3}
If $\xs=(x_1^*,x_2^*,\dots,x_{28}^*)\in\Delta^{27}$ so that $F(\xs)=\alpha_*$, then $x_i^*=x_j^*$ for every $i,j\in\{1,5,9,13,15,19,23,27\}$. Also for every $i,j\in\{2,6,8,12,16,20,22,26\}$, $i,j\in\{3,4,10,11,17,18,24,25\}$ and $i,j\in\{7,14,21,28\}$ the equality  $x_i^*=x_j^*$ holds.
\end{lemma}
\begin{proof}
Consider the permutations $\tau_1$, $\tau_2$ and $\tau_3$ in the symmetric group $S_{28}$ defined below: 
\begin{equation*}
\begin{multlined}
\tau_1=(1\ \ 5)(2\ \ 6)(3\ \ 4)
(8\ \ 16)(9\ \ 15)(10\ \ 17)(11\ \ 18)
{(12\ \ 20)(13\ \ 19)(14\ \ 21)(22\ \ 26)(23\ \ 27)(24\ \ 25)
,}
\end{multlined}
\end{equation*}
\begin{equation*}
\begin{multlined}
\tau_2=(1\ \ 23)(2\ \ 22)(3\ \ 24)(4\ \ 25)(5\ \ 27)(6\ \ 26)(7\ \ 28)(8\ \ 12)
{(9\ \ 13)(10\ \ 11)
(15\ \ 19)(16\ \ 20)(17\ \ 18)
,}
\end{multlined}
\end{equation*}
\begin{equation*}
\begin{multlined}
\tau_3=(1\ \ 13)(2\ \ 12)(3\ \ 11)(4\ \ 10)(5\ \ 9)(6\ \ 8)(7\ \ 14)(15\ \ 27)\\
\shoveleft[8.5cm]{(16\ \ 26)(17\ \ 25)(18\ \ 24)(19\ \ 23)(20\ \ 22)(21\ \ 28).}
\end{multlined}
\end{equation*}
Let $T_l\co\Delta^{27}\to\Delta^{27}$ be the transformation defined by $x_i\mapsto x_{\tau_l(i)}$  for $l=1,2,3$. Note that $T_l(\Delta^{27})=\Delta^{27}$ for every $l$. Let $H_l\co\Delta^{27}\to\mathbb{R}$ be the map so that $H_l(\tb{x})=\max\{(f_i\circ T_l)(\tb{x})\co i=1,2,\dots,28\}$. Then we have $f_i(T_l(\tb{x}))=f_{\tau_l(i)}(\tb{x})$ for every $\tb{x}\in\Delta^{27}$ for every $i=1,2,\dots,28$ for every $l=1,2,3$. This implies that $F(\tb{x})=H_l(\tb{x})$ for every $\tb{x}$ and for every $l$. Since $\xs$ is unique by \fullref{prop:3:1}, we obtain $T_l^{-1}(\xs)=\xs$ for $l=1,2,3$. Then the lemma follows. 
\end{proof}

\fullref{lem:3:3} implies that $f_i(\xs)=f_j(\xs)$ for every $i,j\in\{1,5,9,13,15,19,23,27\}$. Also 
for every $i,j\in\{2,6,8,12,16,20,22,26\}$,
 $i,j\in\{3,4,10,11,17,18,24,25\}$ and 
$i,j\in\{7,14,21,28\}$ we have $f_i(\xs)=f_j(\xs)$. Therefore, there are four values to consider at $\xs$ to compute $\alpha_*$: $f_1(\xs)$, $f_2(\xs)$, $f_3(\xs)$ and $f_7(\xs)$ which are given as 
\begin{eqnarray}
\frac{1-2(x_1^*+x_2^*+x_3^*)-x_7^*}{2(x_1^*+x_2^*+x_3^*)+x_7^*}\cdot\frac{1-x_1^*}{x_1^*} & =& \alpha_*,\label{3:1}\\
\frac{1-7(x_1^*+x_2^*+x_3^*)-4x_7^*}{7(x_1^*+x_2^*+x_3^*)+4x_7^*}\cdot\frac{1-x_2^*}{x_2^*}& \leq & \alpha_*, \label{3:2}\\
\frac{1-7(x_1^*+x_2^*+x_3^*)-4x_7^*}{7(x_1^*+x_2^*+x_3^*)+4x_7^*}\cdot\frac{1-x_3^*}{x_3^*}& \leq& \alpha_*, \label{3:3}\\
\frac{1-6(x_1^*+x_2^*+x_3^*)-3x_7^*}{6(x_1^*+x_2^*+x_3^*)+3x_7^*}\cdot\frac{1-x_7^*}{x_7^*}& \leq& \alpha_*.\label{3:4}
\end{eqnarray}
We shall show next that $f_2(\xs)=f_3(\xs)=f_7(\xs)=\alpha_*$. To this purpose we will need the statement below:
\begin{lemma}\label{flipside}
For $1\leq k\leq n-1$, let $f_1$,\dots, $f_k$ be smooth functions on an open neighbourhood $U$ of the $(n-1)-$simplex $\Delta^{n-1}$ in $\mathbb{R}^n$. If at some $\tb{x}\in\Delta^{n-1}$ the collection $\{\nabla f_1(\tb{x}), \nabla f_2(\tb{x}),\dots,\nabla f_k(\tb{x}), \langle 1,\dots,1\rangle\}$ of vectors in $\mathbb{R}^n$ is linearly independent, then there exists a vector $\vec u\in T_{\tb{x}}\Delta^{n-1}$ such that each $f_i$ for $i=1,\dots,k$ decreases in the direction of $\vec u$ at $\tb{x}$.
\end{lemma}
Interested readers may refer to \cite[Lemma 4.10]{Y1} for its proof.  We have the following statement:
\begin{proposition}\label{prop:3:2}
Let $\mathcal{F}=\{f_i\}$ for ${i\in I=\{1,2,\dots,28\}}$ be the set of displacement functions listed in  \fullref{dispfunc} and $F$ be as in (\ref{F}).  If $\xs$ is the point such that $F(\xs)=\alpha_*$, then $\xs$ is in the set $\Delta_{27}=\{\tb{x}\in\Delta^{27}\co f_i(\tb{x})=f_j(\tb{x})\tnr{ for every } i,j\in I\}$. 
\end{proposition}
\begin{proof}
By \fullref{lem:3:3} it is enough to show that $f_2(\xs)=f_3(\xs)=f_7(\xs)=\alpha_*$. Remember that  $C_i^j$ denotes the partial derivative of $f_i$ with respect to $x_j$ at $\xs$. We calculate the constants below
\begin{equation*}
\begin{array}{lll}
\dis C_1^1=-\frac{\sigma(x_1^*)}{\left(\Sigma_{J}^1(\xs)\right)^2}-\frac{\sigma\left(\Sigma_{J}^1(\xs)\right)}{(x_1^*)^2}, & \dis C_1^2=-\frac{\sigma(x_1^*)}{ \left(\Sigma_{J}^1(\xs)\right)^2}, & C_5^1=-\dis\frac{\sigma(x_5^*)}
{\left(\Sigma_J^1(\xs)\right)^2},\\
 C_5^5=-\dis\frac{\sigma(x_5^*)}{\left(\Sigma_{J}^1(\xs)\right)^2}-\frac{\sigma\left(\Sigma_{J}^1(\xs)\right)}{(x_5^*)^2},  & C_9^8=-\dis\frac{\sigma(x_9^*)}{ \left(\Sigma_J^2(\xs)\right)^2}, & C_{13}^8=-\dis\frac{\sigma(x_{13}^*)}
{\left(\Sigma_J^2(\xs)\right)^2},
\end{array}
\end{equation*}
\begin{equation*}
\begin{array}{ll}
 C_9^9=-\dis\frac{\sigma(x_9^*)}{\left(\Sigma_{J}^2(\xs)\right)^2}-\frac{\sigma\left(\Sigma_{J}^2(\xs)\right)}{(x_9^*)^2}, &  C_{13}^{13}=-\dis\frac{\sigma(x_{13}^*)}{\left(\Sigma_{J}^2(\xs)\right)^2}-\frac{\sigma\left(\Sigma_{J}^2(\xs)\right)}{(x_{13}^*)^2},
\end{array}
\end{equation*}
\begin{equation*}
\begin{array}{lll}
C_{15}^{15}=-\dis\frac{\sigma(x_{15}^*)}{\left(\Sigma_{J}^3(\xs)\right)^2}-\frac{\sigma\left(\Sigma_{J}^3(\xs)\right)}{(x_{13}^*)^2},  & C_{15}^{16}=-\dis\frac{\sigma(x_{15}^*)}
{\left(\Sigma_J^3(\xs)\right)^2},& C_{19}^{15}=-\dis\frac{\sigma(x_{19}^*)}
{\left(\Sigma_J^3(\xs)\right)^2},
\end{array}
\end{equation*}
\begin{equation*}
\begin{array}{lll}
C_{19}^{19}=-\dis\frac{\sigma(x_{19}^*)}{\left(\Sigma_{J}^3(\xs)\right)^2}-\frac{\sigma\left(\Sigma_{J}^3(\xs)\right)}{(x_{19}^*)^2},  & C_{23}^{22}=-\dis\frac{\sigma(x_{23}^*)}
{\left(\Sigma_J^4(\xs)\right)^2}, & C_{27}^{22}=-\dis\frac{\sigma(x_{27}^*)}
{ \left(\Sigma_J^4(\xs)\right)^2}, 
\end{array}
\end{equation*}
\begin{equation*}
\begin{array}{ll}
C_{23}^{23}=-\dis\frac{\sigma(x_{23}^*)}{\left(\Sigma_{J}^4(\xs)\right)^2}-\frac{\sigma\left(\Sigma_{J}^4(\xs)\right)}{(x_{23}^*)^2}, & 
 C_{27}^{27}=-\dis\frac{\sigma(x_{27}^*)}{\left(\Sigma_{J}^4(\xs)\right)^2}-\frac{\sigma\left(\Sigma_{J}^4(\xs)\right)}{(x_{27}^*)^2}.
 \end{array}
\end{equation*}

\noindent Since we have $\Sigma_J^1(\xs)=\Sigma_J^2(\xs)=\Sigma_J^3(\xs)=\Sigma_J^4(\xs)$ by \fullref{lem:3:3}, we derive that 
\begin{equation*}
\begin{array}{c}
C_1^1=C_5^5=C_9^9=C_{13}^{13}=C_{15}^{15}=C_{19}^{19}=C_{23}^{23}=C_{27}^{27}, \\
 C_1^2=C_5^1=C_9^8=C_{13}^{8}=C_{15}^{16}=C_{19}^{15}=C_{23}^{22}=C_{27}^{22}.
\end{array}
\end{equation*}
We have  $\Sigma_I^4(\xs)=\Sigma_I^6(\xs)=\Sigma_I^7(\xs)=\Sigma_I^1(\xs)=\Sigma_I^8(\xs)=\Sigma_I^2(\xs)=\Sigma_I^3(\xs)=\Sigma_I^5(\xs)$ by \fullref{lem:3:3}. For the constants given in \fullref{firstfive} this implies that
\begin{equation*}
\begin{array}{c}
C_2^2=C_6^6=C_8^8=C_{12}^{12}=C_{16}^{16}=C_{20}^{20}=C_{22}^{22}=C_{26}^{26}, \\
C_2^1=C_6^1=C_8^1=C_{12}^{4}=C_{16}^{1}=C_{20}^{1}=C_{21}^{1}=C_{26}^{1},\\
C_3^3=C_4^4=C_{10}^{10}=C_{11}^{11}=C_{17}^{17}=C_{18}^{18}=C_{24}^{24}=C_{25}^{25}, \\
C_3^4=C_4^1=C_{10}^1=C_{11}^{1}=C_{17}^{1}=C_{18}^{1}=C_{24}^{1}=C_{25}^{1}.
\end{array}
\end{equation*}
Note that we get $\Sigma_1^J(\xs)=\Sigma_2^J(\xs)=\Sigma_3^J(\xs)=\Sigma_4^J(\xs)$ by  \fullref{lem:3:3}. As a result we also see that 
\begin{equation*}
\begin{array}{cc}
C_7^7=C_{14}^{14}=C_{21}^{21}=C_{28}^{28}, & C_{7}^{8}=C_{14}^{1}=C_{21}^{1}=C_{28}^{1}.
\end{array}
\end{equation*}
Consider the $28\times 28$ matrix below whose rows are $\nabla f_1(\xs)$, $\nabla f_2(\xs)$, \dots, $\nabla f_{28}(\xs)$ 
\begin{equation*}
\scalebox{0.77}{$
\left[
\begin{array}{cccccccccccccccccccccccccccc}
 C_1^1 & C_1^2 & C_1^2 & C_1^2 & C_1^2 & C_1^2 & C_1^2 & 0 & 0 & 0 & 0 & 0 & 0 & 0 & 0 & 0 & 0 & 0 & 0 & 0 & 0 & 0 & 0 & 0 & 0 & 0 & 0 & 0 \\
 C_2^1 & C_2^2 & C_2^1 & C_2^1 &C_2^1& C_2^1 & C_2^1 & C_2^1 & C_2^1 & C_2^1 & 0 & 0 & 0 & C_2^1 & C_2^1 & C_2^1 & C_2^1 & C_2^1 & C_2^1 & C_2^1 & C_2^1 & C_2^1 & C_2^1 & C_2^1 & C_2^1 & C_2^1 & C_2^1 & C_2^1 \\
 0 & 0 & C_3^3 & C_3^4 &C_3^4 & C_3^4 & C_3^4 & C_3^4 & C_3^4 & C_3^4 & C_3^4 & C_3^4 & C_3^4 & C_3^4 & C_3^4 & C_3^4 & C_3^4 & C_3^4 & C_3^4 & C_3^4 & C_3^4 & C_3^4 & C_3^4 & C_3^4 & C_3^4 & C_3^4 & C_3^4 & C_3^4 \\
C_3^4 &  C_3^4 &  C_3^4 & C_3^3 & 0 & 0 &  C_3^4 &  C_3^4 &  C_3^4 &  C_3^4 &  C_3^4 &  C_3^4 &  C_3^4 &  C_3^4 &  C_3^4 &  C_3^4 &  C_3^4 &  C_3^4 &  C_3^4 &  C_3^4 &  C_3^4 &  C_3^4 &  C_3^4 &  C_3^4 &  C_3^4 &  C_3^4 &  C_3^4 & C_3^4 \\
 C_1^2 & C_1^2 & C_1^2 & C_1^2 & C_1^1 & C_1^2 & C_1^2  & 0 & 0 & 0 & 0 & 0 & 0 & 0 & 0 & 0 & 0 & 0 & 0 & 0 & 0 & 0 & 0 & 0 & 0 & 0 & 0 & 0 \\
 C_2^1 &   C_2^1 &   C_2^1 &  C_2^1& C_2^1 &  C_2^2 & C_2^1 &  C_2^1 &  C_2^1 &  C_2^1 &  C_2^1 &  C_2^1 &  C_2^1 &  C_2^1 &  C_2^1 &  C_2^1 &  C_2^1& 0 & 0 & 0 & C_2^1 &  C_2^1 &  C_2^1 &  C_2^1 &  C_2^1 &  C_2^1 &  C_2^1 &  C_2^1 \\
 0 & 0 & 0 & 0 & 0 & 0 & C_7^7  & C_7^8 & C_7^8 & C_7^8 &C_7^8 &C_7^8 &C_7^8 &C_7^8 &C_7^8 &C_7^8 &C_7^8 &C_7^8 &C_7^8 &C_7^8 &C_7^8 &C_7^8 &C_7^8 &C_7^8 &C_7^8 &C_7^8 &C_7^8 &C_7^8  \\
 C_2^1 & C_2^1 & C_2^1 & C_2^1 & C_2^1 & C_2^1 & C_2^1 & C_2^2 & C_2^1 & C_2^1 & C_2^1 & C_2^1 & C_2^1 & C_2^1 & C_2^1 & C_2^1 & C_2^1 & C_2^1 & C_2^1 & C_2^1 & C_2^1 & 0 & 0 & 0 & C_2^1 & C_2^1 & C_2^1 & C_2^1 \\
 0 & 0 & 0 & 0 & 0 & 0 & 0 &  C_1^2 & C_1^1 & C_1^2 & C_1^2 & C_1^2 & C_1^2 & C_1^2 & 0 & 0 & 0 & 0 & 0 & 0 & 0 & 0 & 0 & 0 & 0 & 0 & 0 & 0 \\
C_{3}^4 & C_{3}^4 & C_{3}^4 & C_{3}^4 & C_{3}^4 & C_{3}^4 & C_{3}^4 & 0 & 0 & C_{3}^{3} & C_{3}^4 & C_{3}^4 & C_{3}^4 & C_{3}^4 & C_{3}^4 & C_{3}^4 & C_{3}^4 & C_{3}^4 & C_{3}^4 & C_{3}^4 & C_{3}^4 & C_{3}^4 & C_{3}^4 & C_{3}^4 & C_{3}^4 & C_{3}^4 & C_{3}^4 & C_{3}^4 \\
 C_{3}^4 & C_{3}^4 & C_{3}^4 & C_{3}^4 & C_{3}^4 & C_{3}^4 & C_{3}^4 & C_{3}^4 & C_{3}^4 & C_{3}^4  & C_{3}^{3} & 0 & 0 & C_{3}^4 & C_{3}^4 & C_{3}^4 & C_{3}^4 & C_{3}^4 & C_{3}^4 & C_{3}^4 & C_{3}^4 & C_{3}^4 & C_{3}^4 & C_{3}^4 & C_{3}^4 & C_{3}^4 & C_{3}^4 & C_{3}^4 \\
 0 & 0 & 0 & C_{2}^1 & C_{2}^1 & C_{2}^1 & C_{2}^1 & C_{2}^1 & C_{2}^1 & C_{2}^1 & C_{2}^1 & C_{2}^{2 }& C_{2}^1 & C_{2}^1 & C_{2}^1 & C_{2}^1 & C_{2}^1 & C_{2}^1 & C_{2}^1 & C_{2}^1 & C_{2}^1 & C_{2}^1 & C_{2}^1 & C_{2}^1 & C_{2}^1 & C_{2}^1 & C_{2}^1 & C_{2}^1 \\
 0 & 0 & 0 & 0 & 0 & 0 & 0 &  C_1^2 & C_1^2 & C_1^2 & C_1^2 & C_1^2 & C_1^1 & C_1^2 & 0 & 0 & 0 & 0 & 0 & 0 & 0 & 0 & 0 & 0 & 0 & 0 & 0 & 0 \\
 C_{7}^8 & C_{7}^8 & C_{7}^8 & C_{7}^8 & C_{7}^8 & C_{7}^8 & C_{7}^8 & 0 & 0 & 0 & 0 & 0 & 0 & C_{7}^{7}  & C_{7}^8 & C_{7}^8 & C_{7}^8 & C_{7}^8 & C_{7}^8 & C_{7}^8 & C_{7}^8 & C_{7}^8 & C_{7}^8 & C_{7}^8 & C_{7}^8 & C_{7}^8 & C_{7}^8 & C_{7}^8 \\
 0 & 0 & 0 & 0 & 0 & 0 & 0 & 0 & 0 & 0 & 0 & 0 & 0 & 0 &  C_1^1 & C_1^2 & C_1^2 & C_1^2 & C_1^2 & C_1^2 & C_1^2 & 0 & 0 & 0 & 0 & 0 & 0 & 0 \\
 C_{2}^1 &  C_{2}^1 &  C_{2}^1 &  C_{2}^1 &  C_{2}^1 &  C_{2}^1 &  C_{2}^1 &  C_{2}^1 &  C_{2}^1 &  C_{2}^1 &  C_{2}^1 &  C_{2}^1 &  C_{2}^1 &  C_{2}^1 &  C_{2}^1 &  C_{2}^{2} &  C_{2}^1 &  C_{2}^1 &  C_{2}^1 &  C_{2}^1 &  C_{2}^1 &  C_{2}^1 &  C_{2}^1 &  C_{2}^1 & 0 & 0 & 0 &  C_{2}^1 \\
C_{3}^4& C_{3}^4 & C_{3}^4 & C_{3}^4 & C_{3}^4 & C_{3}^4 & C_{3}^4 & C_{3}^4 & C_{3}^4 & C_{3}^4 & C_{3}^4 & C_{3}^4 & C_{3}^4 & C_{3}^4 & 0 & 0 & C_{3}^{3} & C_{3}^4 & C_{3}^4 & C_{3}^4 & C_{3}^4 & C_{3}^4 & C_{3}^4 & C_{3}^4 & C_{3}^4 & C_{3}^4 & C_{3}^4 & C_{3}^4 \\
 C_{3}^4 &  C_{3}^4 &  C_{3}^4 &  C_{3}^4 &  C_{3}^4 &  C_{3}^4 &  C_{3}^4 &  C_{3}^4 &  C_{3}^4 &  C_{3}^4 &  C_{3}^4 &  C_{3}^4 &  C_{3}^4 &  C_{3}^4 &  C_{3}^4 &  C_{3}^4 &  C_{3}^4 &  C_{3}^3 & 0 & 0 &  C_{3}^4 &  C_{3}^4 &  C_{3}^4 &  C_{3}^4 &  C_{3}^4 &  C_{3}^4 &  C_{3}^4 & C_{3}^4 \\
 0 & 0 & 0 & 0 & 0 & 0 & 0 & 0 & 0 & 0 & 0 & 0 & 0 & 0 &  C_1^2 & C_1^2 & C_1^2 & C_1^2 & C_1^1 & C_1^2 & C_1^2  & 0 & 0 & 0 & 0 & 0 & 0 & 0 \\
 C_{2}^{1}  & C_{2}^{1} & C_{2}^{1} & 0 & 0 & 0 & C_{2}^{1} & C_{2}^{1} & C_{2}^{1} & C_{2}^{1} & C_{2}^{1} & C_{2}^{1} & C_{2}^{1} & C_{2}^{1} & C_{2}^{1} & C_{2}^{1} & C_{2}^{1} & C_{2}^{1} & C_{2}^{1} & C_{2}^{2} & C_{2}^{1} & C_{2}^{1} & C_{2}^{1} & C_{2}^{1} & C_{2}^{1} & C_{2}^{1} & C_{2}^{1} & C_{2}^{1} \\
 C_{7}^8 & C_{7}^8 & C_{7}^8 & C_{7}^8 & C_{7}^8 & C_{7}^8 & C_{7}^8 & C_{7}^8 & C_{7}^8 & C_{7}^8 & C_{7}^8 & C_{7}^8 & C_{7}^8 & C_{7}^8 & 0 & 0 & 0 & 0 & 0 & 0 & C_{7}^{7}  & C_{7}^8 & C_{7}^8 & C_{7}^8 & C_{7}^8 & C_{7}^8& C_{7}^8 &C_{7}^8 \\
 C_{2}^{1} & C_{2}^{1} & C_{2}^{1} & C_{2}^{1} & C_{2}^{1} & C_{2}^{1} & C_{2}^{1} & 0 & 0 & 0 & C_{2}^{1} & C_{2}^{1} & C_{2}^{1} & C_{2}^{1} & C_{2}^{1} & C_{2}^{1} & C_{2}^{1} & C_{2}^{1} & C_{2}^{1} & C_{2}^{1} & C_{2}^{1} & C_{2}^{2} & C_{2}^{1} & C_{2}^{1} & C_{2}^{1} & C_{2}^{1} & C_{2}^{1} & C_{2}^{1} \\
 0 & 0 & 0 & 0 & 0 & 0 & 0 & 0 & 0 & 0 & 0 & 0 & 0 & 0 & 0 & 0 & 0 & 0 & 0 & 0 & 0 &  C_1^2 & C_1^1 & C_1^2 & C_1^2 & C_1^2 & C_1^2 & C_1^2 \\
 C_{3}^{4} & C_{3}^{4} & C_{3}^{4} & C_{3}^{4} & C_{3}^{4} & C_{3}^{4} & C_{3}^{4} & C_{3}^{4} & C_{3}^{4} & C_{3}^{4} & C_{3}^{4} & C_{3}^{4} & C_{3}^{4} & C_{3}^{4} & C_{3}^{4} & C_{3}^{4} & C_{3}^{4} & C_{3}^{4} & C_{3}^{4} & C_{3}^{4} & C_{3}^{4} & 0 & 0 & C_{3}^{3} & C_{3}^{4} & C_{3}^{4} & C_{3}^{4} & C_{3}^{4} \\
C_{3}^{4} & C_{3}^{4} & C_{3}^{4} & C_{3}^{4} & C_{3}^{4} & C_{3}^{4} & C_{3}^{4} & C_{3}^{4} & C_{3}^{4} & C_{3}^{4} & C_{3}^{4} & C_{3}^{4} & C_{3}^{4} & C_{3}^{4} & C_{3}^{4} & C_{3}^{4} & C_{3}^{4} & C_{3}^{4} & C_{3}^{4} & C_{3}^{4} & C_{3}^{4} & C_{3}^{4} & C_{3}^{4} & C_{3}^{4} & C_{3}^{3} & 0 & 0 & C_{3}^{4} \\
C_{2}^{1} & C_{2}^{1} & C_{2}^{1} & C_{2}^{1} & C_{2}^{1} & C_{2}^{1} & C_{2}^{1} & C_{2}^{1} & C_{2}^{1} & C_{2}^{1} & C_{2}^{1} & C_{2}^{1} & C_{2}^{1} & C_{2}^{1}  & 0 & 0 & 0 & C_{2}^{1} & C_{2}^{1} & C_{2}^{1} & C_{2}^{1} & C_{2}^{1} & C_{2}^{1} & C_{2}^{1} & C_{2}^{1} & C_{2}^{2} & C_{2}^{1} & C_{2}^{1} \\
 0 & 0 & 0 & 0 & 0 & 0 & 0 & 0 & 0 & 0 & 0 & 0 & 0 & 0 & 0 & 0 & 0 & 0 & 0 & 0 & 0 &  C_1^2 & C_1^2 & C_1^2 & C_1^2 & C_1^2 & C_1^1 & C_1^2  \\
C_{7}^8 & C_{7}^8 & C_{7}^8 & C_{7}^8 & C_{7}^8 & C_{7}^8 & C_{7}^8 & C_{7}^8 & C_{7}^8 & C_{7}^8 & C_{7}^8 & C_{7}^8 & C_{7}^8 & C_{7}^8 & C_{7}^8 & C_{7}^8 & C_{7}^8 & C_{7}^8 & C_{7}^8 & C_{7}^8 & C_{7}^8 & 0 & 0 & 0 & 0 & 0 & 0 & C_{7}^{7}   
\end{array}\right].$}
\end{equation*}
Assume that $f_2(\xs)<\alpha_*$. Consider the vector $\vec v_1\in T_{\xs}\Delta^{27}$ with the coordinates:
\begin{equation*}
 (\vec v_1)_i=\left\{\begin{array}{rl}
 1   & \tnr{if $i=1,3,4,5,7,9,10,11,13,14,15,17,18,19,21,23,24,25,27,28$,}\\
-2  & \tnr{if $i=6,12,20,26$,}\\
-3  & \tnr{if $i=2,8,16,22$.}\end{array}\right.
\end{equation*}
For any given  indices $l\in J=\{1,5,9,13,15,19,23,27\}$, $i\in K=\{3,10,17,24\}$, $j\in L=\{4,11,18,25\}$ and $k\in N=\{7,14,21,28\}$ we calculate that 
\begin{equation*}
\nabla f_l(\xs)\cdot \vec v_1=C_1^1-C_1^2=-\frac{\sigma\left(\Sigma_{J}^1(\xs)\right)}{(x_1^*)^2}<0, 
\end{equation*}
\begin{equation*}
\nabla f_i(\xs)\cdot \vec v_1=C_3^3+C_3^4=-\frac{\sigma\left(\Sigma_{I}^1(\xs)\right)}{(x_3^*)^2}-\frac{(x_3^*)^2}{\sigma\left(\Sigma_{I}^1(\xs)\right)}<0, 
\end{equation*}
\begin{equation*}
\nabla f_j(\xs)\cdot \vec v_1=C_3^3=-\frac{\sigma\left(\Sigma_{I}^1(\xs)\right)}{(x_3^*)^2}<0,\quad \nabla f_k(\xs)\cdot \vec v_1=C_7^7=-\frac{\sigma\left(\Sigma_{J}^1(\xs)\right)}{(x_7^*)^2}<0. 
\end{equation*}
This implies that values of $f_l$ for $l\in J\cup K\cup L\cup N$ decrease along a line segment in the direction of $\vec v_1$. For a short distance along $\vec v_1$  values of $f_l$ for $l\in\{2,6,8,12,16,20,22,26\}$ are smaller than $\alpha_*$. So there exists a point $\tb{z}\in\Delta^{27}$ such that  $f_l(\tb{z})<\alpha_*$ for every $l\in I=\{1,2,\dots,28\}$. This is a contradiction. Hence we find that $f_2(\xs)=\alpha_*$. 

Assume that $f_3(\xs)<\alpha_*$. We introduce the vector $\vec v_2\in T_{\xs}\Delta^{27}$ with the coordinates 
\begin{equation*}
 (\vec v_2)_i=\left\{\begin{array}{rl}
 1   & \tnr{if $i=1,2,5,6,7,8,9,12,13,14,15,16,19,20,21,22,23,26,27,28$,}\\
-2  & \tnr{if $i=4,11,18,25$,}\\
-3  & \tnr{if $i=3,10,17,24$.}\end{array}\right.
\end{equation*}
For $l\in J$, $i\in K'=\{2,6,16,20\}$, $j\in L'=\{8,12,22,26\}$ and $k\in N$  we calculate 
\begin{equation*}
\nabla f_l(\xs)\cdot \vec v_2=C_1^1-C_1^2=-\frac{\sigma\left(\Sigma_{J}^1(\xs)\right)}{(x_1^*)^2}<0, \ \  \nabla f_k(\xs)\cdot \vec v_1=C_7^7=-\frac{\sigma\left(\Sigma_{J}^1(\xs)\right)}{(x_7^*)^2}<0,
\end{equation*}
\begin{equation*}
\nabla f_i(\xs)\cdot \vec v_2=C_2^2-C_2^1=-\frac{\sigma\left(\Sigma_{I}^4(\xs)\right)}{(x_2^*)^2}<0, 
\end{equation*}
which show that  values of $f_l$ for $l\in J\cup K'\cup L'\cup N$ decrease along a line segment in the direction of $\vec v_2$. Values of $f_l$ for $l\in\{2,6,8,12,16,20,22,26\}$ are smaller than $\alpha_*$ for a short distance along $\vec v_2$. As a result there exists a point $\tb{w}\in\Delta^{27}$ such that  $f_l(\tb{w})<\alpha_*$ for every $l\in I=\{1,2,\dots,28\}$, a contradiction. We derive that $f_3(\xs)=\alpha_*$. 

Since we have $f_2(\xs)=f_3(\xs)=\alpha_*$, we obtain  $x_2^*=x_3^*$. Then we see that $C_2^1=C_3^4$. 
Also we find that $C_3^3=C_2^2-C_2^1$. Now assume that $f_7(\xs)<\alpha_*$. Then we construct the following $25\times 28$ matrix $\mathcal{M}$:  
\begin{equation*}
\scalebox{0.58}{$
\left[
\begin{array}{cccccccccccccccccccccccccccc}
 C_1^1 & C_1^2 & C_1^2 & C_1^2 & C_1^2 & C_1^2 & C_1^2 & 0 & 0 & 0 & 0 & 0 & 0 & 0 & 0 & 0 & 0 & 0 & 0 & 0 & 0 & 0 & 0 & 0 & 0 & 0 & 0 & 0 \\ 
 C_2^1 & C_2^2 & C_2^1 & C_2^1 &C_2^1& C_2^1 & C_2^1 & C_2^1 & C_2^1 & C_2^1 & 0 & 0 & 0 & C_2^1 & C_2^1 & C_2^1 & C_2^1 & C_2^1 & C_2^1 & C_2^1 & C_2^1 & C_2^1 & C_2^1 & C_2^1 & C_2^1 & C_2^1 & C_2^1 & C_2^1 \\ 
 0 & 0 & C_2^2-C_2^1 & C_2^1 & C_2^1 & C_2^1 & C_2^1 & C_2^1 & C_2^1 & C_2^1 & C_2^1 & C_2^1 & C_2^1 & C_2^1 & C_2^1 & C_2^1 & C_2^1 & C_2^1 & C_2^1 & C_2^1 & C_2^1 & C_2^1 & C_2^1 & C_2^1 & C_2^1 & C_2^1 & C_2^1 & C_2^1 \\ 
C_2^1 &  C_2^1 &  C_2^1 & C_2^2-C_2^1 & 0 & 0 &  C_2^1 &  C_2^1 &  C_2^1 &  C_2^1 &  C_2^1 &  C_2^1 &  C_2^1 &  C_2^1 &  C_2^1 &  C_2^1 &  C_2^1 &  C_2^1 &  C_2^1 &  C_2^1 &  C_2^1 &  C_2^1 &  C_2^1 &  C_2^1 &  C_2^1 &  C_2^1 &  C_2^1 & C_2^1 \\ 
 C_1^2 & C_1^2 & C_1^2 & C_1^2 & C_1^1 & C_1^2 & C_1^2  & 0 & 0 & 0 & 0 & 0 & 0 & 0 & 0 & 0 & 0 & 0 & 0 & 0 & 0 & 0 & 0 & 0 & 0 & 0 & 0 & 0 \\ 
 C_2^1 &   C_2^1 &   C_2^1 &  C_2^1& C_2^1 &  C_2^2 & C_2^1 &  C_2^1 &  C_2^1 &  C_2^1 &  C_2^1 &  C_2^1 &  C_2^1 &  C_2^1 &  C_2^1 &  C_2^1 &  C_2^1& 0 & 0 & 0 & C_2^1 &  C_2^1 &  C_2^1 &  C_2^1 &  C_2^1 &  C_2^1 &  C_2^1 &  C_2^1 \\ 
 C_2^1 & C_2^1 & C_2^1 & C_2^1 & C_2^1 & C_2^1 & C_2^1 & C_2^2 & C_2^1 & C_2^1 & C_2^1 & C_2^1 & C_2^1 & C_2^1 & C_2^1 & C_2^1 & C_2^1 & C_2^1 & C_2^1 & C_2^1 & C_2^1 & 0 & 0 & 0 & C_2^1 & C_2^1 & C_2^1 & C_2^1 \\ 
 0 & 0 & 0 & 0 & 0 & 0 & 0 &  C_1^2 & C_1^1 & C_1^2 & C_1^2 & C_1^2 & C_1^2 & C_1^2 & 0 & 0 & 0 & 0 & 0 & 0 & 0 & 0 & 0 & 0 & 0 & 0 & 0 & 0 \\ 
C_{2}^1 & C_{2}^1 & C_{2}^1 & C_{2}^1 & C_{2}^1 & C_{2}^1 & C_{2}^1 & 0 & 0 & C_{2}^{2}-C_2^1 & C_{2}^1 & C_{2}^1 & C_{2}^1 & C_{2}^1 & C_{2}^1 & C_{2}^1 & C_{2}^1 & C_{2}^1 & C_{2}^1 & C_{2}^1 & C_{2}^1 & C_{2}^1 & C_{2}^1 & C_{2}^1 & C_{2}^1 & C_{2}^1 & C_{2}^1 & C_{2}^1 \\ 
 C_{2}^1 & C_{2}^1 & C_{2}^1 & C_{2}^1 & C_{2}^1 & C_{2}^1 & C_{2}^1 & C_{2}^1 & C_{2}^1 & C_{2}^1  & C_{2}^{2}-C_2^1 & 0 & 0 & C_{2}^1 & C_{2}^1 & C_{2}^1 & C_{2}^1 & C_{2}^1 & C_{2}^1 & C_{2}^1 & C_{2}^1 & C_{2}^1 & C_{2}^1 & C_{2}^1& C_{2}^1 & C_{2}^1 & C_{2}^1 & C_{2}^1 \\ 
 0 & 0 & 0 & C_{2}^1 & C_{2}^1 & C_{2}^1 & C_{2}^1 & C_{2}^1 & C_{2}^1 & C_{2}^1 & C_{2}^1 & C_{2}^{2 }& C_{2}^1 & C_{2}^1 & C_{2}^1 & C_{2}^1 & C_{2}^1 & C_{2}^1 & C_{2}^1 & C_{2}^1 & C_{2}^1 & C_{2}^1 & C_{2}^1 & C_{2}^1 & C_{2}^1 & C_{2}^1 & C_{2}^1 & C_{2}^1 \\ 
 0 & 0 & 0 & 0 & 0 & 0 & 0 &  C_1^2 & C_1^2 & C_1^2 & C_1^2 & C_1^2 & C_1^1 & C_1^2 & 0 & 0 & 0 & 0 & 0 & 0 & 0 & 0 & 0 & 0 & 0 & 0 & 0 & 0 \\ 
 0 & 0 & 0 & 0 & 0 & 0 & 0 & 0 & 0 & 0 & 0 & 0 & 0 & 0 &  C_1^1 & C_1^2 & C_1^2 & C_1^2 & C_1^2 & C_1^2 & C_1^2 & 0 & 0 & 0 & 0 & 0 & 0 & 0 \\ 
 C_{2}^1 &  C_{2}^1 &  C_{2}^1 &  C_{2}^1 &  C_{2}^1 &  C_{2}^1 &  C_{2}^1 &  C_{2}^1 &  C_{2}^1 &  C_{2}^1 &  C_{2}^1 &  C_{2}^1 &  C_{2}^1 &  C_{2}^1 &  C_{2}^1 &  C_{2}^{2} &  C_{2}^1 &  C_{2}^1 &  C_{2}^1 &  C_{2}^1 &  C_{2}^1 &  C_{2}^1 &  C_{2}^1 &  C_{2}^1 & 0 & 0 & 0 &  C_{2}^1 \\ 
C_{2}^1 & C_{2}^1 & C_{2}^1 & C_{2}^1 & C_{2}^1 & C_{2}^1 & C_{2}^1 & C_{2}^1 & C_{2}^1 & C_{2}^1 & C_{2}^1 & C_{2}^1 & C_{2}^1 & C_{2}^1 & 0 & 0 & C_{2}^{2}-C_2^1 & C_{2}^1 & C_{2}^1 & C_{2}^1 & C_{2}^1 & C_{2}^1 & C_{2}^1 & C_{2}^1 & C_{2}^1 & C_{2}^1 & C_{2}^1 & C_{2}^1 \\ 
 C_{2}^1 &  C_{2}^1 &  C_{2}^1 &  C_{2}^1 &  C_{2}^1 &  C_{2}^1 &  C_{2}^1 &  C_{2}^1 &  C_{2}^1 &  C_{2}^1 &  C_{2}^1 &  C_{2}^1 &  C_{2}^1 &  C_{2}^1 &  C_{2}^1 &  C_{2}^1 &  C_{2}^1 &  C_{2}^2-C_2^1 & 0 & 0 &  C_{2}^1 &  C_{2}^1 &  C_{2}^1 &  C_{2}^1 &  C_{2}^1 &  C_{2}^1 &  C_{2}^1 & C_{2}^1 \\ 
 0 & 0 & 0 & 0 & 0 & 0 & 0 & 0 & 0 & 0 & 0 & 0 & 0 & 0 &  C_1^2 & C_1^2 & C_1^2 & C_1^2 & C_1^1 & C_1^2 & C_1^2  & 0 & 0 & 0 & 0 & 0 & 0 & 0 \\ 
 C_{2}^{1}  & C_{2}^{1} & C_{2}^{1} & 0 & 0 & 0 & C_{2}^{1} & C_{2}^{1} & C_{2}^{1} & C_{2}^{1} & C_{2}^{1} & C_{2}^{1} & C_{2}^{1} & C_{2}^{1} & C_{2}^{1} & C_{2}^{1} & C_{2}^{1} & C_{2}^{1} & C_{2}^{1} & C_{2}^{2} & C_{2}^{1} & C_{2}^{1} & C_{2}^{1} & C_{2}^{1} & C_{2}^{1} & C_{2}^{1} & C_{2}^{1} & C_{2}^{1} \\ 
 C_{2}^{1} & C_{2}^{1} & C_{2}^{1} & C_{2}^{1} & C_{2}^{1} & C_{2}^{1} & C_{2}^{1} & 0 & 0 & 0 & C_{2}^{1} & C_{2}^{1} & C_{2}^{1} & C_{2}^{1} & C_{2}^{1} & C_{2}^{1} & C_{2}^{1} & C_{2}^{1} & C_{2}^{1} & C_{2}^{1} & C_{2}^{1} & C_{2}^{2} & C_{2}^{1} & C_{2}^{1} & C_{2}^{1} & C_{2}^{1} & C_{2}^{1} & C_{2}^{1} \\ 
 0 & 0 & 0 & 0 & 0 & 0 & 0 & 0 & 0 & 0 & 0 & 0 & 0 & 0 & 0 & 0 & 0 & 0 & 0 & 0 & 0 &  C_1^2 & C_1^1 & C_1^2 & C_1^2 & C_1^2 & C_1^2 & C_1^2 \\ 
 C_{2}^{1} & C_{2}^{1} & C_{2}^{1} & C_{2}^{1} & C_{2}^{1} & C_{2}^{1} & C_{2}^{1} & C_{2}^{1} & C_{2}^{1} & C_{2}^{1} & C_{2}^{1} & C_{2}^{1} & C_{2}^{1} & C_{2}^{1} & C_{2}^{1} & C_{2}^{1} & C_{2}^{1} & C_{2}^{1} & C_{2}^{1} & C_{2}^{1} & C_{2}^{1} & 0 & 0 & C_{2}^{2}-C_2^1 & C_{2}^{1} & C_{2}^{1} & C_{2}^{1} & C_{2}^{1} \\ 
C_{2}^{1} & C_{2}^{1} & C_{2}^{1} & C_{2}^{1} & C_{2}^{1} & C_{2}^{1} & C_{2}^{1} & C_{2}^{1} & C_{2}^{1} & C_{2}^{1} & C_{2}^{1} & C_{2}^{1} & C_{2}^{1} & C_{2}^{1} & C_{2}^{1} & C_{2}^{1} & C_{2}^{1} & C_{2}^{1} & C_{2}^{1} & C_{2}^{1} & C_{2}^{1} & C_{2}^{1} & C_{2}^{1} & C_{2}^{1} & C_{2}^{2}-C_2^1 & 0 & 0 & C_{2}^{1} \\ 
C_{2}^{1} & C_{2}^{1} & C_{2}^{1} & C_{2}^{1} & C_{2}^{1} & C_{2}^{1} & C_{2}^{1} & C_{2}^{1} & C_{2}^{1} & C_{2}^{1} & C_{2}^{1} & C_{2}^{1} & C_{2}^{1} & C_{2}^{1}  & 0 & 0 & 0 & C_{2}^{1} & C_{2}^{1} & C_{2}^{1} & C_{2}^{1} & C_{2}^{1} & C_{2}^{1} & C_{2}^{1} & C_{2}^{1} & C_{2}^{2} & C_{2}^{1} & C_{2}^{1} \\ 
 0 & 0 & 0 & 0 & 0 & 0 & 0 & 0 & 0 & 0 & 0 & 0 & 0 & 0 & 0 & 0 & 0 & 0 & 0 & 0 & 0 &  C_1^2 & C_1^2 & C_1^2 & C_1^2 & C_1^2 & C_1^1 & C_1^2 \\ 
  1 & 1 & 1 & 1 & 1 & 1 & 1 & 1 & 1 & 1 & 1 & 1 & 1 & 1 & 1 & 1 & 1 & 1 & 1 & 1 & 1 & 1 & 1 & 1 & 1 & 1 & 1 & 1 
\end{array}\right].$}
\end{equation*}

\noindent Let $R_l$ denote the $l$th row of $\mathcal{M}$ for $l\in\{1,2,\dots,25\}$. Applying from left to right and row by row, we perform on $\mathcal{M}$ the row reduction operations listed in the 
\fullref{table:2:12} simultaneously:

\begin{table}[H]
\begin{center}
\scalebox{.85}{
\begin{tabular}{|c|c|c|c|c|}
\hline
$-C_2^1R_{25}+R_{23}\rightarrow R_{23} $                 & $ -C_2^1R_{25}+R_{22}\rightarrow R_{22} $               & $-C_2^1R_{25}+R_{21}\rightarrow R_{21} $                  & $-C_2^1R_{25}+R_{19}\rightarrow R_{19} $  &   $-C_2^1R_{25}+R_{18}\rightarrow R_{18} $                                  \\ \hline %
$-C_2^1R_{25}+R_{16}\rightarrow R_{16} $               & $-C_2^1R_{25}+R_{15}\rightarrow R_{15} $                  &  $-C_2^1R_{25}+R_{14}\rightarrow R_{14} $  &  $-C_2^1R_{25}+R_{10}\rightarrow R_{10} $      & $-C_2^1R_{25}+R_{9}\rightarrow R_{9} $                                  \\ \hline %
$-C_2^1R_{25}+R_{7}\rightarrow R_{7} $                     &  $-C_2^1R_{25}+R_{6}\rightarrow R_{6} $    &	$-C_2^1R_{25}+R_{4}\rightarrow R_{4} $  &	 $-C_2^1R_{25}+R_{2}\rightarrow R_{2} $		&	$-C_1^1R_{25}+R_{1}\rightarrow R_{1} $\\ \hline  %
$-C_1^2R_{25}+R_{5}\rightarrow R_{5} $  &	$R_{18}+R_{11}\rightarrow R_{11} $                             &  $R_{19}+R_{11}\rightarrow R_{11} $                           &  $R_{18}+R_{3}\rightarrow R_{3} $                                 &  $R_{19}+R_{3}\rightarrow R_{3} $ 				 \\ \hline %
$-R_{11}+R_{3}\rightarrow R_{3} $                                &  $-2R_{18}+R_{4}\rightarrow R_{4} $                            & $-R_{19}+R_{7}\rightarrow R_{7} $                                & $R_{18}+R_{4}\rightarrow R_{4} $ 	&	$\dis\frac{1}{C_2^2-C_2^1}R_{3}\rightarrow R_{3} $					\\ \hline %
 $\dis\frac{1}{C_2^2-C_2^1}R_{4}\rightarrow R_{4} $    & $-R_{19}+R_{9}\rightarrow R_{9} $                                &  $R_{19}+R_{7}\rightarrow R_{7} $ 	&	$-R_{21}+R_{7}\rightarrow R_{7} $  	& $\dis\frac{1}{C_2^2-C_2^1}R_{7}\rightarrow R_{7} $				\\ \hline   %
 $R_{12}+R_{1}\rightarrow R_{1} $                 		       & $R_{13}+R_{1}\rightarrow R_{1} $	&  $R_{20}+R_{1}\rightarrow R_{1} $	 &    $\dis\frac{1}{C_1^2-C_1^1}R_{1}\rightarrow R_{1} $	         & 	$R_8+R_5\rightarrow R_{5} $				 \\ \hline
$R_{13}+R_{5}\rightarrow R_{5} $ & $R_{20}+R_5\rightarrow R_{5} $                                   &  $\dis\frac{1}{C_1^1-C_1^2}R_{5}\rightarrow R_{5} $   & $-R_{10}+R_2\rightarrow R_{2} $                                   & $\dis\frac{1}{C_2^2-C_2^1}R_{2}\rightarrow R_{2} $ 			\\ \hline
$-R_{16}+R_6\rightarrow R_{6} $                                  &  $\dis\frac{1}{C_2^2-C_2^1}R_{6}\rightarrow R_{6} $   & $-R_{12}+R_8\rightarrow R_{8} $                                   & $\dis\frac{1}{C_1^1-C_1^2}R_{8}\rightarrow R_{8} $ &   $\dis\frac{1}{C_2^2-C_2^1}R_{9}\rightarrow R_{9}$			\\ \hline
$-R_{17}+R_{13}\rightarrow R_{13} $                           & $\dis\frac{1}{C_1^1-C_1^2}R_{13}\rightarrow R_{13} $ &  $-R_{24}+R_{20}\rightarrow R_{20} $  &	$\dis\frac{1}{C_1^1-C_1^2}R_{20}\rightarrow R_{20}$ &	$C_2^1R_{4}+R_{18}\rightarrow R_{18} $				\\ \hline
$C_2^1R_{5}+R_{18}\rightarrow R_{18} $                   & $C_2^1R_{6}+R_{18}\rightarrow R_{18} $ & 	$C_2^1R_{7}+R_{19}\rightarrow R_{19} $	& $C_2^1R_{8}+R_{19}\rightarrow R_{19} $      & 	$C_2^1R_{9}+R_{19}\rightarrow R_{19} $		\\ \hline
$-C_2^1R_{8}+R_{18}\rightarrow R_{18} $ &	$-C_1^2R_{7}+R_{12}\rightarrow R_{12} $                    & $-C_1^2R_{8}+R_{12}\rightarrow R_{12} $                   & $-C_1^2R_{9}+R_{12}\rightarrow R_{12} $                  &  $-R_{2}+R_{1}\rightarrow R_{1} $  						\\ \hline 
$-R_{3}+R_{1}\rightarrow R_{1} $                                  & $-R_{4}+R_{1}\rightarrow R_{1} $                                  & $-R_{5}+R_{1}\rightarrow R_{1} $                                &  $-R_{6}+R_{1}\rightarrow R_{1} $                           & $-R_{7}+R_{1}\rightarrow R_{1}$ \\ \hline 
$-R_{8}+R_{1}\rightarrow R_{1} $                                  & $-R_{9}+R_{1}\rightarrow R_{1} $                                 & $-C_2^1R_{1}+R_{11}\rightarrow R_{11} $                   &  $-C_2^1R_{8}+R_{11}\rightarrow R_{11} $ &	$-R_{11}+R_{10}\rightarrow R_{10} $				\\ \hline 
$2R_{18}+R_{10}\rightarrow R_{10} $                           &  $-R_{19}+R_{10}\rightarrow R_{10} $                         &  $R_{19}+R_{11}\rightarrow R_{11} $ &	$R_{19}+R_{18}\rightarrow R_{18}$  &  $C_2^1R_{13}+R_{15}\rightarrow R_{15} $		\\ \hline 
$-2C_2^1R_{13}+R_{11}\rightarrow R_{11} $              & $-C_1^2R_{13}+R_{17}\rightarrow R_{17} $  &  $-C_2^1R_{13}+R_{18}\rightarrow R_{18} $		&   $C_2^1R_{13}+R_{23}\rightarrow R_{23} $           &	$-R_{23}+R_{15}\rightarrow R_{15} $		\\ \hline 
$-R_{22}+R_{14}\rightarrow R_{14} $ &	$-R_{18}+R_{11}\rightarrow R_{11} $		&  $\dis\frac{1}{C_2^2-C_2^1}R_{14}\rightarrow R_{14} $			& 	 $\dis\frac{1}{C_2^2-C_2^1}R_{15}\rightarrow R_{15} $		&	$-C_1^2R_{14}+R_{17}\rightarrow R_{17} $ 		\\ \hline 
$-C_1^2R_{15}+R_{17}\rightarrow R_{17} $		 &  $C_2^1R_{14}+R_{23}\rightarrow R_{23}$     	        & $C_2^1R_{15}+R_{23}\rightarrow R_{23}$	&   $-C_2^1R_{20}+R_{18}\rightarrow R_{18} $  & 	$C_2^1R_{20}+R_{21}\rightarrow R_{21} $			 \\ \hline
$-C_1^2R_{20}+R_{24}\rightarrow R_{24} $  		& $R_{18}\leftrightarrow R_{19}$ &  $R_{17}\leftrightarrow R_{18}$	    &    $R_{16}\leftrightarrow R_{17} $      &  	$R_{15}\leftrightarrow R_{16}$					 \\  \hline  
$R_{14}\leftrightarrow R_{15}$ &	$R_{13}\leftrightarrow R_{14} $	&	$R_{20}\leftrightarrow R_{21}$		&	$R_{22}\leftrightarrow R_{23}$		&	$R_{21}\leftrightarrow R_{22} $							\\   \hline
&  $R_{23}\leftrightarrow R_{24}$                                   &  $R_{22}\leftrightarrow R_{23}$                  		        & $R_{20}\leftrightarrow R_{21}$                        		 &  \\   \hline
\end{tabular}}
\caption{Row reduction operations on $\mathcal{M}$.}\label{table:2:12}
\end{center}
\end{table}
\noindent Then we see that $\mathcal{M}$ is row equivalent to the matrix $\widetilde{\mathcal{M}}$ below:

\begin{equation*}
\scalebox{0.55}{$
\left[
\begin{array}{cccccccccc|cccc|ccc|cccccccc|ccc}
 0 & 0 & 0 & 0 & 0 & 0 & 1 & 0 & -1 & 0 & 2 & 2 & 1 & 1 & -1 & 1 & 1 & 2 & 1 & 2 & 1 & 2 & -1 & 2 & 1 & 1 & 1 & 1 \\
 0 & 1 & 0 & 0 & 0 & 0 & 0 & 0 & 0 & 0 & -1 & 0 & 0 & 0 & 0 & 0 & 0 & 0 & 0 & 0 & 0 & 0 & 0 & 0 & 0 & 0 & 0 & 0 \\
 0 & 0 & 1 & 0 & 0 & 0 & 0 & 0 & 0 & 0 & 0 & -1 & 0 & 0 & 0 & 0 & 0 & 0 & 0 & 0 & 0 & 0 & 0 & 0 & 0 & 0 & 0 & 0 \\
 0 & 0 & 0 & 1 & 0 & 0 & 0 & 0 & 0 & 0 & 0 & 0 & 0 & 0 & 0 & 0 & 0 & 0 & 0 & -1 & 0 & 0 & 0 & 0 & 0 & 0 & 0 & 0 \\
 0 & 0 & 0 & 0 & 1 & 0 & 0 & 0 & 1 & 0 & 0 & 0 & 0 & 0 & 1 & 0 & 0 & 0 & 0 & 0 & 0 & 0 & 1 & 0 & 0 & 0 & 0 & 0 \\
 0 & 0 & 0 & 0 & 0 & 1 & 0 & 0 & 0 & 0 & 0 & 0 & 0 & 0 & 0 & 0 & 0 & -1 & 0 & 0 & 0 & 0 & 0 & 0 & 0 & 0 & 0 & 0 \\
 0 & 0 & 0 & 0 & 0 & 0 & 0 & 1 & 0 & 0 & 0 & 0 & 0 & 0 & 0 & 0 & 0 & 0 & 0 & 0 & 0 & 0 & 0 & -1 & 0 & 0 & 0 & 0 \\
 0 & 0 & 0 & 0 & 0 & 0 & 0 & 0 & 1 & 0 & 0 & 0 & -1 & 0 & 0 & 0 & 0 & 0 & 0 & 0 & 0 & 0 & 0 & 0 & 0 & 0 & 0 & 0 \\
 0 & 0 & 0 & 0 & 0 & 0 & 0 & 0 & 0 & 1 & 0 & 0 & 0 & 0 & 0 & 0 & 0 & 0 & 0 & 0 & 0 & -1 & 0 & 0 & 0 & 0 & 0 & 0 \\\hline 
 0 & 0 & 0 & 0 & 0 & 0 & 0 & 0 & 0 & 0 & C_2^2-C_2^1 & C_2^1-C_2^2 & C_2^1 & 0 & 0 & 0 & 0 & -C_2^1 & 0 & C_2^2-2 C_2^1 & 0 & 4 C_2^1-2 C_2^2 & 0 & 2 C_2^1 & 0 & 0 & 0 & 0\\ 
 0 & 0 & 0 & 0 & 0 & 0 & 0 & 0 & 0 & 0 & -C_2^1 & C_2^2-2 C_2^1 & 0 & 0 & 0 & 0 & 0 & 0 & C_2^1 & 0 & 0 & C_2^2-2 C_2^1 & C_2^1 & -C_2^1 & 0 & 0 & 0 & 0 \\
 0 & 0 & 0 & 0 & 0 & 0 & 0 & 0 & 0 & 0 & C_1^2 & C_1^2 & C_1^1+C_1^2 & C_1^2 & 0 & 0 & 0 & 0 & 0 & 0 & 0 & C_1^2 & 0 & C_1^2 & 0 & 0 & 0 & 0 \\
 0 & 0 & 0 & 0 & 0 & 0 & 0 & 0 & 0 & 0 & 0 & 0 & -C_2^1 & 0 & 0 & 0 & 0 & 0 & 0 & 0 & 0 & C_2^2-2 C_2^1 & 0 & -C_2^1 & 0 & 0 & 0 & 0 \\\hline 
 0 & 0 & 0 & 0 & 0 & 0 & 0 & 0 & 0 & 0 & 0 & 0 & 0 & 0 & 1 & 0 & 0 & 0 & -1 & 0 & 0 & 0 & 0 & 0 & 0 & 0 & 0 & 0 \\
 0 & 0 & 0 & 0 & 0 & 0 & 0 & 0 & 0 & 0 & 0 & 0 & 0 & 0 & 0 & 1 & 0 & 0 & 0 & 0 & 0 & 0 & 0 & 0 & -1 & 0 & 0 & 0 \\
 0 & 0 & 0 & 0 & 0 & 0 & 0 & 0 & 0 & 0 & 0 & 0 & 0 & 0 & 0 & 0 & 1 & 0 & 0 & 0 & 0 & 0 & 0 & 0 & 0 & -1 & 0 & 0 \\\hline 
 0 & 0 & 0 & 0 & 0 & 0 & 0 & 0 & 0 & 0 & 0 & 0 & 0 & 0 & 0 & 0 & 0 & C_2^2-2 C_2^1 & -C_2^1 & -C_2^1 & 0 & 0 & 0 & 0 & 0 & 0 & 0 & 0 \\
 0 & 0 & 0 & 0 & 0 & 0 & 0 & 0 & 0 & 0 & 0 & 0 & 0 & 0 & 0 & 0 & 0 & C_1^2 & C_1^1+C_1^2 & C_1^2 & C_1^2 & 0 & 0 & 0 & C_1^2 & C_1^2 & 0 & 0 \\
 0 & 0 & 0 & 0 & 0 & 0 & 0 & 0 & 0 & 0 & 0 & 0 & 0 & 0 & 0 & 0 & 0 & -C_2^1 & C_2^1 & C_2^2-2 C_2^1 & 0 & C_2^2-2 C_2^1 & 0 & -C_2^1 & 0 & 0 & C_2^1 & 0 \\
 0 & 0 & 0 & 0 & 0 & 0 & 0 & 0 & 0 & 0 & 0 & 0 & 0 & 0 & 0 & 0 & 0 & 0 & -C_2^1 & 0 & 0 & 0 & 0 & 0 & -C_2^1 & C_2^2-2 C_2^1 & 0 & 0 \\
 0 & 0 & 0 & 0 & 0 & 0 & 0 & 0 & 0 & 0 & 0 & 0 & 0 & 0 & 0 & 0 & 0 & 0 & 0 & 0 & 0 & -C_2^1 & 0 & C_2^2-2 C_2^1 & 0 & 0 & -C_2^1 & 0 \\
 0 & 0 & 0 & 0 & 0 & 0 & 0 & 0 & 0 & 0 & 0 & 0 & 0 & 0 & 0 & 0 & 0 & 0 & 0 & 0 & 0 & C_1^2 & 0 & C_1^2 & C_1^2 & C_1^2 & C_1^1+C_1^2 & C_1^2 \\
 0 & 0 & 0 & 0 & 0 & 0 & 0 & 0 & 0 & 0 & 0 & 0 & 0 & 0 & 0 & 0 & 0 & 0 & 0 & 0 & 0 & 0 & 1 & 0 & 0 & 0 & -1 & 0 \\
 0 & 0 & 0 & 0 & 0 & 0 & 0 & 0 & 0 & 0 & 0 & 0 & 0 & 0 & 0 & 0 & 0 & 0 & 0 & 0 & 0 & 0 & 0 & 0 & C_2^2-2 C_2^1 & -C_2^1 & -C_2^1 & 0 \\\hline 
 1 & 1 & 1 & 1 & 1 & 1 & 1 & 1 & 1 & 1 & 1 & 1 & 1 & 1 & 1 & 1 & 1 & 1 & 1 & 1 & 1 & 1 & 1 & 1 & 1 & 1 & 1 & 1 \\
\end{array}\right].$}
\end{equation*}
Note that in the presentation $\mathcal{\widetilde{M}}$ is  partitioned. Let $\mathcal{\widetilde{M}}_{2,2}$ and $\mathcal{\widetilde{M}}_{4,4}$  denote the $(2,2)$ and $(4,4)$ partitions, respectively, of $\mathcal{\widetilde{M}}$ counting from left-to-right and top-to-bottom.
The matrix $\widetilde{\mathcal{M}}$ has full rank if and only if $det(\widetilde{\mathcal{M}}_{2,2})\neq 0$ and $det(\widetilde{\mathcal{M}}_{4,4})\neq 0$. We have 
\begin{eqnarray*}
det(\widetilde{\mathcal{M}}_{2,2}) & = & C_1^2 C_2^1\left(C_2^1 - C_2^2\right)\left(3 C_2^1 - C_2^2\right),\\
det(\widetilde{\mathcal{M}}_{4,4}) & = & \left(C_1^2\right)^2 C_2^1 \left(C_2^1 - C_2^2\right)^2 \left(2 C_2^1 - C_2^2\right) \left(3 C_2^1 - C_2^2\right).
\end{eqnarray*}
We know that $C_1^2\neq 0$, $C_2^1\neq 0$ and $C_2^1 - C_2^2\neq 0$. So $\widetilde{\mathcal{M}}$ has full rank if and only if $3 C_2^1 - C_2^2\neq 0$ and $2C_2^1 - C_2^2\neq 0$, where $\Sigma_I^4(\xs)=x_{11}^*+x_{12}^*+x_{13}^*=x_1^*+2x_2^*$, 
\begin{equation*}
3 C_2^1 - C_2^2  =  \frac{\sigma(\Sigma_I^4(\xs))}{(x_2^*)^2}-\frac{2\sigma(x_2^*)}{(\Sigma_I^4(\xs))^2}=\frac{\Sigma_I^4(\xs)(1-\Sigma_I^4(\xs))-2x_2^*(1-x_2^*)}{(x_2^*)^2(\Sigma_I^4(\xs))^2},
\end{equation*}
\begin{equation*}
2 C_2^1 - C_2^2  =  \frac{\sigma(\Sigma_I^4(\xs))}{(x_2^*)^2}-\frac{\sigma(x_2^*)}{(\Sigma_I^4(\xs))^2}=\frac{\Sigma_I^4(\xs)(1-\Sigma_I^4(\xs))-x_2^*(1-x_2^*)}{(x_2^*)^2(\Sigma_I^4(\xs))^2}.
\end{equation*}
Assume on the contrary that $3 C_2^1 - C_2^2=0$. We simplify the previous equality and get
\begin{equation}\label{eqn:3:12}
(x_1^*+2x_2^*)(1-x_1^*-2x_2^*)-2x_2^*(1-x_2^*)=0\ \ \tnr{or}\ \ x_2^*=-x_1^*+\sqrt{\frac{x_1^*+(x_1^*)^2}{2}}
\end{equation}
as $x_2^*>0$.  Since $\xs\in\Delta^{27}$, we have $8(x_1^*+2x_2^*)+4x_7^*=1$. This implies $0<x_1^*<\Sigma_I^4(\xs)=x_1^*+2x_2^*<\tfrac{1}{8}$. By (\ref{eqn:3:12}), we have $x_2^*<x_1^*$ if and only if $x_1^*>\tfrac{1}{7}$. Using the equality $f_2(\xs)=f_3(\xs)$ and the formulas of $f_1(\xs)$, $f_2(\xs)$ and $f_3(\xs)$ in (\ref{3:1}), (\ref{3:2}) and (\ref{3:3}), we find that $\sigma(x_2^*)=3\sigma(\Sigma_I^4(\xs))\sigma(x_1^*)$, where $\sigma(\Sigma_I^4(\xs))>1$. So we deduce that $x_2^*<x_1^*$. 
This is a contradiction. 

Next assume that $2 C_2^1 - C_2^2=0$. Then we get  $(x_1^*+2x_2^*)(1-x_1^*-2x_2^*)-x_2^*(1-x_2^*)=0$. This gives 
\begin{equation*}\label{eqn:3:13}
x_2^*=\frac{1-x_1^*}{3}\quad\tnr{or}\quad x_2^*=-x_1^*.
\end{equation*}
Since $x_2^*>0$, we obtain $x_1^*+3x_2^*=1$ or $7x_1^*+13x_2^*+4x_7^*=0$, a contradiction. This shows that $M$ has full rank.

By  \fullref{flipside}, there exists a direction $\vec v_3\in T_{\xs}\Delta^{27}$ such that values of $f_l$ for $l\in I-\{7,14,21,28\}$ decrease along a line segment in the direction of $\vec v_3$. Values of $f_l$ for $l\in\{7,14,21,28\}$ are smaller than $\alpha_*$ for a short distance along $\vec v_3$. As a result there exists a point $\tb{w}\in\Delta^{27}$ such that  $f_l(\tb{w})<\alpha_*$ for every $l\in I=\{1,2,\dots,28\}$, a contradiction. Therefore we obtain that $f_7(\xs)=\alpha_*$. This concludes the proof.
\end{proof}

\fullref{prop:3:1} and \fullref{prop:3:2} establish the properties of $F$ given in (\ref{b}) in the introduction. Once these properties are verified, the computation of $\alpha_*$, and consequently the infimums of the maximum of the displacement functions in $\mathcal{F}$ and $\mathcal{G}$ on $\Delta^{27}$, is straightforward. In other words, we have the statements below:

\begin{theorem}\label{thm:3:1}
Let $F\co\Delta^{27}\to\mathbb{R}$ be defined by $\tb{x}\to\max\{f(\tb{x})\co f\in\mathcal{F}\}$, where $\mathcal{F}$ is the set of functions listed in \fullref{dispfunc}. Then $\inf_{\tb{x}\in\Delta^{27}}F(\tb{x})=\alpha_*=24.8692...$ the unique real root of the polynomial
$21 x^4 - 496 x^3 - 654 x^2 + 24 x + 81$ greater than $9$.
\end{theorem}
\begin{proof}
Since $\xs\in\Delta^{27}$, we have $8x_1^*+8x_2^*+8x_3^*+4x_7^*=1$ by \fullref{lem:3:3}. We plug $x_1^*+x_2^*+x_3^*=\tfrac{1}{8}-x_7^*/2$ into $f_7(\xs)=\alpha_*$ in (\ref{3:4}). Then we find $x_7^*=1/(1+3\alpha_*)$. Using $x_7^*$, we obtain from $f_1(\xs)=\alpha_*$ in (\ref{3:1}) that $x_1^*=3/(3+\alpha_*)$. Because we have $f_2(\xs)=f_3(\xs)$ by \fullref{prop:3:2}, using the formulas in (\ref{3:2}) and (\ref{3:3}) we find 
\begin{equation*}
x_2^*=x_3^*=\frac{3 (\alpha_* -1)}{21 \alpha_*^2+14 \alpha_* -3}.
\end{equation*}
When we plug all these values into the equation $2x_1^*+2x_2^*+2x_3^*+x_7^*=\tfrac{1}{4}$ we see that $\alpha_*$ satisfies the equation $21 x^4 - 496 x^3 - 654 x^2 + 24 x + 81=0$ which has the roots
$$\alpha_1=-1.1835...,\quad\alpha_2=-0.3968...,\quad \alpha_3=0.3302...,\quad \alpha_4=24.8692....$$  
The conclusion of the theorem follows from \fullref{lemtwo}.
\end{proof}

\begin{theorem}\label{thm:3:2}
Let $G\co\Delta^{27}\to\mathbb{R}$ be defined by $\tb{x}\to\max\{f(\tb{x})\co f\in\mathcal{G}\}$, where $\mathcal{G}$ is the set of functions listed in \fullref{dispfunc}. Then $\inf_{\tb{x}\in\Delta^{27}}G(\tb{x})=24.8692...$.
\end{theorem}
\begin{proof}
Since $\mathcal{F}\subset\mathcal{G}$, we have $G(\tb{x})\geq F(\tb{x})$ for every $\tb{x}\in\Delta^{27}$. Note that we obtain the coordinates of $\tb{x}^*$ as
$$x_1^*= 0.1076...,\quad x_2^*=x_3^*= 0.0053...,\quad x_7^*= 0.0132...$$ by \fullref{thm:3:1}. Then for the indices $l\in\{3,4,10,11,17,18,24,25\}$ we find that 
$g_l(\xs)=2.4822...$. For the indices $l\in\{1,5,9,13,15,17,19,23,27\}$ we have ${g_l(\xs)= 1.1131...}$. Similarly we compute that  $h_l(\xs)=u_l(\xs)=0.4028...$ for $l\in\{7,14,21,28\}$ and  $h_l(\xs)=0.1111...$ for $l\in\{1,5,9,13,15,19,23,27\}$.  Because $G(\xs)=F(\xs)$, we are done.  
\end{proof}


\section{Proof of the Main Theorem}\label{sec4}

To prove the main theorem of this paper we shall require two preliminary statements. The first is the following:

\begin{lemma}\label{duetoref}
Let $\xi$ and $\eta$ be two non-commuting loxodromic isometries of $\hyp$. If $z_2$ is the mid-point of the shortest geodesic segment connecting the axes of $\xi$ and $\eta^{-1}\xi\eta$, then $d_{\xi}z_2<d_{\eta\xi\eta^{-1}}z_2$. 
\end{lemma}
\begin{proof}
Let us denote the $\lambda$-displacement cylinder for a loxodromic isometry $\gamma$ by $Z_{\lambda}(\gamma)$. Let $\lambda=d_{\xi}z_2$. The point $z_2\in Z_{\lambda}(\xi)$ is the only point in the set $Z_{\lambda}(\xi)\cap Z_{\lambda}(\eta^{-1}\xi\eta)$. Because $\eta\cdot z_2\neq z_2$ and $\eta\cdot z_2$ is the only element in $Z_{\lambda}(\eta\xi\eta^{-1})\cap Z_{\lambda}(\xi)$, the point $z_2$ cannot be in $Z_{\lambda}(\eta\xi\eta^{-1})$. Hence the conclusion follows.
\end{proof}

The second statement below is proved using arguments analogous to the ones introduced in \cite[Theorem 9.1]{CSParadox}, \cite[Theorem 5.1]{Y1} and \cite[Theorem 4.1]{Y2}. Therefore we shall not provide a detailed proof. 
\begin{theorem}\label{thm:4:1}
Let $\xi$ and $\eta$ be two non--commuting isometries of $\hyp$. Suppose that $\Gamma=\langle\xi,\eta\rangle$ is a purely loxodromic free Kleinian group. Let $\Phi_1  =  \{\xi,\eta,\eta^{-1},\xi^{-1}\}$ and 
$$\Gamma_{\jrg}=\Phi_1\cup\{\xi\eta\xi^{-1},\xi^{-1}\eta\xi,\eta\xi\eta^{-1},\eta^{-1}\xi\eta,\xi\eta^{-1}\xi^{-1},\xi^{-1}\eta^{-1}\xi,\eta\xi^{-1}\eta^{-1},\eta^{-1}\xi^{-1}\eta\}.$$ Then we have $\max\nolimits_{\gamma\in\Gamma_{\jrg}}\left\{d_{\gamma}z\right\}\geq 1.6068...$ for any $z\in\hyp$. 
\end{theorem} 
\begin{proof}
Assume that $\Gamma=\langle\xi,\eta\rangle$ is geometrically infinite. The conclusion of the theorem follows from \fullref{dispfunc}, \fullref{thm:3:2} and the following inequality
\begin{equation*}\label{inequality2}
\max_{\gamma\in\Gamma_{\jrg}}\left\{d_{\gamma}z\right\} \geq  \tfrac{1}{2}\log G(\tb{m})
                         \geq  \tfrac{1}{2}\log\left(\inf_{\tb{x}\in\Delta^{27}} G(\tb{x})\right)=\tfrac{1}{2}\log 24.8692...=1.6068...,
\end{equation*}
where $\tb{m}=(\nu_{\xi\eta^{-1}\xi^{-1}}(S_{\infty}),\dots,\nu_{\xi^{-2}}(S_{\infty}))\in\Delta^{27}$.

Assume that $\Gamma=\langle\xi,\eta\rangle$ is geometrically finite.  Because $\Gamma=\langle\xi,\eta\rangle$ is torsion-free, each isometry $\gamma\in\Gamma_{\jrg}$ has infinite order. This implies that $\gamma\cdot z\neq z$ for every $z\in\hyp$. 
Since 
$\tr{dist}(z,\gamma_1\gamma_2\cdot z)=\tr{dist}(\gamma_1^{-1}\cdot z,\gamma_2\cdot z)$ and $\tr{dist}(z,\gamma_1\cdot z)=\tr{dist}(z,\gamma_1^{-1}\cdot z)$ for all $\gamma_1, \gamma_2\in\Gamma=\langle\xi,\eta\rangle$, we have 
\begin{equation*}
\begin{array}{cccc}
\textnormal{dist}(z,\xi\eta\xi^{-1}\cdot z)=\textnormal{dist}(\xi^{-1}\cdot z,\eta\xi^{-1}\cdot z)=\textnormal{dist}(\xi^{-1}\cdot z,\eta^{-1}\xi^{-1}\cdot z)=\textnormal{dist}(z,\xi\eta^{-1}\xi^{-1}\cdot z),\\
\textnormal{dist}(z,\xi^{-1}\eta\xi\cdot z)=\textnormal{dist}(\xi\cdot z,\eta\xi\cdot z)=\textnormal{dist}(\xi\cdot z,\eta^{-1}\xi\cdot z)=\textnormal{dist}(z,\xi^{-1}\eta^{-1}\xi\cdot z),\\
\textnormal{dist}(z,\eta\xi\eta^{-1}\cdot z)=\textnormal{dist}(\eta^{-1}\cdot z,\xi\eta^{-1}\cdot z)=\textnormal{dist}(\eta^{-1}\cdot z,\xi^{-1}\eta^{-1}\cdot z)=\textnormal{dist}(z,\eta\xi^{-1}\eta^{-1}\cdot z),\\
\textnormal{dist}(z,\eta^{-1}\xi\eta\cdot z)=\textnormal{dist}(\eta\cdot z,\xi\eta\cdot z)=\textnormal{dist}(\eta\cdot z,\xi^{-1}\eta\cdot z)=\textnormal{dist}(z,\eta^{-1}\xi^{-1}\eta\cdot z).
\end{array}
\end{equation*}
Therefore, all of the hyperbolic displacements under the isometries in $\Gamma_{\jrg}$ are realised by the geodesic line segments 
joining the points $\{z\}\cup\{\gamma\cdot z\co\gamma\in\Phi\}$, where  $\Phi=\{\xi,\eta^{-1},\eta,\xi^{-1}\}\cup\{\xi\eta^{-1},\xi\eta,\eta\xi,\eta\xi^{-1}\}.$
We enumerate the elements of $\Phi$ for some index set $I'\subset\mathbb{N}$ such that $P_0=z$ and $P_i=\gamma_i\cdot z$ for $i\in I'$ and $\gamma_i\in\Phi$.  Let $\Delta_{ij}=\triangle{P_iP_0P_j}$ represent the geodesic triangle with vertices $P_i$, $P_0$ and $P_j$ for $i,j\in I'$ and $i\neq j$. 

Let $\mathfrak{X}$ denote the character variety $PSL(2,\C)\times PSL(2,\C)\simeq\tnr{Isom}^+(\hyp)\times\tnr{Isom}^+(\hyp)$ and  $\mathfrak{GF}$ be the set $\{(\gamma,\beta)\in\mathfrak{X}\co\langle\gamma,\beta\rangle\textnormal{ is free, geometrically finite and without any parabolic}\}.$
For a fixed $z\in\hyp$ let us define the real-valued function $f_{z}\co\mathfrak{X}\to\R$ with the formula
\begin{displaymath}
f_{z}(\xi,\eta)=\max_{\psi\in\Gamma_{\jrg}}\{\dpsi\}.
\end{displaymath}
The function $f_{z}$ is continuous and proper. Therefore, it takes a minimum value at some point $(\xi_0,\eta_0)$ in $\overline{\mathfrak{GF}}$. 
The value $f_{z}(\xi_0,\eta_0)$ is the unique longest side length of one geodesic triangle $\Delta_{ij}$ for some $i,j\in I'$. Let us denote this geodesic triangle with $\Delta$ and their vertices by $\widetilde{P}_i$, $P_0$ and $\widetilde{P}_j$. There are two cases to consider: (1) $\Delta$ is acute or (2) $\Delta$ is not acute. 

Assume that (2) is the case. Then there is a one-step process analogous to the ones described in the 
\begin{figure}[H]
\begin{center}
\includegraphics[scale=.53]{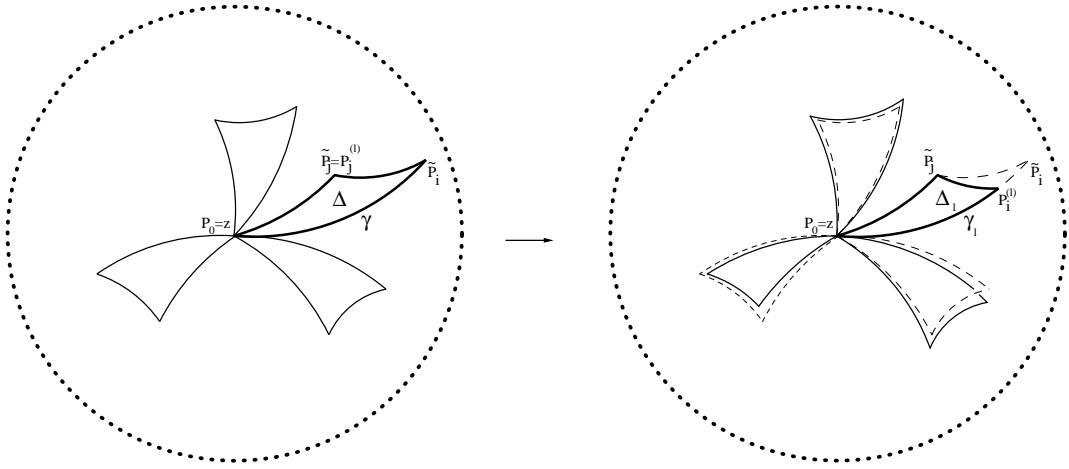}
\end{center}
\caption{Case (2): $\Delta$ is not acute.}
\end{figure}
\noindent proofs of \cite[Theorem 5.1]{Y1} and \cite[Theorem 4.1]{Y2}. This one-step process is illustrated in Figure 1 proving that $(\xi_0,\eta_0)\in\overline{\mathfrak{GF}}-\mathfrak{GF}$. 
\begin{figure}[H]
\begin{center}
\includegraphics[scale=.53]{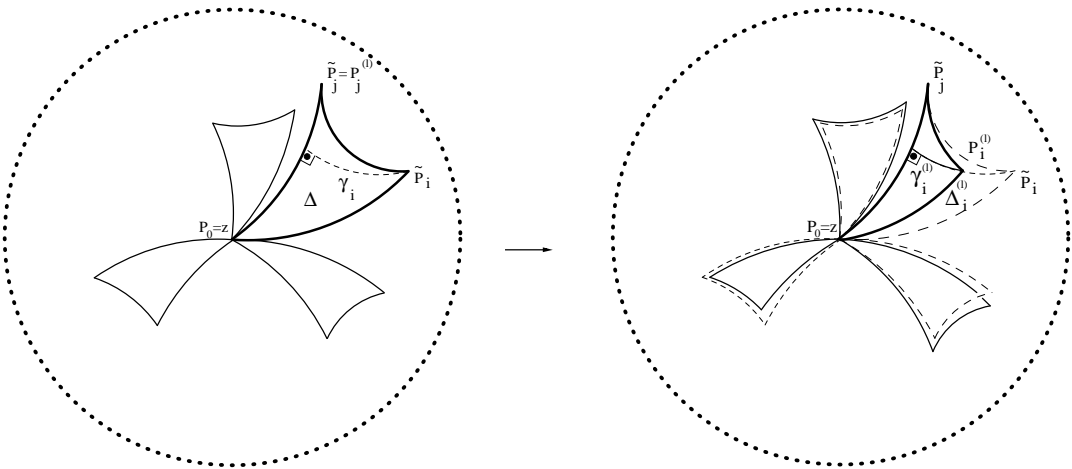}
\end{center}
\caption{Case (1): $\Delta$ is acute.}
\end{figure}
If (1) is the case, then there is a two-step process analogous to the ones described in the proofs of \cite[Theorem 5.1]{Y1} and \cite[Theorem 4.1]{Y2}. This two-step process is illustrated in Figures 2 and 3 proving again that $(\xi_0,\eta_0)\in\overline{\mathfrak{GF}}-\mathfrak{GF}$. 
\begin{figure}[H]
\begin{center}
\includegraphics[scale=.53]{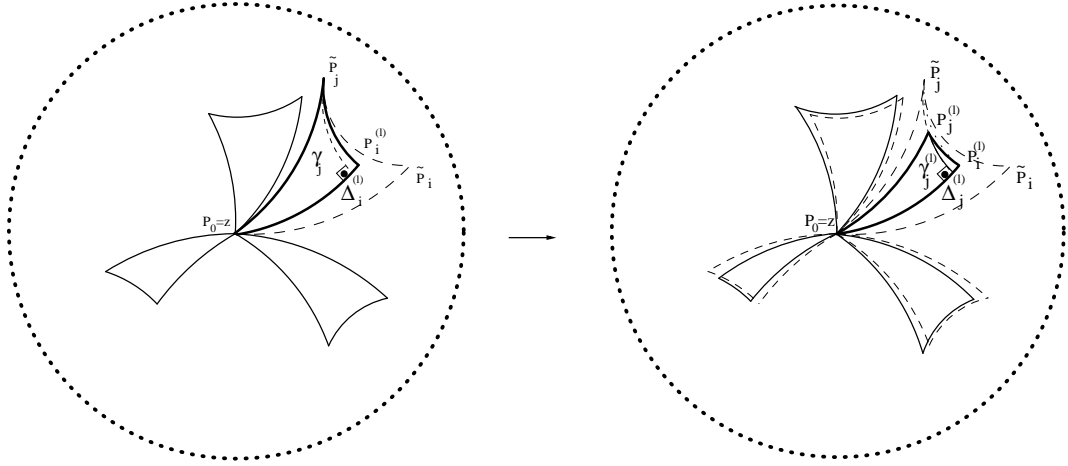}
\end{center}
\caption{Case (1): $\Delta$ is acute.}
\end{figure}

Since geometrically finite case reduces to geometrically infinite case by the facts that the set of $(\xi,\eta)$ such that $\langle\xi,\eta\rangle$ is free, geometrically infinite and without any parabolic is dense in $\overline{\mathfrak{GF}}-\mathfrak{GF}$ and every $(\xi,\eta)\in\mathfrak{X}$ with  $\langle\xi,\eta\rangle$ is free and without any parabolic is in $\overline{\mathfrak{GF}}$, the conclusion of the theorem follows when $\Gamma=\langle\xi,\eta\rangle$ is geometrically finite as well.  For the details of this crucial final step in the proof, readers may refer to \cite[Propositions 8.2 and 9.3]{CSParadox}, \cite[Main Theorem]{CSH} and \cite{CCHS}.
\end{proof}

Using \fullref{duetoref} and \fullref{thm:4:2} we can prove the following statement, the main result of this paper;

\begin{theorem}\label{thm:4:2}
Let $\xi$ and $\eta$ be two non--commuting isometries of $\hyp$. 
Suppose that $\Gamma=\langle\xi,\eta\rangle$ is a purely loxodromic free Kleinian group.
If $d_{\gamma}z_2<1.6068...$ for every $\gamma\in\Phi_2=\{\eta,\xi^{-1}\eta\xi,\xi\eta\xi^{-1}\}$ and $d_{\eta\xi\eta^{-1}}z_2\leq d_{\eta\xi\eta^{-1}}z_1$ for the mid--points $z_1$ and $z_2$ of the shortest geodesic segments joining the axes of $\xi$, $\eta\xi\eta^{-1}$ and $\eta^{-1}\xi\eta$, respectively, then we have
$
|\textnormal{trace}^2(\xi)-4|+|\textnormal{trace}(\xi\eta\xi^{-1}\eta^{-1})-2|\geq1.5937....
$
\end{theorem}
\begin{proof}
We shall mostly follow  the computations given in the proof of Theorem 5.4.5 in Beardon \cite[Section 5.4]{B}. Readers who are interested in further details should refer to this source. 

Considering conjugate elements, for $u=|u|e^{i\theta}$ and $ad-bc=1$ we can assume that
\begin{eqnarray*}
\xi=\left(\begin{array}{cc}
u & 0\\
0 & 1/u
\end{array}\right) & \tnr{and} & \eta=\left(\begin{array}{cc}
a & b\\
c & d
\end{array}\right).
\end{eqnarray*}
Let $\mathcal{A}$  and $T_{\xi}$ denote the axis and translation length of $\xi$, respectively.  Above $\theta$ denotes the angle of rotation of $\xi$ about its axis. Then we have $$|\textnormal{trace}^2(\xi)-4|+|\textnormal{trace}(\xi\eta\xi^{-1}\eta^{-1})-2|=|u-1/u|^2(1+|bc|),$$
where $\sinh^2(\tfrac{1}{2}T_{\xi})+\sin^2\theta=\tfrac{1}{4}|u-1/u|^2$ (see \cite[Equations (5.4.8) and (5.4.10)]{B}).
First we shall determine a lower bound for the term $1+|bc|$. 

By construction $\mathcal{A}$ is the geodesic with end--points $0$ and $\infty$ and $\mathcal{B}=\eta\mathcal{A}$ is the geodesic with end--points $\eta0$ and $\eta\infty$. Since $\Gamma=\langle\xi,\eta\rangle$  is non-elementary, $\mathcal{A}$ and $\mathcal{B}$ don't have a common end--point. This implies that $bc\neq 0$. So the equation 
\begin{equation}\label{Eqn:4:1}
bc=\frac{(1-w)^2}{4w}
\end{equation} 
obtained by the cross--ratios $[1,-1,w,-w]=[0,\infty,b/d,a/c]$ has two solutions. Let $w=\tnr{exp}\ {2(x_0+iy_0)}$ be one of the solutions. We may assume that $|w|\geq 1$. 

Plugging $w=\tnr{exp}\ {2(x_0+iy_0)}$ in (\ref{Eqn:4:1}) we obtain $bc=\sinh^2(x_0+iy_0)$. Then we derive  
\begin{equation*}
\begin{multlined} 
4|bc|^2=|\cosh 2(x_0+iy_0)-1|^2=(\cosh 2x_0 -\cos 2y_0)^2 \\
\shoveleft[5cm]{\geq (\cosh 2x_0 -1)^2 =(\cosh^2x_0+\sinh^2x_0-1)^2\geq(\cosh^2x_0-1)^2,}
\end{multlined}
\end{equation*}
which gives that $2|bc|\geq\cosh^2x_0-1=\sinh^2x_0$. This implies the following inequality
\begin{equation}\label{eqn:4}
1+|bc| \geq\tfrac{1}{2} \sinh^2x_0+1
          = \tfrac{1}{2} \cosh^2x_0+\tfrac{1}{2}
                 \geq  \tfrac{1}{2} \cosh^2x_0.
\end{equation}
Let $d_z\mathcal{A}$ denote the shortest distance between $z$ and $\mathcal{A}$.  Since $\xi$ and $\eta\xi\eta^{-1}$ have the same trace squared, the same translation length and consequently the same value of $\sin^2\theta$, for every $z\in\hyp$ we obtain
\begin{equation}\label{eqn:1}
\sinh^2\tfrac{1}{2}d_{\xi}z  =  \sinh^2(\tfrac{1}{2}T_{\xi})\cosh^2d_z\mathcal{A}+\sin^2\theta\sinh^2d_z\mathcal{A}\\
                                                \leq \left( \sinh^2(\tfrac{1}{2}T_{\xi})+\sin^2\theta\right)\cosh^2d_z\mathcal{A},
\end{equation}
\begin{equation}\label{eqn:2}
\sinh^2\tfrac{1}{2}d_{\eta\xi\eta^{-1}}z  =  \sinh^2(\tfrac{1}{2}T_{\xi})\cosh^2d_z\mathcal{B}+\sin^2\theta\sinh^2d_z\mathcal{B}\\
                                                \leq \left( \sinh^2(\tfrac{1}{2}T_{\xi})+\sin^2\theta\right)\cosh^2d_z\mathcal{B}.
\end{equation}
Then by using the inequalities in (\ref{eqn:1}) and (\ref{eqn:2}) and the fact that $\sinh^2x$ and $\cosh^2x$ are increasing for $x>0$, for every $z\in\hyp$ we derive that 
\begin{equation}\label{eqn:3}
\sinh^2\tfrac{1}{2}\max\{d_{\xi}z,d_{\eta\xi\eta^{-1}}z\}\leq\tfrac{1}{4}\big|u-1/u\big|^2\cosh^2\max\{d_z\mathcal{A},d_z\mathcal{B}\}.
\end{equation}
At this point consider the M\"{o}bius transformation $\psi$ taking $0$, $\infty$, $\beta0$, $\beta\infty$ to  $1$, $-1$, $w$, $-w$. Then we have
\begin{equation*}\label{Eqn:4:2}
d_{\mathcal{A}}\mathcal{B}=d_{\psi\mathcal{A}}\psi\mathcal{B}=\log|w|=2x_0,
\end{equation*}
where $d_{\mathcal{A}}\mathcal{B}$ denotes the shortest distance between $\mathcal{A}$ and $\mathcal{B}$.  
Since we have $d_{z_1}\mathcal{A}=d_{z_1}\mathcal{B}=x_0$ and $d_{\xi}z_1=d_{\eta\xi\eta^{-1}}z_1$, by the inequalities in (\ref{eqn:4}) and (\ref{eqn:3}) we derive that
\begin{equation}\label{trace_ineq}
\sinh^2\tfrac{1}{2}d_{\xi}z_1  \leq  \tfrac{1}{4}\big|u-1/u\big|^2\cosh^2d_{z_0}\mathcal{A}
 \leq   \tfrac{1}{2}\big|u-1/u\big|^2(1+|bc|).
\end{equation}
Now assume on the contrary that $|\textnormal{trace}^2(\xi)-4|+|\textnormal{trace}(\xi\eta\xi^{-1}\eta^{-1})-2|<1.5937...$. 
Because we have $
d_{\eta\xi\eta^{-1}}z_2\leq d_{\eta\xi\eta^{-1}}z_1=d_{\xi}z_1$ and $d_{\gamma}z_2<1.6068...$ for every $\gamma\in\{\eta,\xi\eta\xi^{-1},\xi^{-1}\eta\xi\}$ by the hypothesis, we get $d_{\gamma}z_2<1.6068...$ for every $\gamma\in\Gamma_{\o}$ by the inequality in (\ref{trace_ineq}) and \fullref{duetoref}. This contradicts with \fullref{thm:4:1}. 
\end{proof}

Notice that all of the computations given in this paper to prove \fullref{thm:4:1} and \fullref{thm:4:2} can be repeated also for a finitely generated purely loxodromic free Kleinian group $\Gamma=\langle\xi_1,\xi_2,\dots,\xi_n\rangle$ satisfying a hypothesis similar to the one in \fullref{thm:4:2}. An analog of the decomposition $\Gamma_{\mathcal{D}}$ defined in (\ref{dJ}) is required. For a fixed $n>2$, let 
$$\Psi^n=\{\xi_i^2,\xi_i^{-2}\co i=1,\dots, n\}\cup\{\xi_i\xi_j\xi_k^{-1}\co i\neq j,\ j\neq k,\ i,j,k=1,\dots,n\}$$ 
and $\Phi_1^n=\Psi_r^n=\Xi\cup\Xi^{-1}$, where $\Xi=\{\xi_i\co i=1,\dots,n\}$ and $\Xi^{-1}=\{\xi_i^{-1}\co i=1,\dots,n\}$. When the group $\Gamma=\langle\xi_1,\xi_2,\dots,\xi_n\rangle$ is geometrically infinite, the following is the relevant decomposition:
\begin{equation}\label{dJFG}
\Gamma=\{1\}\cup\Psi_r^n\cup\bigcup_{\psi\in\Psi^n}J_{\psi}.
\end{equation}
Let us name this decomposition $\Gamma_{\mathcal{D}^n}$. The rest follows again from the Culler-Shalen machinery introduced in \cite{CSParadox} and the solution method for the optimisation problems described in this text, \cite{Y1} and \cite{Y2}. Consider the subset of  isometries
\begin{equation}\label{conjG}
\Gamma_{\jrg}^n=\Phi_1^n\cup\{\xi_i\xi_j\xi_i^{-1}\co i\neq j,\ \ i,j=1,2,\dots,n\}
\end{equation}
of $\Psi_r^n\cup\Psi^n$. We first prove an analog of \fullref{thm:2:1} for $\Gamma_{\mathcal{D}^n}$. We list all of the group--theoretical relations as in \fullref{lem:2:1} for the isometries in $\Gamma_{\jrg}^n$. By \fullref{lem1.2} and the group-theoretical relations, we state analog of \fullref{dispfunc} to list all of the displacement functions $\mathcal{G}^n=\{f_l\}$ for the indices $l=1,2,\dots,2 n (8 n^2-10 n+3)$ for the isometries in $\Gamma_{\jrg}^n$. 

These displacement functions satisfy generalised versions of the properties in (\ref{a}) and (\ref{b}) for the decomposition $\Gamma_{\mathcal{D}^n}$. In other words we can prove statements similar to \fullref{prop:3:1} and \fullref{prop:3:2}. With a suitable enumeration of the isometries in $\Gamma_{\jrg}^n$ as in (\ref{list:1:4}), an analog of \fullref{prop:3:1} for $\Gamma_{\mathcal{D}^n}$ implies that it is enough to compare the values of four functions  
\begin{eqnarray*}
\begin{array}{c}
\dis\frac{1-2(n-1)(x_1^*+(n-1)x_2^*+(n-1)x_3^*)-x^*}{2(n-1)(x_1^*+(n-1)x_2^*+(n-1)x_3^*)+x^*}\cdot\frac{1-x_1^*}{x_1^*}  = \alpha_*,\label{4:1}\\
\dis\frac{1-(4n^2-4n-1)(x_1^*+(n-1)x_2^*+(n-1)x_3^*)-2nx^*}{(4n^2-4n-1)(x_1^*+(n-1)x_2^*+(n-1)x_3^*)+2nx^*}\cdot\frac{1-x_2^*}{x_2^*} \leq \alpha_*, \label{4:2}\\
\dis\frac{1-(4n^2-4n-1)(x_1^*+(n-1)x_2^*+(n-1)x_3^*)-2nx^*}{(4n^2-4n-1)(x_1^*+(n-1)x_2^*+(n-1)x_3^*)+2nx^*}\cdot\frac{1-x_3^*}{x_3^*} \leq \alpha_*, \label{4:3}\\
\dis\frac{1-(2n-1)(2(n-1)(x_1^*+(n-1)x_2^*+(n-1)x_3^*)+x_{2(n-1)(2n-1)+1}^*)}{(2n-1)(2(n-1)(x_1^*+(n-1)x_2^*+(n-1)x_3^*)+x_{2(n-1)(2n-1)+1}^*)}\cdot\frac{1-x_{2(n-1)(2n-1)+1}^*}{x_{2(n-1)(2n-1)+1}^*} \leq \alpha_*,\label{4:4}
\end{array}
\end{eqnarray*}
where $\alpha_*$ is the infimum of the maximum of the displacement functions in $\mathcal{G}^n$ on the simplex $\Delta^{(2n-1)^3}$. 
Using analog of \fullref{prop:3:2} for $\Gamma_{\mathcal{D}^n}$ and 
the computations given in \fullref{thm:3:1} and \fullref{thm:3:2}, we can prove the following generalisation of \fullref{thm:4:1}:
\begin{conjecture}\label{conj:4:1}
Let $\Xi=\{\xi_1,\xi_2,\dots,\xi_n\}$ for $n>2$ be a set of non-commuting isometries of $\hyp$ and $\Xi^{-1}=\{\xi_1^{-1},\xi_2^{-1},\dots,\xi_n^{-1}\}$. Suppose that $\Gamma=\langle\xi_1,\xi_2,\dots,\xi_n\rangle$ is a purely loxodromic free Kleinian group. Let $\Phi_1^n=\Xi\cup\Xi^{-1}$ and $\Gamma_{\jrg}^n$ be as in (\ref{conjG}).
Then we have $$\max_{\gamma\in\Gamma_{\jrg}^n}d_{\gamma}z\geq\tfrac{1}{2}\log\alpha_n$$ for every $z\in\hyp$. Above $\alpha_n$ is the only real root of the polynomial $p_n(x)$ greater than $(2n-1)^2$, where
\begin{equation*}
\begin{multlined}
p_n(x)=(8n^3-12n^2+2n+1)\ x^4+\\
\shoveleft[2cm]{(-64n^6+192n^5-192n^4+64n^3+4n^2+2n-4)\ x^3}\ +\\
\shoveleft[3.5cm]{(-96n^5+224n^4-168n^3+52n^2-18n+6)\ x^2}\ +\\
\shoveleft[5cm]{(32n^5-112n^4+128n^3-68n^2+22n-4)\ x}\ +\\
\shoveleft[8.8cm]{16 n^4-32 n^3+24 n^2-8 n+1}.
\end{multlined}
\end{equation*}
\end{conjecture}
\noindent The proof of \fullref{conj:4:1} goes along the same lines as the proof of \fullref{thm:4:1} when $\Gamma=\langle\xi_1,\xi_2\dots,\xi_n\rangle$ is geometrically finite. 

This conjecture and arguments analogous to the ones presented in the proof of \fullref{thm:4:2} imply the following generalisation of \fullref{thm:4:2}:
\begin{conjecture}\label{conj:4:2}
Let $\Gamma=\langle\xi_1,\xi_2,\dots,\xi_n\rangle$ and $\alpha_n$ be as described in \fullref{conj:4:1}. 
Assume that there exists an isometry $\xi_i$ for $i\neq 1$ so that  $d_{\xi_i\xi_1\xi_i^{-1}}z_2\leq d_{\xi_i\xi_1\xi_i^{-1}}z_1$ and $d_{\gamma}z_2<\tfrac{1}{2}\log\alpha_n$ for every isometry $\gamma\in\Phi_n=\Gamma_{\o}^n-\{\xi_1,\xi_1^{-1},\xi_i^{-1}\xi_1\xi_i,\xi_i^{-1}\xi_1^{-1}\xi_i,\xi_i\xi_1\xi_i^{-1},\xi_i\xi_1^{-1}\xi_i^{-1}\}$, where $z_1$ and $z_2$ are the mid-points of the shortest geodesic segments connecting the axes of $\xi_1$, $\xi_i\xi_1\xi_i^{-1}$ and $\xi_i^{-1}\xi_1\xi_i$, respectively.
Then  we have
\begin{equation*}\label{GJ}
|\textnormal{trace}^2(\xi_1)-4|+|\textnormal{trace}(\xi_1\xi_i\xi_1^{-1}\xi_i^{-1})-2|\geq 2\sinh^2\left(\tfrac{1}{4}\log\alpha_n\right).
\end{equation*}
\end{conjecture}
\noindent The details of the outlines of the proofs of \fullref{conj:4:1} and \fullref{conj:4:2} given above will be left to future studies.

%
%
%
%

\end{document}